\documentclass[10.5pt,a4paper]{article}
\usepackage{amsfonts,amssymb}\usepackage{bbm}
\usepackage{graphicx,latexsym,euscript,makeidx,color,bm}
\usepackage{amsmath,amsfonts,amssymb,amsthm,thmtools,mathrsfs,enumerate}
\usepackage[colorlinks,linkcolor=blue,anchorcolor=green,citecolor=red]{hyperref}
\usepackage{graphicx,tikz,exscale,pagecolor}  

\usepackage{booktabs,subfigure,xcolor}

\usepackage{geometry}
\allowdisplaybreaks[4]
  \def\cA{{\cal A}}  
  \def\cB{{\cal B}}  
\def\dbC{\mathbb{C}}    
 \def\sD{\mathscr{D}}   
\def\dbE{\mathbb{E}}  \def\cE{{\cal E}}  
\def\dbF{\mathbb{F}} \def\sF{\mathscr{F}} \def\cF{{\cal F}}  
    
\def\dbH{\mathbb{H}}   \def\BH{{\bm H}} 
    
  \def\cJ{{\cal J}}  
    
  \def\cL{{\cal L}}  
  \def\cM{{\cal M}}

\def\dbP{\mathbb{P}}    
    
\def\dbR{\mathbb{R}}    
\def\dbS{\mathbb{S}}    
    
\def\dbU{\mathbb{U}} \def\sU{\mathscr{U}}

 \def\sX{\mathscr{X}} \def\cX{{\cal X}}  
  \def\cY{{\cal Y}}  
  \def\cZ{{\cal Z}}  

\def\BTh{\boldsymbol\Theta}

\def\ss{\smallskip}             \def\hb{\hbox}
\def\ms{\medskip}              
        \def\lan{\langle}    \def\as{\hb{a.s.}}
\def\ds{\displaystyle}   \def\ran{\rangle}    \def\tr{\hb{tr$\,$}}
         
\def\no{\noindent}          \def\var{\hb{var$\,$}}
         \def\det{\hb{det\,}}
\def\nn{\nonumber}         
\def\rf{\eqref}            
\def\cd{\cdot}             \def\essinf{\hb{essinf}}
\def\cds{\cdots}
\def\deq{\triangleq}     \def\({\Big (}       \def\ba{\begin{aligned}}
\def\les{\leqslant}      \def\){\Big )}       \def\ea{\end{aligned}}
\def\ges{\geqslant}      \def\[{\Big[}        \def\bel{\begin{equation}\label}
\def\ti{\tilde}          \def\]{\Big]}        \def\ee{\end{equation}}
\def\wt{\widetilde}      \def\q{\quad}        
\def\h{\widehat}         \def\qq{\qquad}      
\def\ns{\noalign{\ss}}                                              
\def\a{\alpha}  \def\G{\Gamma}   \def\g{\gamma}   \def\Om{\Omega}  
   \def\D{\Delta}   \def\d{\delta}   \def\F{\Phi}     
   \def\Th{\Theta}  \def\th{\theta}    \def\si{\sigma}
\def\f{\varphi}   \def\l{\lambda}  \def\m{\mu}      \def\e{\varepsilon}
\def\t{\tau}    \def\i{\infty}   \def\k{\kappa}   
\def\ba{\begin{array}}                \def\ea{\end{array}}
\def\bel{\begin{equation}\label}      \def\ee{\end{equation}}

\def\5n{\negthinspace \negthinspace \negthinspace \negthinspace \negthinspace }
\def\4n{\negthinspace \negthinspace \negthinspace \negthinspace }
\def\3n{\negthinspace \negthinspace \negthinspace }
\def\2n{\negthinspace \negthinspace }
\def\1n{\negthinspace }

\def\liminf{\mathop{\underline{\rm lim}}}

\def\essinf{\mathop{\rm essinf}}

\def\be{\begin{equation}}
\def\bel{\begin{equation}\label}
\def\ee{\end{equation}}
\def\bea{\begin{eqnarray}}
\def\eea{\end{eqnarray}}
\def\bt{\begin{theorem}\label}
\def\et{\end{theorem}}
\def\bc{\begin{corollary}\label}
\def\ec{\end{corollary}}
\def\bex{\begin{example}\label}
\def\ex{\end{example}}
\def\bl{\begin{lemma}\label}
\def\el{\end{lemma}}
\def\bp{\begin{proposition}\label}
\def\ep{\end{proposition}}
\def\br{\begin{remark}\label}
\def\er{\end{remark}}
\def\ba{\begin{array}}
\def\ea{\end{array}}
\def\bde{\begin{definition}\label}
\def\ede{\end{definition}}

\newtheoremstyle{thry}
{}      
{}      
{\sl}   
{}      
{\bf}   
{.}     
{.5em}  
{}      

\theoremstyle{thry}

\newtheorem{theorem}{Theorem}[section]
\newtheorem{proposition}[theorem]{Proposition}
\newtheorem{corollary}[theorem]{Corollary}
\newtheorem{lemma}[theorem]{Lemma}

\theoremstyle{definition}
\newtheorem{definition}[theorem]{Definition}
\newtheorem{example}[theorem]{Example}

\newtheorem{taggedassumptionx}{}
\newenvironment{taggedassumption}[1]
 {\renewcommand\thetaggedassumptionx{#1}\taggedassumptionx}
 {\endtaggedassumptionx}

\theoremstyle{remark}
\newtheorem{remark}[theorem]{Remark}

\makeatletter
   
   \@addtoreset{equation}{section}
   \newcommand{\setword}[2]{%
   \phantomsection
   #1\def\@currentlabel{\unexpanded{#1}}\label{#2}%
   }
\makeatother

\begin{document}

\title{\bf Optimal Controls for Forward-Backward Stochastic Differential Equations:
Time-Inconsistency and Time-Consistent Solutions}


\author{
Hanxiao Wang\thanks{ College of Mathematics and Statistics, Shenzhen University, Shenzhen 518060, China
 (Email: {\tt hxwang@szu.edu.cn}). }
                           ~~~
Jiongmin Yong\thanks{Department of Mathematics, University of Central Florida,
                           Orlando 32816, USA (Email: {\tt jiongmin.yong@ucf.edu}).
                           This author is supported in part by NSF Grant DMS-1812921.}
                          ~~~
Chao Zhou\thanks{ Department of Mathematics and Risk Management Institute, National University of Singapore,
                           Singapore 119076, Singapore (Email: {\tt matzc@nus.edu.sg}).
                           This author is supported by  NSFC Grant 11871364
                           and Singapore MOE  AcRF Grants A-800453-00-00, R-146-000-271-112, R-146-000-284-114.}}

\maketitle

\no\bf Abstract. \rm
This paper is concerned with an optimal control problem for a forward-backward stochastic differential equation
(FBSDE, for short) with a recursive cost functional determined by a backward stochastic Volterra integral equation
(BSVIE, for short). It is found that such an optimal control problem is time-inconsistent in general,
even if the cost functional is reduced to a classical Bolza  type one as in Peng \cite{Peng1993},
Lim--Zhou \cite{Lim-Zhou2001}, and Yong \cite{Yong2010}. Therefore, instead of finding a global optimal control
(which is time-inconsistent), we will look for a time-consistent and locally optimal equilibrium strategy,
which can be constructed via the solution of an associated equilibrium Hamilton--Jacobi--Bellman (HJB, for short) equation.
A verification theorem for the local optimality of the equilibrium strategy is proved by means of
the generalized Feynman--Kac formula for BSVIEs and some stability estimates of the representation for
parabolic partial differential equations (PDEs, for short). Under certain conditions,
it is proved that the equilibrium HJB equation, which is a nonlocal PDE, admits a unique classical solution.
As special cases and applications, the linear-quadratic problems, a mean-variance model,
a social planner problem with heterogeneous Epstein--Zin utilities,
and a Stackelberg game are briefly investigated.
It turns out that our framework can cover not only the optimal control problems for FBSDEs studied in \cite{Peng1993,Lim-Zhou2001,Yong2010}, and so on,
but also the problems of the general discounting and some nonlinear appearance of conditional expectations for the terminal state,
studied in Yong \cite{Yong2012,Yong2014} and Bj\"{o}rk--Khapko--Murgoci \cite{Bjork-Khapko-Murgoci2017}.

\ms

\no\bf Keywords. \rm Time-inconsistent optimal control problem,
controlled forward-backward stochastic differential equation,
backward stochastic Volterra integral equation,
equilibrium strategy, equilibrium Hamilton--Jacobi--Bellman equation,
recursive utility, Feynman--Kac formula.

\ms

\no\bf AMS subject classifications. \rm 93E20, 49N70, 60H10, 60H20, 35K10, 49L20, 90C39.

\section{Introduction}

Let $(\Om,\cF,\dbP)$ be a complete probability space on which a standard one-dimensional Brownian motion
$W=\{W(t);0\les t<\i\}$ is defined. The augmented natural filtration of $W(\cd)$ is denoted by $\dbF=\{\cF_t\}_{t\ges0}$.
Let $T>0$ be a fixed time horizon. We denote
\begin{align*}
&\sX_t=L^2_{\cF_t}(\Om;\dbR^n)=\big\{\xi:\Om\to\dbR^n~| ~\xi \hbox{ is  $\cF_{t}$-measurable, } \dbE[|\xi|^2]<\i\big\},\q t\in[0,T],\\
&\sD=\big\{(t,\xi)~|~t\in[0,T),\,\xi\in\sX_t\big\},\\
&\sU[t,T]=\Big\{\f:[t,T]\times\Om\to U\bigm|\f(\cd)\hb{~is $\dbF$-progressively measurable},~\dbE\int^T_t|\f(s)|^2ds<\i\Big\},\q t\in[0,T],
\end{align*}
where $U\subseteq\dbR^\ell$ is a nonempty measurable set (either bounded or unbounded).
For any given {\it initial pair} $(t,\xi)\in\sD$ and {\it control process} $u(\cd)\in\sU[t,T]$,
consider the following controlled (decoupled) forward-backward stochastic differential equation (FBSDE, for short) on the time horizon $[t,T]$:
\bel{State}\left\{\2n\ba{ll}
\ds dX(s)=b(s,X(s),u(s))ds+\si(s,X(s),u(s))dW(s),\\
\ns\ds dY(s)=-g(s,X(s),u(s),Y(s),Z(s))ds+Z(s)dW(s),\\
\ns X(t)=\xi,\q Y(T)=h(X(T)),\ea\right.\ee
where $b,\si:[0,T]\times\dbR^n\times U\to\dbR^n$, $g:[0,T]\times\dbR^n\times U\times\dbR^m\times\dbR^m\to\dbR^m$,
and $h:\dbR^n\to\dbR^m$ are given deterministic mappings. Under certain mild conditions,
for any $(t,\xi)\in\sD$ and $u(\cd)\in\sU[t,T]$, \rf{State} admits a unique adapted solution $(X(\cd),Y(\cd),Z(\cd))\equiv \big(X(\cd\,;t,\xi,u(\cd)),Y(\cd\,;t,\xi,u(\cd)),Z(\cd\,;t,\xi,u(\cd))\big)$, which is called a {\it state process}.
To measure the performance of the control $u(\cd)$, we introduce the following {\it recursive cost functional}:
\bel{cost-BSVIE0}
J(t,\xi;u(\cd))=Y^0(t),\ee
where $Y^0(\cd)$ is uniquely determined by the following backward stochastic Volterra integral equation (BSVIE, for short):
\begin{align}
Y^0(r)&=h^0(r,X(r),X(T),Y(r))+\int_r^Tg^0(r,s,X(r),X(s),u(s),Y(s),Z(s),Y^0(s),Z^0(r,s))ds\nn\\
&\q-\int_r^TZ^0(r,s)dW(s),\qq\qq r\in[t,T],\label{cost-BSVIE1}
\end{align}
for which $(Y^0(\cd),Z^0(\cd\,,\cd))$ is the adapted solution. Here,  $h^0:[0,T]\times\dbR^n\times\dbR^n\times\dbR^m\to\dbR$ and $g^0:\D^*[0,T]\times\dbR^n\times\dbR^n\times U\times\dbR^m\times\dbR^m
\times\dbR\times\dbR\to\dbR$ are given deterministic mappings with
\bel{D}\D^*[0,T]=\big\{(t,s)\in[0,T]^2\bigm|0\les t\les s\les T\big\}\ee
being the upper triangle domain in the square $[0,T]^2$. Note that in the case that
\bel{Bolza1}h^0(r,\ti x,x,y)=h^0(x,y),\qq g^0(r,s,\ti x,x,u,y,z,y^0,z^0)=g^0(s,x,u,y,z),\ee
the recursive cost functional \rf{cost-BSVIE0}--\rf{cost-BSVIE1} is reduced to a Bolza type cost functional
for  FBSDE state equation (see Peng \cite{Peng1993} and Yong \cite{Yong2010}, for examples):
\bel{Bolza2}J(t,\xi;u(\cd))=\dbE_t\[h^0(X(T),Y(t))+\int_t^T g^0(s,X(s),u(s),Y(s),Z(s))ds\],\ee
where $\dbE_t[\,\cd\,]=\dbE[\,\cd\,|\sF_t]$ is the conditional expectation operator. Further, if
\bel{Bolza3}h^0(r,\ti x,x,y)=h^0(x),\qq g^0(r,s,\ti x,x,u,y,z,y^0,z^0)=g^0(s,x,u),\ee
then the cost functional is reduced to the most familiar classical Bolza functional:
\bel{Bolza4}J(t,\xi;u(\cd))=\dbE_t\[h^0(X(T))+\int_t^T g^0(s,X(s),u(s))ds\],\ee
where the two terms on the right-hand side are called {\it terminal} and {\it running} costs, respectively.
Thus, our recursive cost functional is an extension of Bolza type cost functional.
With the state equation \rf{State} and the recursive cost functional \rf{cost-BSVIE0}--\rf{cost-BSVIE1},
we may pose the following optimal control problem:

\ms

{\bf Problem (N)}. For any given initial pair $(t,\xi)\in\sD$, find a control $\bar u(\cd)\in\sU[t,T]$ such that
\bel{inf}J(t,\xi;\bar u(\cd))=\essinf_{u(\cd)\in\sU[t,T]}J(t,\xi;u(\cd))=V(t,\xi).
\ee

Any $\bar u(\cd)\in\sU[t,T]$ satisfying \rf{inf} is called an ({\it open-loop}) {\it optimal control} of Problem (N) for the initial pair $(t,\xi)$;
the corresponding state process $(\bar X(\cd),\bar Y(\cd),\bar Z(\cd))\equiv \big(X(\cd\,;t,\xi,u(\cd)),Y(\cd\,;t,\xi,u(\cd)),Z(\cd\,;t,\xi,u(\cd))\big)$
is called an ({\it open-loop}) {\it optimal state process}; and $V(\cd\,,\cd):\sD\to\dbR$ is called the {\it value function} of Problem (N).

\ms

We now briefly illustrate the major {\bf motivation} of the above framework as follows:
The (vector-valued) process $X(\cd)$ follows a (forward) stochastic differential equation (FSDE, for short).
Components of $X(\cd)$ consist of two types processes: uncontrolled ones (by the individuals),
including  prices of securities (such as bonds, stocks), some economic factors (such as interest rates, unemployment rates, GDP, etc.), and controlled ones (by the individuals), including market values of the investor's wealth (subject to trading strategies),
inventory of commodities (subject to the ordering), amounts of goods (subject to the production), etc.
On the other hand, the components of $Y(\cd)$, following a multi-dimensional backward stochastic differential equation
(BSDE, for short), could include the prices of some European type contingent claims of the underlying assets
(whose prices are some components of $X(\cd)$), and some dynamic risk measures, and so on.
Therefore, it is natural to have an FBSDE as a state equation. Further, the dynamic expected utility/disutility of the total assets will be calculated in the recursive way, which can be described by the adapted solution to a BSVIE (see below).
Putting all the above together, we have the framework and the formulation of the problem.

\ms

Let us now briefly illustrate the recursive cost functional of form \rf{cost-BSVIE0}--\rf{cost-BSVIE1}.
In 1992, Duffie--Epstein \cite{Duffie-Epstein1992,Duffie-Epstein1992-1} introduced a stochastic differential formulation
of recursive utility in the case of information generated by a Brownian motion. In 1997, El Karoui--Peng--Quenez \cite{El Karoui-Peng-Quenez1997}
showed that such a process actually is a part of the adapted solution to a particular BSDE and then they defined a more general class of recursive utilities,
through a general BSDE (see also Lazrak--Quenez \cite{Lazrak-Quenez2003} and Lazrak \cite{Lazrak2004} for further developments).
The main feature of such a recursive process, denoted by $Y^R(\cd)$, is that the current value $Y^R(t)$ depends on the future values $Y^R(s)$,
$t<s\les T$ of the process. Then, on top of the classical Bolza type cost functional \rf{Bolza2}, mimicking \cite{El Karoui-Peng-Quenez1997},
for our FBSDE state equation, it is natural to introduce the following {\it recursive} cost functional:
\bel{cost-R}J^R(t,\xi;u(\cd))= Y^R(t),\ee
where $Y^R(\cd)$ is determined by the following equation:
\bel{cost-R1}
Y^R(r)=\dbE_r\[h^0(X(T),Y(r))+\int_r^Tg^0(s,X(s),u(s),Y(s),Z(s),Y^R(s))ds\],\qq r\in[t,T].\ee
From the above, we see that the value $Y^R(t)$ depends on the values $Y^R(s)$ for $s\in[t,T]$,
through equation \rf{cost-R1}. Hence the cost process $Y^R(\cd)$ has a recursive feature, and thus its name.
By Yong \cite{Yong2008}, for some process $Z^R(\cd\,,\cd)$,
the pair $(Y^R(\cd),Z^R(\cd\,,\cd))$ is the adapted solution to the following BSVIE:
\bel{cost-BSVIE-R}
Y^R(r)=h^0(X(T),Y(r))+\int_r^Tg^0(s,X(s),u(s),Y(s),Z(s),Y^R(s))ds-\int_r^TZ^R(r,s)dW(s),\q r\in[t,T].
\ee
Note that \rf{cost-BSVIE-R} (i.e., \rf{cost-R1}) is not a BSDEs on $[t,T]$,
because the free term $h^0(X(T),Y(r))$ depends on the time variable $r$,
which leads to the adjustment process $Z^R(r,s)$ depending on $s$ and $r$.
Inspired by the above, we introduce the general recursive cost functional \rf{cost-BSVIE0}--\rf{cost-BSVIE1}.
Note that in BSVIE \rf{cost-BSVIE1}, the {\it free term} $h^0(\cd)$ and the {\it generator} $g^0(\cd)$
are allowed to depend on the initial pair $(r,X(r))$ at the current time $r$,
which is motivated by the {\it non-exponential discounting} \cite{Karp2007,Ekeland2008,Yong2012}
and the {\it state-dependent risk aversion} \cite{Bjork-Murgoci-Zhou2014,Hu-Jin-Zhou2017} in finance.
The recursive cost functional of form \rf{cost-BSVIE0}--\rf{cost-BSVIE1} was introduced the first time
by Wang--Yong \cite{Wang-Yong2021}, motivated by the recursive utility/disutility process for classical optimal control problems.
Comparing with the cost functional studied in \cite{Wang-Yong2021}, we see that the free term $h^0(\cd)$
and the generator $g^0(\cd)$ of BSVIE \rf{cost-BSVIE1} are additionally allowed to depend on the initial state $X(r)$ and the backward process $(Y(\cd),Z(\cd))$.
Moreover, we highlight that \rf{cost-BSVIE0}--\rf{cost-BSVIE1} can also be regarded as a recursive version of the cost functional studied in Bj\"{o}rk--Khapko--Murgoci \cite{Bjork-Khapko-Murgoci2017}, because $\dbE_\cd[X(T)]$ is the backward state process $Y(\cd)$ of a trivial BSDE.

\ms

It is well-known by now that the introduction of BSDEs by Bismut \cite{Bismut1973, Bismut1978}
in the early 1970s was for the purpose of studying optimal control of FSDEs.
The later developments of general BSDEs by Pardoux--Peng \cite{Pardoux-Peng1990}
(see also Duffie--Epstein \cite{Duffie-Epstein1992} and El Karoui--Peng--Quenez \cite{El Karoui-Peng-Quenez1997}),
and the extension to FBSDEs by Antonelli \cite{Antonelli1993},  Ma--Protter--Yong \cite{Ma-Protter-Yong1994},
Hu--Peng \cite{Hu-Peng1995} (see also the books of Ma--Yong \cite{Ma-Yong1999} and Zhang \cite{Zhang2017})
have been attracting many researchers' attention.
Among many other publications, a big number of literature on the optimal control problems for FBSDEs/BSDEs keep appearing.
See, Peng \cite{Peng1993}, Xu \cite{Xu1995}, Dokuchaev--Zhou \cite{Dokuchaev1999}, Ji--Zhou \cite{Ji-Zhou2006},
Shi--Wu \cite{Shi-Wu2006}, Huang--Wang--Xiong \cite{Huang-Wang-Xiong2009}, Yong \cite{Yong2010},
Wang--Wu--Xiong \cite{Wang-Wu-Xiong2013}, and Hu--Ji--Xue \cite{Hu-Ji-Xue2018} on the Pontryagin's maximum principle
for controlled BSDEs/FBSDEs; Lim--Zhou \cite{Lim-Zhou2001}, Wang--Wu--Xiong \cite{Wang-Wu-Xiong2015},
Huang--Wang--Wu \cite{Huang-Wang-Wu2016}, Wang--Xiao--Xiong \cite{Wang-Xiao-Xiong2018}, Li--Sun--Xiong \cite{Li-Sun-Xiong2019},
Hu--Ji--Xue \cite{Hu-Ji-Xue2019}, Sun--Wang \cite{Sun-Wang2019}, Sun--Wu--Xiong \cite{Sun-Wu-Xiong2021}, Sun--Wang--Wen \cite{Sun-Wang-Wen2021} on the linear-quadratic (LQ, for short) optimal control problems for BSDEs/FBSDEs; and so on. It is observed that the problems investigated in the above listed works are all essentially the special cases of Problem (N),
and have been treated as usual stochastic optimal control problems.
There is an {\it essential feature} has been overlooked in all the above, which we now indicate that.

\ms

For a dynamic optimal control problem, suppose that at a given initial pair $(t,\xi)\in\sD$,
the problem has an (open-loop) optimal control $\bar u(\cd)\equiv\bar u(\cd\,;t,\xi)$
with the (open-loop) optimal state being $\bar X(\cd)\equiv X(\cd\,;t,\xi,\bar u(\cd))$.
Then, we could not expect the following:
\bel{J=J*}J\big(\t,\bar X(\t);\bar u(\cd)\big|_{[\t,T]}\big)=\inf_{u(\cd)\in\sU[\t,T]}J(\t,\bar X(\cd);u(\cd)),\qq\forall\t\in(t,T],~\as\ee
In other words, an optimal control selected at a given initial pair might not stay optimal thereafter. Then, we say that the optimal control problem is {\bf time-inconsistent}.
It turns out that, in general, Problem (N) is time-inconsistent,
as the {\it dynamical programming principle} (DPP, for short) does not hold.
This reveals a surprising feature of Problem (N). To see that, let us elaborate the time-inconsistency
in a little more details, from which we will see how Problem (N) is generally time-inconsistent.

\ms

$\bullet$ \it Time-preferences and discounting.
\rm Suppose the continuously compound interest rate is a constant $\l>0$.
Then one needs to deposit an amount $e^{-\l T_0}$ at $\t$ in order to get 1 unit at $\t+T_0$.
We call $e^{-\l T_0}$ the {\it discount factor} of the time interval $[\t,\t+T_0]$,
which could also be defined as the {\it value} of this time interval.
Clearly, such a value $e^{-\l T_0}$ of $[\t,\t+T_0]$ is independent of the initial time $\t$
and it is also independent of the time $t\in[0,\infty)$ at which $[\t,\t+T_0]$ is evaluated, either $t\les\t$ or $t>\t$.
Because of this, such an exponential evaluation is said to be {\it rational}. Or equivalently,
rationality can be described by the {\it exponential discounting}. On the other hand,
it is common that most people overweight the utility of the immediate future events,
which can be convinced by the fact that one often regrets the (optimal) decisions made earlier.
This means that people evaluate the immediate future time period more expensively than it should be,
which amounts to saying that the discount factor for that time interval is larger than the rational one.
Hence, we  need to replace the exponential discounting by more general ones to more precisely describe the real situations.

\ms

In the above recursive cost functional \rf{cost-BSVIE0}--\rf{cost-BSVIE1}, if we have
$$h^0(r,\ti x,x,y)=e^{-\l(T-r)}h^0(x),\qq g^0(r,s,\ti x,x,u,y,z,y^0,z^0)=e^{-\l(s-r)}g^0(s,x,u),$$
for some discount rate $\l>0$, then the cost functional is reduced to the classical exponential discounting Bolza cost functional:
$$J(t,\xi;u(\cd))=\dbE_t\[e^{-\l(T-t)}h^0(X(T))+\int_t^T
e^{-\l(s-t)}g^0(s,X(s),u(s))ds\].$$
In this case, there are no (European type) contingent claims involved, and there are no dynamic risks taken into account. Therefore, the BSDE for $(Y(\cd),Z(\cd))$ in \rf{State} is irrelevant. Also, the involved individual is completely {\it rational} (as far as the time-preferences are concerned). For such a case, the corresponding Problem (N) is time-consistent.
Now, if $e^{-\l(T-t)}$ and $e^{-\l(s-t)}$ are replaced by some non-exponential decay functions,
the cost functionals are referred to as non-exponential ones, which describe some kinds of {\it irrationality} of time-preferences for the involved individuals. In this case, namely, the cost functional is given by \rf{cost-BSVIE0}--\rf{cost-BSVIE1}, our Problem (N) is time-inconsistent. The earliest mathematical consideration in this aspect was given by Strotz \cite{Strotz1955}, followed by Pollak \cite{Pollak1968}, and the recent works of  Ekeland--Pirvu \cite{Ekeland2008}, Ekeland--Lazrak \cite{Ekeland2010}, Yong \cite{Yong2012,Yong2014,Yong2017}, Wei--Yong--Yu \cite{Wei-Yong-Yu2017},
Mei--Yong \cite{Mei-Yong2019}, Mei--Zhu \cite{Mei-Zhu2020}, Wang--Yong \cite{Wang-Yong2021}, Hamaguchi \cite{Hamaguchi2020}, and Hern\'{a}ndez--Possamai \cite{Hernandez-2020} for various kinds of problems relevant to non-exponential discounting.

\ms

$\bullet$ \it Risk-preferences and nonlinear appearance of conditional expectations of the (terminal) state.
\rm Different groups of people should have different opinions of risks on the in-coming events. This is referred to as people's {\it subjective} risk-preferences. One way to describe this is to allow the conditional expectation of the state to (nonlinearly) appear in the cost functional. It turns out that such a formulation will lead to time-inconsistency of the optimal control problem in general.
See Basak--Chabakauri \cite{Basak-Chabakauri2010}, Hu--Jin--Zhou \cite{Hu-Jin-Zhou2012,Hu-Jin-Zhou2017},
Bj\"ork--Murgoci \cite{Bjork-Murgoci2014}, Bj\"ork--Murgoci--Zhou \cite{Bjork-Murgoci-Zhou2014},
Bj\"{o}rk--Khapko--Murgoci \cite{Bjork-Khapko-Murgoci2017},  Yong \cite{Yong2017},
and He--Jiang \cite{He-Jiang2021} for some relevant results.

\ms

Let us now make an interesting observation for our Problem (N). Let $m=n$, and
$$h(x)=x,\q g(s,x,u,y,z)\equiv0,\q h^0(r,\ti x,x,y)=h^0(x,y),\q g^0(r,s,\ti x,x,u,y,z,y^0,z^0)=g^0(s,x,u,y),$$
then
$$Y(s)=\dbE_s[X(T)],\qq s\in[t,T],$$
and the recursive cost functional \rf{cost-BSVIE0}--\rf{cost-BSVIE1} becomes
$$
J(t,\xi;u(\cd))=\dbE_t\[h^0\big(X(T),\dbE_t[X(T)]\big)+\int_t^Tg^0\big(s,X(s),u(s),
\dbE_s[X(T)]\big)ds\].$$
In the above, $\dbE_t[X(T)]$ appears nonlinearly and the corresponding optimal control problem is time-inconsistent. From the above observation, we see that the state equation being an FBSDE can include many situations of nonlinear appearance of conditional expectations. Therefore, Problem (N) is intrinsically time-inconsistent. Some special cases were investigated by Basak--Chabakauri \cite{Basak-Chabakauri2010}, Hu--Jin--Zhou \cite{Hu-Jin-Zhou2012,Hu-Jin-Zhou2017},  Bj\"{o}rk--Murgoci--Zhou \cite{Bjork-Murgoci-Zhou2014}, Bj\"{o}rk--Khapko--Murgoci \cite{Bjork-Khapko-Murgoci2017}.
Our Problem (N) also partially covers the case studied in Yong \cite{Yong2017}.

\ms

We have seen that Problem (N) is generally time-inconsistent.
Therefore, we should treat it from the angle differently from the usual classical ones.
Before going further, let us present the following simple example,
from which we will see further the essential reason for Problem (N) to be time-inconsistent.

\bex{example2} Consider the one-dimensional (degenerate) FBSDE state equation
\bel{state-example2}\left\{\2n\ba{ll}
\ds\dot X(s)=0,\qq s\in[t,T],\\
\ns\ds\dot{Y}(s)=u(s),\qq s\in[t,T],\\
\ns\ds X(t)=x,\qq Y(T)=0,\ea\right.\ee
with the cost functional
\bel{cost-example2}J(t,x;u(\cd))=\int_t^T[Y(s)+u(s)+|u(s)|^2]ds.\ee
A straightforward calculation (see \autoref{example-y} for details) shows that at the initial pair $(t,x)$,
the unique optimal control $\bar u(\cd\,;t,x)$ is given by
$$\bar u(s)\equiv\bar u(s;t,x)={s-t-1\over 2},\qq s\in[t,T].$$
Then, for any $\t\in(t,T)$, the unique optimal control at $(\t,\bar X(\t))\equiv(\t,x)$ is given by
$$\ti u(s)\equiv\ti u(s;\t,\bar X(\t))={s-\t-1\over2},\qq s\in[\t,T].$$
Clearly,
$$
\bar u(s)\neq \ti u(s),\qq s\in[\t,T].
$$
Thus, the problem is time-inconsistent.

\ex

It is worthy of pointing out that in the above example, \rf{cost-example2} is a  Bolza type cost functional for FBSDE state equations, and unlike Yong \cite{Yong2012} and Bj\"{o}rk--Khapko--Murgoci \cite{Bjork-Khapko-Murgoci2017}, neither non-exponential discounting
nor conditional expectations (nonlinearly) appear. Furthermore, the controlled system \rf{state-example2} is a deterministic ordinary differential equation, and the terminal cost of \rf{cost-example2} equals zero,
due to which \rf{cost-example2} is also a Lagrange type cost functional.
This tells us that \textbf{an optimal control problem could be time-inconsistent solely
because the state equation is a forward-backward one}.
Hence, the time-inconsistency feature is intrinsically contained in the optimal control problems for FBSDEs. Such a feature distinguishes the current paper from the previous ones concerning the time-inconsistency, in other aspects.

\ms

Having the above time-inconsistent feature of the problem, we now highlight the main results of this paper.

\begin{enumerate}[(i)]

\item
Using Pontryagin's maximum principle, we will show that Problem (N) is generically time-inconsistent.
The advantage of such an approach is that we are not satisfied with just some counterexamples,
instead, we will show that if $\bar u(\cd)$ is optimal at $(t,\xi)$, which will satisfy the Pontragin's type maximum principle (MP, for short) on $[t,T]$, the $\bar u(\cd)\big|_{[t,T]}$ hardly satisfies MP on $[\t,T]$. Therefore, \rf{J=J*} should not be expected in general.

\item
Since Problem (N) is time-inconsistent in general,
finding an optimal control at any given initial pair $(t,\xi)$ is not very useful.
Instead, one should find an {\it equilibrium strategy}, which is time-consistent and possesses certain kind of local optimality. Inspired by Yong \cite{Yong2012}, we derive the {\it equilibrium HJB equation} associated with Problem (N),
through which an equilibrium strategy can be constructed. Our equilibrium HJB equation can cover the results obtained in
Yong \cite{Yong2012} and  Bj\"{o}rk--Khapko--Murgoci \cite{Bjork-Khapko-Murgoci2017}.
In the case that the recursive cost functional is governed by a BSDE, one could apply the method of multi-person differential games,
by viewing that the controller is playing a cooperative game with all his incarnations in the future.
Such an idea can be traced back to the work of Pollak \cite{Pollak1968} in 1968.
Later, the approach was adopted and further developed in \cite{Ekeland2010, Ekeland2008, Yong2012,Yong2014,Yong2017,
Bjork-Murgoci2014, Bjork-Murgoci-Zhou2014, Bjork-Khapko-Murgoci2017, Wei-Yong-Yu2017, Mei-Yong2019, Mei-Zhu2020, Wang-Yong2021}.
We point out that the multi-person differential game approach used in \cite{Yong2012,Wei-Yong-Yu2017}
does not directly apply to Problem (N) of the current paper, because the DPP does not hold
for controlled FBSDEs even if the cost functional does not depend on the initial values $(t,X(t),Y(t))$.
We overcome the difficulty by making use of the Feymann--Kac formula for BSVIEs,
which has been recently well-developed in our works \cite{Wang-Yong2019,Wang2021,Wang-Yong-Zhang2021}.
In the proof of the verification theorem, some technical assumptions imposed in
\cite{Wei-Yong-Yu2017,Wang-Yong2021} and \cite{Bjork-Khapko-Murgoci2017} are relaxed.

\item
When the diffusion term of the forward state equation does not depend on the control $u(\cd)$,
the equilibrium HJB equation associated with Problem (N) is a system of semi-linear parabolic partial differential equations
with non-local terms. Under the non-degenerate condition, the well-posdness of the equilibrium HJB equation is established
in the sense of classical solutions.

\item
Some comparisons between our equilibrium HJB equations and those derived by Peng \cite{Peng19971},
by Yong \cite{Yong2012,Yong2014}, and by Bj\"{o}rk--Khapko--Murgoci \cite{Bjork-Khapko-Murgoci2017} are carefully made, respectively. We find that the backward controlled equation has a significant influence on  the form of the
associated equilibrium HJB equation.
When Problem (N) is reduced to the problem studied by Bj\"{o}rk--Khapko--Murgoci \cite{Bjork-Khapko-Murgoci2017},
the form of our equilibrium HJB equations is more natural than their so-called {\it extended HJB equation}.
We note that there was no rigorous proof on the well-posedness of the extended HJB equation presented in \cite{Bjork-Khapko-Murgoci2017}.
In addition to the above, the ``HJB equation" associated with the value function of Problem (N) is formally derived,
provided Problem (N) has an optimal control with the closed-loop representation,
which can be regarded as a PDE approach version of Peng \cite{Peng1993} and Yong \cite{Yong2010}.
By comparing the equilibrium HJB equation and the ``HJB equation", we find that the optimality condition
of the ``HJB equation" is not a minimization problem in the finite
dimensional space, due to which the ``HJB equation" is not useful and Problem (N) is time-inconsistent in general.

\item
The linear-quadratic optimal control problems for FBSDEs are briefly studied
and a linear equilibrium strategy is  obtained, provided the associated Riccati equation is solvable. This partially covers the work of Yong \cite{Yong2017}. Further, as applications, a mean-variance model, a social planner model of Merton's  consumption--portfolio selection with heterogeneous Epstein--Zin utilities, and a Stackelberg game are investigated, which are all special cases of Problem (N). It is shown that these specific problems are all time-inconsistent, and by the theoretical results obtained in the paper, the associated equilibrium strategies can be explicitly constructed.
\end{enumerate}

The rest of this paper is organized as follows.
In Section \ref{sec:main-results}, we state the main results of our paper, with some explanations.
In Section \ref{sec:comparison}, we compare the results obtained in the paper with the existing ones.
The linear-quadratic problem is studied in Section \ref{sec:LQ},
and three applications are presented in Section \ref{sec:SC}.
In Section \ref{sec:Verification}, the verification theorem is proved.
Some technical and lengthy proofs are given in Section \ref{sec:Proofs}.

\section{The Main Results}\label{sec:main-results}

\subsection{Preliminaries: Notations and Feynman--Kac formula}

Let $T>0$ be a given time horizon and recall the upper triangle domain $\D^*[0,T]$ from \rf{D}. Let $\dbS^n$  be the subspace of $\dbR^{n\times n}$ consisting of symmetric matrices and
$U\subseteq \dbR^l$ be a nonempty measurable set which could be bounded or unbounded.
We will use $K>0$ to represent a generic constant which could be different from line to line. For any Euclidean space $\dbH$ (as well as $\dbH_1$, $\dbH_2$),
we introduce the following spaces:
\begin{align*}
%
& L_\dbF^2(\Om;C([0,T];\dbH))
=\Big\{\f:[0,T]\times\Om\to\dbH\bigm|\f(\cd)~\hb{is $\dbF$-adapted, pathwise continuous, }\dbE\big[\ds\sup_{0\les s\les T}|\f(s)|^2\big]<\i \Big\};\\
& C([0,T];L_\dbF^2(\Om;\dbH))
=\Big\{\f:[0,T]\times\Om\to\dbH\bigm|\f(\cd)~\hb{is $\dbF$-adapted, $\dbE[\f(\cd)]$ is continuous, }
\ds\sup_{0\les s\les T}\dbE\big[|\f(s)|^2\big]<\i \Big\};\\
& L_\dbF^2(0,T;\dbH)
=\Big\{\f:[0,T]\times\Om\to\dbH\bigm|\f(\cd)~\hb{is $\dbF$-progressively measurable on $[0,T]$, }\dbE\int_0^T|\f(s)|^2ds<\i \Big\};\\
& C([0,T];L_\dbF^2(\cd\,,T;\dbH))
\1n=\1n\Big\{\f\1n:\1n\D^*[0,T]\1n\times\1n\Om\1n\to\1n\dbH\bigm|\f(t,\cd)\1n\in\1n L_\dbF^2(t,T;\dbH),~t\1n\in\1n[0,T],~\dbE\1n\int_\cd^T\3n|\f(\cd\,,s)|^2ds\1n\in\1n C([0,T]) \Big\};\\
&L^\i(\dbH_1;\dbH_2)=\big\{\f:\dbH_1\to\dbH_2\bigm|\f(\cd)~\hb{is essentially bounded}\big\};\\
&C^k(\dbH_1;\dbH_2)=\big\{\f:\dbH_1\to\dbH_2\bigm|\f(\cd)~\hb{is $j$-th continuously differentiable for any $0\les j\les k$}\big\};\\
& C_b^k(\dbH_1;\dbH_2)=\big\{\f:\dbH_1\to\dbH_2\bigm|\f(\cd)\in C^k(\dbH_1;\dbH_2),~\hb{the $j$-th derivatives are bounded, $0\les j\les k$}\big\}.
\end{align*}

To guarantee the well-posedness of the controlled FBSDE \rf{State} and BSVIE \rf{cost-BSVIE1} governing the recursive cost functional, we introduce the following assumptions.

\begin{taggedassumption}{(H1)}\label{ass:H1}
Let the mappings $b,\si:[0,T]\times\dbR^n\times U\to\dbR^n$,
$g:[0,T]\times\dbR^n\times U\times\dbR^m\times\dbR^m\to\dbR^m$,
and $h:\dbR^n\to\dbR^m$ be continuous. There exists a constant $L>0$ such that
\begin{align*}
&|b(s,0,u)|+|\si(s,0,u)|+|h(0)|+|g(s,0,u,0,0)|\les L(1+|u|),\\
&|b(s,x_1,u)-b(s,x_2,u)|+|\si(s,x_1,u)-\si(s,x_2,u)|+|h(x_1)-h(x_2)|\\
&+|g(s,x_1,u,y_1,z_1)-g(s,x_2,u,y_2,z_2)|\les L\big[|x_1-x_2|+|y_1-y_2|+|z_1-z_2|\big],\\
&\qq\qq\qq\q\forall(s,u)\in [0,T]\times U,~ (x_i,y_i,z_i)\in\dbR^n\times\dbR^m\times\dbR^m,~ i=1,2.
\end{align*}
\end{taggedassumption}

\begin{taggedassumption}{(H2)}\label{ass:H2}
Let the mappings $h^0:[0,T]\times\dbR^n\times\dbR^n\times\dbR^m\to\dbR$
and $g^0:\D^*[0,T]\times\dbR^n\times\dbR^n\times U\times\dbR^m\times\dbR^m\times\dbR\times\dbR\to\dbR$ be continuous. There exists a constant $L>0$ such that
\begin{align*}
&|h^0(t,0,0,0)|+|g^0(t,s,0,0,u,0,0,0,0)|\les L(1+|u|),\\
&|g^0(t_1,s,\ti x_1,x_1,u,y_1,z_1,y^0_1,z^0_1)-g^0(t_2,s,\ti x_2,x_2,u,y_2,z_2,y^0_2,z^0_2)|\\
&+|h^0(t_1,\ti x_1,x_1,y_1)-h^0(t_2,\ti x_2, x_2,y_2)|\\
&\q\les L\big[|t_1-t_2|+|\ti x_1-\ti x_2|+|x_1-x_2|+|y_1-y_2|+|z_1-z_2|+|y^0_1-y^0_2|+|z^0_1-z^0_2|\big],\\
&\qq\qq\qq\forall(t_i,s)\in\D^*[0,T],~\ti x_i,x_i\in\dbR^n,~u\in U,~y_i,z_i\in\dbR^m,~y^0_i,z^0_i\in\dbR,~i=1,2.
\end{align*}
\end{taggedassumption}

\ms

By Yong--Zhou \cite[Chapter 7]{Yong-Zhou1999} and Yong \cite{Yong2008},
we have the following results about the well-posedness of (decoupled) FBSDE \rf{State} and BSVIE \rf{cost-BSVIE1}.

\begin{lemma}\label{lmm:well-posedness-SDE}
Let {\rm\ref{ass:H1}} hold. Then for any initial pair $(t,\xi)\in\sD$ and control $u(\cd)\in\sU[t,T]$,
state equation \rf{State} admits a unique adapted solution $(X(\cd),Y(\cd),Z(\cd))\in L_\dbF^2(\Om;C([t,T];\dbR^n))\times L_\dbF^2(\Om;C([t,T];\dbR^m)) \times L_\dbF^2(t,T;\dbR^m)$. Moreover, there exists a constant $K>0$, independent of $(t,\xi)$ and $u(\cd)$, such that
\bel{|X|}\dbE_t\[\sup_{t\les r\les T}\big[|X(r)|^2+|Y(r)|^2\big]+\int_t^T|Z(s)|^2ds\]\les K\dbE_t\[1+|\xi|^2+ \int_t^T|u(s)|^2ds\].\ee
In addition, if {\rm\ref{ass:H2}} also holds, then for any initial pair $(t,\xi)\in\sD$, control $u(\cd)\in\sU[t,T]$,
and the corresponding state process $(X(\cd),Y(\cd),Z(\cd))$, BSVIE \rf{cost-BSVIE1} admits a unique adapted solution
$(Y^0(\cd),Z^0(\cd\,,\cd))\in C([t,T];L_\dbF^2(\Om;\dbR))\times C([t,T];L_\dbF^2(\cd\,, T;\dbR))$.
Moreover, there exists a constant $K>0$, independent of $(t,\xi)$ and $u(\cd)$, such that
\bel{|Y|+|Z|}
\sup_{t\les r\les T}\dbE_t\[|Y^0(r)|^2+\int_r^T|Z^0(r,s)|^2ds\]\les K\dbE_t\[1+|\xi|^2+\int_t^T|u(s)|^2ds\].\ee
\end{lemma}

\ms

As another preparation, we consider the following system of FBSDEs and BSVIEs without controls:
\bel{FBSDE-BSVIE}\left\{\2n\ba{ll}
\ds X(r)=\xi+\int_t^rb(s,X(s))ds+\int_t^r\si(s,X(s))dW(s),\q r\in[t,T],\\
\ns\ds Y(r)=h(X(T))+\int_r^T g(s,X(s),Y(s),Z(s))ds-\int_r^T Z(s)dW(s),\q r\in[t,T],\\
\ns\ds Y^0(r)=h^0(r,X(r),X(T),Y(r))+\int_r^Tg^0(r,s,X(r),X(s),Y(s),Z(s),Y^0(s),
Z^0(r,s))ds\\
\ns\ds\qq\q\,~-\int_r^TZ^0(r,s)dW(s),\qq r\in[t,T],\ea\right.\ee
where the coefficients $b(\cd),\si(\cd),h(\cd),g(\cd),h^0(\cd),g^0(\cd)$ satisfy \ref{ass:H1}--\ref{ass:H2}
(independent of the control $u$). Suggested by Wang--Yong \cite{Wang-Yong2019}
and Wang--Yong--Zhang \cite{Wang-Yong-Zhang2021},
we introduce the following system of semi-linear PDEs:
\bel{FK-PDE}\left\{\2n\ba{ll}
\ds\Th_s^k(s,x)+{1\over 2}\tr[\Th_{xx}^k(s,x)\si(s,x)\si(s,x)^\top]+ \Th^k_x(s,x)b(s,x)\\
\ns\ds\q+g\big(s,x,\Th(s,x),\Th_x(s,x)\si(s,x)\big)=0,\q(s,x)\in[t,T]\times\dbR^n,\q 1\les k\les m,\\
\ns\ds\Th^0_s(r,s,\ti x,x,y)+{1\over 2}\tr[\Th^0_{xx}(r,s,\ti x,x,y)\si(s,x)\si(s,x)^\top]+\Th^0_x(r,s,\ti x,x,y)b(s,x)\\
\ns\ds\q+g^0\big(r,s,\ti x,x,\Th(s,x),\Th_x(s,x)\si(s,x),\Th^0(s,s,x,x,\Th(s,x)),\Th^0_x(r,s,\ti x,x,y)\si(s,x)\big)=0,\\
\ns\ds\qq\qq\qq\qq\qq\qq\qq\qq\qq\qq(r,s,\ti x,x,y)\in\D^*[t,T]\times\dbR^n\times\dbR^n\times\dbR^m,\\
\ns\ds\Th(T,x)=h(x),\q\Th^0(r,T,\ti x,x,y)=h^0(r,\ti x,x,y),\qq(r,\ti x,x,y)\in[t,T]\times\dbR^n\times\dbR^n\times\dbR^m,
\ea\right.\ee
with $\Th=(\Th^1,\cds,\Th^m)^\top$. Note that $\Th^0(\cd)$ is a function of $(r,s,\ti x,x,y)$,
and $\Th^0_x(\cd), \Th^0_{xx}(\cd)$ are the derivatives with respect to the 4th argument.
We have the following representation theorem.

\bp{Prop:FK-BSVIE}
Suppose that PDE \rf{FK-PDE} admits a classical solution $(\Th(\cd\,,\cd),\Th^0(\cd\,,\cd\,,\cd\,,\cd\,,\cd))$.
Assume that the system \rf{FBSDE-BSVIE} of FBSDEs and BSVIEs  admits a unique adapted solution $(X(\cd),Y(\cd),Z(\cd),Y^0(\cd),Z^0(\cd,\cd))$.
Then the following representation holds:
\begin{align}
&Y(r)=\Th(r,X(r)),\q Z(r)=\Th_x(r,X(r))\si(r,X(r)),\qq r\in[t,T],\,\as,\nn\\
&Y^0(r)=\Th^0\big(r,r,X(r),X(r),\Th(r,X(r))\big),\qq r\in[t,T],\,\as\nn\\
& Z^0(r,s)=\Th^0_x\big(r,s,X(r),X(s),\Th(r,X(r))\big)\si(s,X(s)),\q (r,s)\in\D^*[t,T],\,\as
\label{rep}
\end{align}
\ep

\begin{proof}
By the Feynman--Kac formula for BSDEs (see Pardoux--Peng \cite{Pardoux--Peng1992}, for example), we get
$$Y(r)=\Th(r,X(r)),\q Z(r)=\Th_x(r,X(r))\si(r,X(r)),\q r\in[t,T].$$
Substituting the above into the BSVIE in \rf{FBSDE-BSVIE}, we get
\begin{align*}
Y^0(r)&=h^0(r,X(r),X(T),\Th(r,X(r)))-\int_r^TZ^0(r,s)dW(s)\\
&\q+\int_r^Tg^0\big(r,s,X(r),X(s),\Th(s,X(s)),\Th_x(s,X(s))\si(s,X(s)),Y^0(s),Z^0(r,s)\big)ds,\q r\in[t,T].
\end{align*}
Then by the Feynman--Kac formula for BSVIEs (see \cite{Wang-Yong2019,Wang2021}), we have
\begin{align*}
&Y^0(r)=\h\Th^0\big(r,r,X(r),X(r)\big),\qq r\in[t,T],\,\as,\nn\\
& Z^0(r,s)=\h\Th^0_x\big(r,s,X(r),X(s))\big)\si(s,X(s)),\q (r,s)\in\D^*[t,T],\,\as
\end{align*}
with $\h\Th^0(\cd\,,\cd\,,\cd\,,\cd)$ being the classical solution to the following PDE:
$$\left\{\begin{aligned}
&\h\Th^0_s(r,s,\ti x,x)+{1\over 2}\tr[\h\Th^0_{xx}(r,s,\ti x,x)\si(s,x)\si(s,x)^\top]+\h\Th^0_x(r,s,\ti x,x)b(s,x)\\
&\q+g^0\big(r,s,\ti x,x,\Th(s,x),\Th_x(s,x)\si(s,x),\h\Th^0(s,s,x,x),\h\Th^0_x(r,s,\ti x,x)\si(s,x)\big)=0,\\
&\qq\qq\qq\qq\qq\qq\qq\qq\qq\qq(r,s,\ti x,x)\in\D^*[t,T]\times\dbR^n\times\dbR^n,\\
&\h\Th^0(r,T,\ti x,x)=h^0(r,\ti x,x,\Th(r,\ti x)),\qq(r,\ti x,x)\in[t,T]\times\dbR^n\times\dbR^n.
\end{aligned}\right.$$
Clearly, $\h\Th^0(\cd\,,\cd\,,\cd\,,\cd)$ is actually given by
$$
\h\Th^0(r,s,\ti x,x)=\Th^0(r,s,\ti x,x,\Th(r,\ti x)),\q (r,s,\ti x,x)\in\D^*[0,T]\times\dbR^n\times\dbR^n.
$$
Then the desired results can be obtained immediately.
\end{proof}

\begin{remark} \autoref{Prop:FK-BSVIE} is a generalization of the representation/Feynman--Kac formula for Markovian BSVIEs,
which was established by Wang--Yong \cite{Wang-Yong2019} and Wang \cite{Wang2021}, in the sense of classical solutions.
Under the non-degenerate assumption, the well-posedness of PDE \rf{FK-PDE} will be established by an analytic method, as a byproduct of \autoref{thm:HJB}. The probabilistic approach, without the non-degenerate assumption,
can be also obtained by the arguments in Wang--Yong--Zhang \cite{Wang-Yong-Zhang2021}.
\end{remark}

\subsection{Time-Inconsistency Analysis of Problem (N)}

In this subsection, we shall discuss the time-inconsistency of Problem (N) from the Pontryagin's maximum principle viewpoint.
For simplicity, we consider the case that \rf{Bolza1} holds so that the cost functional reads as \rf{Bolza2}
(of Bolza type, without involving BSVIEs). Also, we suppose that the control domain $U\equiv\dbR^l$ and all involved functions are continuously differentiable. Let $\left(\bar X^{t,\xi}(\cd),\bar u^{t,\xi}(\cd),\bar Y^{t,\xi}(\cd),\bar Z^{t,\xi}(\cd)\right)$ be an optimal 4-tuple (supposing it exists) of Problem (N) on $[t,T]$ with a given initial pair $(t,\xi)\in\sD$, for which we assume to be time-consistent. Then, for any $\t\in(t,T]$,
\bel{J=J}J(\t,\bar X^{t,\xi}(\t);\bar u^{t,\xi}(\cd){\bf1}_{[\t,T]}(\cd))
=\inf_{u(\cd)\in\sU[\t,T]}J(\t,\bar X^{t,\xi}(\t);u(\cd)),\ee
and
\begin{align}
&\big(\bar X^{t,\xi}(s),\,\bar u^{t,\xi}(s),\,\bar Y^{t,\xi}(s),\,\bar Z^{t,\xi}(s)\big)
\nn\\
&\q=\big(\bar X^{\t,\bar X^{t,\xi}(\t)}(s),\,\bar u^{\t,\bar X^{t,\xi}(\t)}(s),\,
\bar Y^{\t,\bar X^{t,\xi}(\t)}(s),\,\bar Z^{\t,\bar X^{t,\xi}(\t)}(s)\big),\q s\in[\t,T],~\as\label{u=u-1}
\end{align}
Now, we denote
\begin{align*}
&\bar b^{t,\xi}_x(s)=b_x(s,\bar X^{t,\xi}(s),\bar u^{t,\xi}(s)),\q&&\bar b_u^{t,\xi}(s)=b_u(s,\bar X^{t,\xi}(s),\bar u^{t,\xi}(s)),\\
&\bar g^{t,\xi}_x(s)=g_x(s,\bar X^{t,\xi}(s),\bar u^{t,\xi}(s),\bar Y^{t,\xi}(s),\bar  Z^{t,\xi}(s)),
\q&&\bar g^{t,\xi}_u(s)=g_u(s,\bar X^{t,\xi}(s),\bar u^{t,\xi}(s),\bar Y^{t,\xi}(s),\bar  Z^{t,\xi}(s)),\\
&\bar g^{t,\xi}_y(s)=g_y(s,\bar X^{t,\xi}(s),\bar u^{t,\xi}(s),\bar Y^{t,\xi}(s),\bar  Z^{t,\xi}(s)),\q
&&\bar g^{t,\xi}_z(s)=g_z(s,\bar X^{t,\xi}(s),\bar u^{t,\xi}(s),\bar Y^{t,\xi}(s),\bar Z^{t,\xi}(s)),\\
&\bar h^{t,\xi}_x(T)=h_x(\bar X^{t,\xi}(T)),\q
&&\bar h^{0,t,\xi}_x(t)=h_x(\bar X^{t,\xi}(T),\bar Y^{t,\xi}(t)),\end{align*}
and $\bar\si^{t,\xi}_x(s)$, $\bar\si^{t,\xi}_u(s)$, $\bar g^{0,t,\xi}_x(s)$, $\bar g^{0,t,\xi}_u(s)$, $\bar g^{0,t,\xi}_y(s)$, $\bar g^{0,t,\xi}_z(s)$, $\bar h^{0,t,\xi}_y(t)$ are defined similarly.
Then by applying the Pontryagin's maximum principle (see \cite{Peng1993,Hu-Ji-Xue2018,Hu-Ji-Xue2019}, for examples), to the optimal 4-tuple on $[t,T]$ and $[\t,T]$, respectively,
we get the following {\it stationarity conditions}:
\begin{align}
&\bar g^{0,t,\xi}_u(s)^\top+\bar g_u^{t,\xi}(s)^\top\cX^{t,\xi}(s)+\bar b^{t,\xi}_u(s)^\top\cY^{t,\xi}(s)+\bar\si_u^{t,\xi}(s)^\top
\cZ^{t,\xi}(s)=0,\qq s\in[t,T],\label{MP-BSDE}\\
&\bar g^{0,t,\xi}_u(s)^\top+\bar g_u^{t,\xi}(s)^\top\cX^{\t,\bar X^{t,\xi}(\t)}(s)+\bar b^{t,\xi}_u(s)^\top\cY^{\t,\bar X^{t,\xi}(\t)}(s)+\bar\si_u^{t,\xi}(s)^\top\cZ^{\t,\bar X^{t,\xi}(\t)}(s)=0,\qq s\in[\t,T],\label{MP-BSDE*}
\end{align}
where $(\cY^{t,\xi}(\cd),\cZ^{t,\xi}(\cd))$ is the co-state process pair of $\bar X^{t,\xi}(\cd)$, and $\cX^{t,\xi}(\cd)$
is the co-state process of $(\bar Y^{t,\xi}(\cd)$, $\bar Z^{t,\xi}(\cd))$, for which the following holds on $[t,T]$ almost surely:
\bel{Optimality-system-BSDE}\left\{\begin{aligned}
\ds d\cY^{t,\xi}(s)&=-\big[\bar g_x^{t,\xi}(s)^\top\cX^{t,\xi}(s)+\bar b_x^{t,\xi}(s)^\top\cY^{t,\xi}(s)
+\bar\si_x^{t,\xi}(s)^\top\cZ^{t,\xi}(s)+\bar g^{0,t,\xi}_x(s)^\top\big]ds+\cZ^{t,\xi}(s)dW(s),\\
\ns\ds d\cX^{t,\xi}(s)&=\big[\bar g_y^{t,\xi}(s)^\top\cX^{t,\xi}(s)+\bar g^{0,t,\xi}_y(s)^\top\big]ds
+\big[\bar g_z^{t,\xi}(s)^{\top}\cX^{t,\xi}(s)+\bar g^{0,t,\xi}_z(s)^\top\big]dW(s),\\
\ns\ds\cY^{t,\xi}(T)&=\bar h_x^{t,\xi}(T)^\top\cX^{t,\xi}(T)+\bar h^{0,t,\xi}_x(t)^\top,\qq
\cX^{t,\xi}(t)=\dbE_t[\bar h^{0,t,\xi}_y(t)^\top],\end{aligned}\right.\ee
and $(\cY^{\t,\bar X^{t,\xi}(\t)}(\cd),\cZ^{\t,\bar X^{t,\xi}(\t)}(\cd))$ is the co-state process pair of $\bar X^{\t,\bar X^{t,\xi}(\t)}(\cd)$, and $\cX^{\t,\bar X^{t,\xi}(\t)}(\cd)$ is the co-state process of $(\bar Y^{\t,\bar X^{t,\xi}(\t)}(\cd)$, $\bar Z^{\t,\bar X^{t,\xi}(\t)}(\cd))$, for which the following holds on $[\t,T]$ almost surely:
\bel{Optimality-system-BSDE-tau}\3n\left\{\2n\begin{aligned}
 d\cY^{\t,\bar X^{t,\xi}(\t)}(s)&=-\big[\bar g_x^{t,\xi}(s)^\top\cX^{\t,\bar X^{t,\xi}(\t)}(s)
\1n+\1n\bar b_x^{t,\xi}(s)^\top\cY^{\t,\bar X^{t,\xi}(\t)}(s)\1n
+\1n\bar\si_x^{t,\xi}(s)^\top\cZ^{\t,\bar X^{t,\xi}(\t)}(s)\1n+\1n\bar g^{0,t,\xi}_x(s)^\top\big]ds\\
&\q+\cZ^{\t,\bar X^{t,\xi}(\t)}(s)dW(s),\\
 d\cX^{\t,\bar X^{t,\xi}(\t)}(s)&=\big[\bar g_y^{t,\xi}(s)^\top\cX^{\t,\bar X^{t,\xi}(\t)}(s)+\bar g^{0,t,\xi}_y(s)^\top\big]ds
+\big[\bar g_z^{t,\xi}(s)^{\top}\cX^{\t,\bar X^{t,\xi}(\t)}(s)+\bar g^{0,t,\xi}_z(s)^\top\big]dW(s),\\
\cY^{\t,\bar X^{t,\xi}(\t)}(T)&=\bar h_x^{t,\xi}(T)^\top\cX^{\t,\bar X^{t,\xi}(\t)}(T)+\bar h^{0,t,\xi}_x(\t)^\top,\qq
\cX^{\t,\bar X^{t,\xi}(\t)}(\t)=\dbE_\t[\bar h^{0,t,\xi}_y(\t)^\top].\end{aligned}\right.\ee
We may state the conclusion as follows.

\begin{proposition}\label{Prop:MP}
If the optimal 4-tuple $(\bar X^{t,\xi}(\cd),\bar u^{t,\xi}(\cd),\bar Y^{t,\xi}(\cd),\bar Z^{t,\xi}(\cd))$ is time-consistent, then \rf{MP-BSDE*} holds for any $\t\in(t,T]$, subject to \rf{Optimality-system-BSDE-tau}.
\end{proposition}

The necessary condition \rf{MP-BSDE*} with $\t\in[t,T]$
can be regarded as a {\it dynamic} version of the Pontryagin's maximum principle.
Interestingly, we can use it to characterize the time-consistency of the optimal controls.

\ms
Now, let us make a careful comparison between \rf{Optimality-system-BSDE} and \rf{Optimality-system-BSDE-tau}.
First of all, these decoupled FBSDEs have exactly the same coefficients.
If we restrict \rf{Optimality-system-BSDE} on $[\t,T]$, then it has the initial condition $\cX^{t,\xi}(\t)$,
and we do not expect the following:
\bel{}
\cX^{t,\xi}(\t)=\dbE_\t[\bar h^{0,t,\xi}_y(\t)^\top]\equiv\dbE_\t\big[h_x(\bar X^{t,\xi}(T),\bar Y^{t,\xi}(t))^\top\big],\qq\t\in[t,T].
\ee
Hence, in general, the following cannot be guaranteed:
$$\ba{ll}
\ns\ds\big(\cX^{\t,\bar X^{t,\xi}(\t)}(s),\,\bar u^{\t,\bar X^{t,\xi}(\t)}(s),\,
\cY^{\t,\bar X^{t,\xi}(\t)}(s),\,\cZ^{\t,\bar X^{t,\xi}(\t)}(s)\big)=\big(\cX^{t,\xi}(s),\,\bar u^{t,\xi}(s),\,\cY^{t,\xi}(s),\,\cZ^{t,\xi}(s)\big),\\
\ns\ds\qq\qq\qq\qq\qq\qq\qq\qq\qq\qq\qq\qq\qq\qq s\in[\t,T],~\as,~\forall\t\in[t,T].\ea$$
Consequently, having \rf{MP-BSDE}, it is too much to request \rf{MP-BSDE*}. From this, we see that Problem (N) is intrinsically time-inconsistent.

\subsection{Equilibrium Strategy and Equilibrium HJB Equation}

Since Problem (N) is time-inconsistent in general, we shall find the equilibrium strategy,
whose definition is given as follows.

\begin{definition}\label{feedback-strategy-In}
A mapping $\Psi:[0,T]\times\dbR^n\to U$ is called a {\it feedback strategy} (of state equation \rf{State}) on $[0,T]$
if for every $(t,\xi)\in\sD$, the following {\it closed-loop} system:
\bel{closed-loop-system}\left\{\2n\ba{ll}
\ds dX(s)=b(s,X(s),\Psi(s,X(s)))ds+\si(s,X(s),\Psi(s,X(s)))dW(s),\q s\in[t,T],\\
\ns\ds dY(s)=-g(s,X(s),\Psi(s,X(s)),Y(s),Z(s))ds+Z(s)dW(s),\q s\in[t,T],\\
\ns\ds X(t)=\xi,\qq Y(T)=h(X(T)),\ea\right.\ee
admits a unique adapted solution $(X(\cd),Y(\cd),Z(\cd))\equiv \big(X(\cd\,;t,\xi,\Psi(\cd\,,\cd)),
Y(\cd\,;t,\xi,\Psi(\cd\,,\cd)),Z(\cd\,;t,\xi,\Psi(\cd\,,\cd))\big)\equiv (X^\Psi(\cd)$, $Y^\Psi(\cd),Z^\Psi(\cd))$
and the {\it outcome} $u^\Psi(\cd)\equiv\Psi(\cd\,,X^\Psi(\cd))$ of $\Psi(\cd\,,\cd)$ belongs to $\sU[t,T]$.
\end{definition}

We now introduce the following definition. 

\begin{definition}\label{def-equilibrium-In} A feedback strategy $\bar\Psi(\cd\,,\cd)$, with $\bar X(\cd)$ being the forward component of the corresponding state process, is called an {\it equilibrium strategy} if
\bel{def-equilibrium1}
\liminf_{\e\to 0^+} {J(t,\bar X(t);\Psi^\e(\cd))-J(t,\bar X(t);\bar\Psi(\cd))\over\e}\ges 0,\ee
for any $t\in[0,T)$ and $u\in L_{\cF_t}^2(\Om;U)$, where
\bel{def-equilibrium2}
\Psi^\e(s)\deq\left\{\2n\ba{ll}
\ds\bar\Psi(s,X^\e(s)),\q&s\in[t+\e,T];\\
\ns\ds u,\q&s\in[t,t+\e),\ea\right.\ee
with $X^\e(\cd)\deq X^{\Psi^\e}(\cd)\equiv X(\cd\,;t,\bar X(t),\Psi^\e(\cd\,,\cd))$ being the forward component of the state process corresponding to $\Psi^\e(\cd\,,\cd)$.

\end{definition}

The intuition behind \autoref{def-equilibrium-In} is similar to that in \cite{Yong2012,Hu-Jin-Zhou2012,Bjork-Khapko-Murgoci2017,Wang-Yong2021}.
At any given time $t$, the controller is playing a game (cooperatively) with all his/her incarnations in the future by minimizing his/her cost functional on $[t,t+\e)$, and knowing that he/she will lose the control of the system beyond $t+\e$. We now briefly list our  main results as follows.
\ms

To find an equilibrium strategy of Problem (N), we introduce the following spaces:
\begin{align*}
&\cA[0,T]:=[0,T]\times\dbR^n\times U\times\dbR^m\times\dbR^{m\times n},\\
&\cA^0[0,T]:=\D^*[0,T]\times\dbR^n\times\dbR^n\times U\times\dbR^m
\times\dbR^{m\times n}\times\dbR\times\dbR^{1\times n}.
\end{align*}
For simplicity, we denote
$$\BTh=(\th,p)\in\dbR^m\times\dbR^{m\times n},\qq\BTh^0=(\th^0,p^0)\in\dbR\times\dbR^{1\times n}.$$
Now, we define the following Hamiltonians:
\begin{align}
&\BH(s,x,u,\BTh,P)\1n=\1n\tr[Pa(s,x,u)]\1n+\1n pb(s,x,u)\1n+\1n g\big(s,x,u,\th,p\si(s,x,u)
\big),\nn\\
&\BH^0(t,s,\ti x,x,u,\BTh,\BTh^0,P^0)\1n=\1n
\tr[P^0a(s,x,u)]\1n+\1n p^0b(s,x,u)\nn\\
&\qq\qq\qq\qq\qq\qq~\1n+\1n g^0(t,s,\ti x,x,u,\th,p\si(s,x,u),\th^0,p^0\si(s,x,u)),\nn\\
&\h\BH^0(t,s,\ti x,x,u,\BTh,P,\BTh^0,q^0,P^0)\1n=\1n{\BH^0(t,s,\ti x,x,u,\BTh,\BTh^0,P^0)}
+q^0\BH(s,x,u,\BTh,P),\nn\\
&\qq\qq\qq(t,s,\ti x,x,u,\BTh,\BTh^0)\in\cA^0[0,T],~P\in\big[\dbS^n\big]^m,~q^0\in
\dbR^{1\times m},~P^0\in\dbS^n,\label{HH}
\end{align}
where $a(\cd)$ is defined by
$$
a(s,x,u)={1\over 2}\si(s,x,u)\si(s,x,u)^\top,\q (s,x,u)\in[0,T]\times\dbR^n\times U,
$$
and, with $P=(P^1,P^2,\cds,P^m)^\top\in[\dbS^n]^m$,
$$\tr\big[Pa(s,x,u)\big]=\begin{pmatrix}\tr\big[P^1a(s,x,u)\big]\\
\tr\big[P^2a(s,x,u)\big]\\ \vdots\\ \tr\big[P^ma(s,x,u)\big]\end{pmatrix}.$$
In what follows, we will use the following hypothesis.

\begin{taggedassumption}{(H3)}\label{ass:H3}
Suppose that there exists a unique mapping $\psi(\cd)$ such that
\begin{align}
\h\BH^0(t,s,\ti x,x,\psi(t,s,\ti x,x,\BTh,P,\BTh^0,q^0,P^0),\BTh,P,\BTh^0,q^0,P^0)
=\inf_{u\in U}\h\BH^0(t,s,\ti x,x,u,\BTh,P,\BTh^0,q^0,P^0),\nn\\
\forall(t,s,\ti x,x,u,\BTh,\BTh^0)\in\cA^0[0,T],~P\in\big[\dbS^n\big]^m,~q^0\in
\dbR^{1\times m},~P^0\in\dbS^n.\label{local-opt}
\end{align}
Moreover, we suppose that $\psi(\cd)$ is smooth enough with bounded derivatives.
\end{taggedassumption}

We now introduce the following {\it equilibrium HJB equation}:
\bel{HJB-BSVIE-Main}\left\{\2n\ba{ll}
\ds\Th_s(s,x)+\BH\big(s,x,\bar\Psi(s,x),\Th(s,x),\Th_x(s,x),\Th_{xx}(s,x)\big)
=0,\q(s,x)\in[0,T]\times\dbR^n,\\
\ns\ds\Th^0_s(t,s,\ti x,x,y)\1n+\1n\BH^0\big(t,s,\ti x,x,\bar\Psi(s,x),\Th(s,x),\Th_x(s,x),\Th^0(s,s,x,x,\Th(s,x)),\\
\ns\ds\qq \Th_x^0(t,s,\ti x,x,y),\Th^0_{xx}(t,s,\ti x,x,y)\big)=0,
\q(t,s,\ti x,x,y)\in\D^*[0,T]\times\dbR^n\times\dbR^n\times\dbR^m,\\
\ns\ds\Th(T,x)=h(x),\q x\in\dbR^n,\\
\ns\ds\Th^0(t,T,\ti x,x,y)=h^0(t,\ti x,x,y),\q(t,\ti x,x,y)\in[0,T]\times\dbR^n\times\dbR^n\times\dbR^m,\ea\right.\ee
where
\begin{align}
\bar\Psi(s,x)&=\psi\big(s,s,x,x,\Th(s,x),\Th_x(s,x),\Th_{xx}(s,x),\Th^0(s,s,x,x,
\Th(s,x)),\Th^0_x(s,s,x,x,\Th(s,x)),\nn\\
&\qq\q\Th^0_y(s,s,x,x,\Th(s,x)),\Th^0_{xx}(s,s,x,x,\Th(s,x))\big),\qq(s,x)\in[0,T]\times\dbR^n.\label{bar Psi}
\end{align}
We have the following result.

\bt{theorem-equilibrium-BSVIE-Main}
Let $\bar\Psi:[0,T]\times\dbR^n\to U$ be defined by \rf{bar Psi},
with $(\Th(\cd),\Th^0(\cd))$ being the classical solution to the equilibrium HJB equation \rf{HJB-BSVIE-Main}.
Let $\bar\Psi(\cd,\cd)$ be a feedback strategy.
Then it is an equilibrium strategy of {\rm Problem (N)}.
\et

\begin{remark}
If  the cost functional reads as \rf{Bolza2}, then we have
\begin{align}
&\BH^0(s,x,u,\BTh,\BTh^0,P^0)\1n=\1n
\tr[P^0a(s,x,u)]\1n+\1n p^0b(s,x,u)\1n+\1n g^0(s,x,u,\th,p\si(s,x,u)),\nn\\
&\h\BH^0(s,x,u,\BTh,P,\BTh^0,q^0,P^0)\1n=\1n
\tr[P^0a(s,x,u)]\1n+\1n p^0b(s,x,u)\1n+\1n g^0(s,x,u,\th,p\si(s,x,u))\nn\\
&\qq\qq+q^0\(\tr[Pa(s,u,x)]+pb(s,x,u)+g(s,x,u,\th,p\si(s,x,u))\),\nn\\
&\qq\qq\q\forall(s,x,u,\BTh,P,\BTh^0,q^0,P^0)\in\cA[0,T]\1n\times\1n[\dbS^n]^m\times\dbR\times\dbR^{1\times n}\times\dbR^{1\times m}\times\dbS^n.\label{h-H^0}
\end{align}
\end{remark}

\begin{remark}\label{EHJB-Prob}
\autoref{theorem-equilibrium-BSVIE-Main} is a verification theorem for Problem (N),
whose proof is given in Section \ref{sec:Verification}.
Taking the equilibrium strategy $\bar\Psi(\cd,\cd)$ in \rf{State} and \rf{cost-BSVIE1},
we get the following equilibrium system on $[0,T]$:
$$\left\{\2n\ba{ll}
\ds d\bar X(s)=b(s,\bar X(s),\bar\Psi(s,\bar X(s)))ds+\si(s,\bar X(s),\bar\Psi(s,\bar X(s)))dW(s),\\
\ns\ds d\bar Y(s)=-g(s,\bar X(s),\bar\Psi(s,
\bar X(s)),\bar Y(s),\bar Z(s))ds+\bar Z(s)dW(s),\\
\ns\ds\bar X(0)=\xi,\qq\bar Y(T)=h(\bar X(T)),\ea\right.
$$
and
\begin{align*}
\bar Y^0(r)&=h^0(r,\bar X(r),\bar X(T),\bar Y(r))
\1n+\2n\int_r^T\3n g^0\big(r,s,\bar X(r),\bar X(s),\bar\Psi(s,\bar X(s)),\bar Y(s),\bar Z(s),\bar Y^0(s),\bar Z^0(r,s)\big)ds\\
&\q-\int_r^T\bar Z^0(r,s)dW(s),\q r\in[0,T].
\end{align*}
Then by \autoref{Prop:FK-BSVIE}, we have the following representation formula:
\begin{align*}
&\bar Y(r)=\Th(r,\bar X(r)),\q \bar Z(r)=\Th_x(r,\bar X(r))\si(r,\bar X(r),\bar\Psi(r,\bar X(r))),\\
&\bar Y^0(r)=\Th^0\big(r,r,\bar X(r),\bar X(r),\Th(r,\bar X(r))\big),\\
&\bar Z^0(r,s)=\Th^0_x\big(r,s,\bar X(r),\bar X(s),\Th(r,\bar X(r))\big)\si(s,\bar X(s),\bar\Psi(s,\bar X(s))),
\end{align*}
provided the equilibrium HJB equation \rf{HJB-BSVIE-Main} admits a classical solution $(\Th(\cd),\Th^0(\cd))$.
Thus, the form  of the equilibrium HJB equation \rf{HJB-BSVIE-Main} is very natural, though it seems a little bit complicated.
Using the local optimality condition \rf{local-opt}, the  equilibrium strategy value $\bar\Psi(s,x)$ is determined by
$\Th(s,x)$ and the diagonal value $\Th^0(s,s,x,x,\Th(s,x))$.
\end{remark}

\subsection{Well-posedness of the Equilibrium HJB Equation}

In this subsection, we will present the well-posedness of equation \rf{HJB-BSVIE-Main} to some extent.
Note that \rf{HJB-BSVIE-Main} is a coupled system of fully nonlinear  parabolic PDEs with a non-local feature,
whose well-posedness  is a very challenging problem.
Indeed, even for the equilibrium HJB equation associated with the time-inconsistent problems for SDEs (see Yong \cite{Yong2012}),
the well-posedness is still widely open, except in the small time case (see Lei--Pun \cite{Lei-Pun2021}).
For the small time case, one can construct  a contraction mapping
in a Banach space depending on the terminal conditions and does not need to establish a prior estimate,
which is exactly the main difficulty in establishing the well-posedness.

\ms

We now assume that
\bel{def-si}\si(s,x,u)=\si(s,x),\q (s,x,u)\in[0,T]\times\dbR\times U.\ee
In this case, we denote
\begin{align*}
&\BH(s,x,u,\BTh)=pb(s,x,u)+g(s,x,u,\th,p\si(s,x)),\qq
(s,x,u,\BTh)\in\cA[0,T],\\
&\BH^0(t,s,\ti x,x,u,\BTh,\BTh^0)=p^0b(s,x,u)\1n+\1n g^0(t,s,\ti x,x,u,\th,p\si(s,x),\th^0,p^0\si(s,x)),\\
&\h\BH^0(t,s,\ti x,x,u,\BTh,\BTh^0,q^0)=\BH^0(t,s,\ti x,x,u,\BTh,\BTh^0)+q^0\BH(s,x,u,\BTh),\\
&\qq\qq\qq\qq\qq\qq\qq\q(t,s,\ti x,x,u,\BTh,\BTh^0)\in\cA^0[0,T],~q^0\in\dbR^{1\times m}.
\end{align*}
Then the mapping $\psi(\cd)$ determined by \ref{ass:H3} can be determined by the following:
\begin{align}
\h\BH^0(t,s,\ti x,x,\psi(t,s,\ti x,x,\BTh,\BTh^0,q^0),\BTh,\BTh^0,q^0)
=\inf_{u\in U}\h\BH^0(t,s,\ti x,x,u,\BTh,\BTh^0,q^0),\nn\\
\forall(t,s,\ti x,x,u,\BTh,\BTh^0)\in\cA^0[0,T],~q^0\in\dbR^{1\times m}.
\label{def-psi-si}
\end{align}
Namely, in the current case, $\psi(\cd)$ is independent of $P$ and $P^0$. Then
\rf{HJB-BSVIE-Main} is reduced to the following system of  semilinear PDEs:
\bel{HJB-BSVIE-si}\3n\3n\left\{\2n\ba{ll}
\ns\ds\Th_s(s,x)\1n+\1n\tr[\Th_{xx}(s,x)a(s,x)]\1n+\1n\Th_x(s,x)b(s,x,\bar\Psi(s,x))\\
\ns\ds\q\1n+\1n g(s,x,\bar\Psi(s,x),\Th(s,x),\Th_x(s,x)\si(s,x))\1n=\1n0,\\
\ns\ds\Th^0_s(t,s,\ti x,x,y)+\tr[\Th^0_{xx}(t,s,\ti x,x,y)a(s,x)]+\Th^0_x(t,s,\ti x,x,y)b(s,x,\bar\Psi(s,x))\\
\ns\ds\q+g^0\big(t,s,\ti x,x,\bar\Psi(s,x),\Th(s,x),\Th_x(s,x)\si(s,x),\Th^0(s,s,x,x,\Th(s,x)),
\Th^0_x(t,s,\ti x,x,y)\si(s,x)\big)=0,\\
\ns\ds\Th(T,x)=h(x),\q\Th^0(t,T,\ti x,x,y)=h^0(t,\ti x,x,y),\ea\right.
\ee
with
\begin{align}
\bar\Psi\1n(s,x)\1n=\1n\psi\big(s,s,x,x,\1n\Th(s,x),\1n\Th_x(s,x),\1n\Th^0(s,s,x,x,\1n\Th(s,x)),\1n \Th^0_x(s,s,x,x,\1n\Th(s,x)),\Th^0_y(s,s,x,x,\1n\Th(s,x))\big),\nn\\
 (s,x)\in[0,T]\times\dbR^n.\label{barPsi}
\end{align}
Now, we denote
\begin{align*}
&\ti b\big(s,x,\Th(s,x),\Th_x(s,x),\Th^0(s,s,x,x,\Th(s,x)),\Th^0_x(s,s,x,x,\Th(s,x)),
\Th^0_y(s,s,x,x,\Th(s,x))\big)\\
&\q:= b\big(s,x,\psi\big(s,s,x,x,\Th(s,x),\Th_x(s,x),\Th^0(s,s,x,x,\Th(s,x)),
\Th^0_x(s,s,x,x,\Th(s,x)),\\
&\qq\qq\Th^0_y(s,s,x,x,\Th(s,x))\big)\big),\\
&\ti g\big(s,x,\Th(s,x),\Th_x(s,x),\Th^0(s,s,x,x,\Th(s,x)),\Th^0_x(s,s,x,x,\Th(s,x)),
\Th^0_y(s,s,x,x,\Th(s,x))\big)\\
&\q:= g\big(s,x,\psi\big(s,s,x,x,\Th(s,x),\Th_x(s,x),\Th^0(s,s,x,x,\Th(s,x)),
\Th^0_x(s,s,x,x,\Th(s,x)),\\
&\qq\qq\Th^0_y(s,s,x,x,\Th(s,x))\big),\Th(s,x),\Th_x(s,x)\si(s,x)\big),\\
&\ti g^0\big(t,s,\ti x,x,\Th(s,x),\Th_x(s,x),\Th^0(s,s,x,x,\Th(s,x)),\Th^0_x(s,s,x,x,\Th(s,x)),\\
&\qq\Th^0_y(s,s,x,x,\Th(s,x)),\Th^0_x(t,s,\ti x,x,y)\big)\\
&\q:= g^0\big(t,s,\ti x,x,\psi\big(s,s,x,x,\Th(s,x),\Th_x(s,x),\Th^0(s,s,x,x,\Th(s,x)),
\Th^0_x(s,s,x,x,\Th(s,x)),\\
&\qq\qq\Th^0_y(s,s,x,x,\Th(s,x))\big),\Th(s,x),\Th_x(s,x)\si(s,x),\Th^0_x(t,s,\ti x,x,y)\si(s,x)\big).
\end{align*}
Because of the above dependence, we may write the above \rf{HJB-BSVIE-si} as follows:
\bel{HJB-si-ti}\3n\3n\left\{\2n\ba{ll}
\ds\Th_s(s,x)+\tr[\Th_{xx}(s,x)a(s,x)]+\Th_x(s,x)\\
\ns\ds\q \cd\ti b\big(s,x,\Th(s,x),\Th_x(s,x),\Th^0(s,s,x,x,\Th(s,x)),\Th^0_x(s,s,x,x,\Th(s,x)),
\Th^0_y(s,s,x,x,\Th(s,x))\big)\\
\ns\ds\q+\ti g\big(s,x,\Th(s,x),\Th_x(s,x),\Th^0(s,s,x,x,\Th(s,x)),\Th^0_x(s,s,x,x,\Th(s,x)),\Th^0_y(s,s,x,x,\Th(s,x))\big)=0,\\
\ns\ds\Th^0_s(t,s,\ti x,x,y)+\tr[\Th^0_{xx}(t,s,\ti x,x,y)a(s,x)]+\Th^0_x(t, s,\ti x, x, y)\\
\ns\ds\q\cd\ti b\big(s, x,\Th(s, x),\Th_x(s, x),\Th^0(s, s, x, x,\Th(s, x)),
\Th^0_x(s, s,x, x,\Th(s, x)),\Th^0_y(s, s, x, x,\Th(s, x))\big)\\
\ns\ds\q+\ti g^0\big(t,s,\ti x,x,\Th(s,x),\Th_x(s,x),\Th^0(s,s,x,x,\Th(s,x)),\Th^0_x(s,s,x,x,\Th(s,x)),\\
\ns\ds\qq\Th^0_y(s,s,x,x,\Th(s,x)),\Th^0_x(t,s,\ti x,x,y)\big)=0,\\
\ns\ds\Th(T,x)=h(x),\qq\Th^0(t,T,\ti x,x,y)=h^0(t,\ti x,x,y).\ea\right.\ee
Although the above looks complicated, it actually has a usual HJB equation form, which can be abbreviated as follows, if we suppress some lengthy arguments:
\bel{HJB-si-ti*}\3n\3n\left\{\2n\ba{ll}
\ds\Th_s(s,x)+\tr[\Th_{xx}(s,x)a(s,x)]+\Th_x(s,x)\ti b(s,x)+\ti g(s,x)=0,\\
\ns\ds\Th^0_s(t,s,\ti x,x,y)+\tr[\Th^0_{xx}(t,s,\ti x,x,y)a(s,x)]+\Th^0_x(t,s,\ti x,x,y)\ti b(s,x)+\ti g^0(t,s,\ti x,x,y)=0,\\
\ns\ds\Th(T,x)=h(x),\qq\Th^0(t,T,\ti x,x,y)=h^0(t,\ti x,x,y).\ea\right.\ee
For the above system, we introduce the following assumption.

\begin{taggedassumption}{(H4)}\label{ass:H4}
The mappings
$$\left\{\ba{ll}
\ds(s,x)\mapsto a(s,x),\q x\mapsto h(x),\q(t,\ti x,x,y)\mapsto h^0(t,\ti x,x,y),\\
\ns\ds(s,x,\BTh,\BTh^0,q^0)\mapsto\ti b(s,x,\BTh,\BTh^0,q^0),\q (s,x,\BTh,\BTh^0,q^0)\mapsto \ti g(s,x,\BTh,\BTh^0,q^0),\\
\ns\ds(t,s,\ti x,x,\BTh,\BTh^0,q^0,\h p^0)\mapsto\ti g^0(t,s,\ti x,x,\BTh,\BTh^0,q^0,\h p^0)\ea\right.$$
are bounded, have all required differentiability with bounded derivatives.
Moreover, there exist two constants $\l_0,\l_1>0$ such that
$$\l_0 I\les a(t,x)\les \l_1 I,\q \forall(t,x)\in[0,T]\times\dbR^n.$$
\end{taggedassumption}

\begin{theorem}\label{thm:HJB}
Let {\rm\ref{ass:H4}} hold. Then the equilibrium HJB equation \rf{HJB-si-ti} admits a unique classical solution.
\end{theorem}

The proof of \autoref{thm:HJB} is technical and lengthy, which will be given in Section 7. Note that \rf{HJB-si-ti} contains the diagonal term $\Th^0_y(s,s,x,x,\Th(s,x))$.
To our best knowledge, it is the first time that such an equilibrium HJB equation is derived.

\section{Comparisons with the Existing Results}\label{sec:comparison}

\subsection{Comparison with Peng \cite{Peng19971}}

In \cite{Peng19971}, Peng established the dynamic programming principle (DPP, for short) for the optimal control problem with the state equation
\bel{state-Peng}\left\{\2n\ba{ll}
\ds dX(s)=b(s,X(s),u(s))ds+\si(s,X(s),u(s))dW(s),\q s\in[t,T],\\
\ns\ds dY(s)=-g(s,X(s),u(s),Y(s),Z(s))ds+Z(s)dW(s),\q s\in[t,T],\\
\ns\ds X(t)=\xi,\q Y(T)=h(X(T)),\ea\right.\ee
and the cost functional
\bel{Cost-Peng}J(t,\xi;u(\cd))=Y(t),\ee
with the backward process $Y(\cd)$ being one-dimensional. Such a problem is denoted by Problem (R).
It turns out that in this case, the optimal control problem is time-consistent.
The following provides a time-consistency analysis of  Problem (R) from the viewpoint of Pontryagin's maximum principle.

\begin{proposition}\label{prop:TI-R}
Suppose that $\bar u^{t,\xi}(\cd)$ is an optimal control of {\rm Problem (R)} with the initial pair $(t,\xi)\in\sD$.
Then $\bar u^{t,\xi}(\cd)$ satisfies the necessary condition \rf{MP-BSDE*} for any $\t\in(t,T]$.
\end{proposition}

\begin{proof}
By \rf{MP-BSDE} and the fact that  $h^0(x,y)=y;(x,y)\in\dbR^n\times\dbR$ and $g^0(\cd)\equiv0$, we have
\bel{MP-BSDE1}
\bar b_u^{t,\xi}(s)^\top\h Y^{t,\xi}(s)+\bar\si_u^{t,\xi}(s)^\top\h Z^{t,\xi}(s)+\bar g^{t,\xi}_u(s)^\top\h X^{t,\xi}(s)=0,\q s\in[t,T],
\ee
with
\bel{Optimality-system-BSDE1}\left\{\2n\ba{ll}
\ds d\h Y^{t,\xi}(s)=-\big[\bar g^{t,\xi}_x(s)^\top\h X^{t,\xi}(s)+\bar b^{t,\xi}_x(s)^\top\h Y^{t,\xi}(s)+\bar\si^{t,\xi}_x(s)^\top\h Z^{t,\xi}(s)\big]ds+\h Z^{t,\xi}(s)dW(s),\\
\ns\ds d\h X^{t,\xi}(s)=\bar g^{t,\xi}_y(s)^{\top}\h X^{t,\xi}(s)ds+\bar g^{t,\xi}_z(s)^{\top}\h X^{t,\xi}(s)dW(s),\\
\ns\ds\h Y^{t,\xi}(T)=\bar h^{t,\xi}_x(T)^\top\h X^{t,\xi}(T),\q \h X^{t,\xi}(t)=1.\ea
\right.\ee
Clearly, $\h X^{t,\xi}(s)$ is non-zero for any $s\in[t,T]$. For any $\t\in(t,T]$, define
\bel{hat-X-0t}\left\{
\begin{aligned}
\h X^{\t,\bar X^{t,\xi}(\t)}(s)=\h X^{t,\xi}(s)[\h X^{t,\xi}(\t)]^{-1},\q s\in[\t,T],\\
\h Y^{\t,\bar X^{t,\xi}(\t)}(s)=\h Y^{t,\xi}(s)[\h X^{t,\xi}(\t)]^{-1},\q s\in[\t,T],\\
\h  Z^{\t,\bar X^{t,\xi}(\t)}(s)=\h Z^{t,\xi}(s)[\h X^{t,\xi}(\t)]^{-1},\q s\in[\t,T].
\end{aligned}\right.\ee
From \rf{Optimality-system-BSDE1}, it is easily seen that
$(\h X^{\t,\bar X^{t,\xi}(\t)}(\cd),\h Y^{\t,\bar X^{t,\xi}(\t)}(\cd),\h Z^{\t,\bar X^{t,\xi}(\t)}(\cd))$
is the unique solution to the following FBSDE on $[\t,T]$:
$$\left\{\begin{aligned}
d\h Y^{\t,\bar X^{t,\xi}(\t)}(s)&=-\big[\bar g^{t,\xi}_x(s)^\top\h X^{\t,\bar X^{t,\xi}(\t)}(s)
+\bar b^{t,\xi}_x(s)^\top\h Y^{\t,\bar X^{t,\xi}(\t)}(s)+\bar\si^{t,\xi}_x(s)^\top\h Z^{\t,\bar X^{t,\xi}(\t)}(s)\big]ds\\
&\q+\h Z^{\t,\bar X^{t,\xi}(\t)}(s)dW(s),\\
d\h X^{\t,\bar X^{t,\xi}(\t)}(s)&=\bar g^{t,\xi}_y(s)^{\top}\h X^{\t,\bar X^{t,\xi}(\t)}(s)ds
+\bar g^{t,\xi}_z(s)^{\top}\h X^{\t,\bar X^{t,\xi}(\t)}(s)dW(s),\\
\h Y^{\t,\bar X^{t,\xi}(\t)}(T)&=\bar h^{t,\xi}_x(T)^\top\h X^{\t,\bar X^{t,\xi}(\t)}(T),\q \h X^{\t,\bar X^{t,\xi}(\t)}(\t)=1.
\end{aligned}\right.$$
Then by \rf{MP-BSDE1}, we get
\begin{align*}
&\bar g^{t,\xi}_u(s)^\top\h X^{\t,\bar X^{t,\xi}(\t)}(s)+\bar b^{t,\xi}_u(s)^\top\h Y^{\t,\bar X^{t,\xi}(\t)}(s)
+\bar\si^{t,\xi}_u(s)^\top\h Z^{\t,\bar X^{t,\xi}(\t)}(s)\\
&\q=\big[\bar g^{t,\xi}_u(s)^\top\h X^{t,\xi}(s)+\bar b^{t,\xi}_u(s)^\top\h Y^{t,\xi}(s)
+\bar\si^{t,\xi}_u(s)^\top\h Z^{t,\xi}(s)\big][\h X^{t,\xi}(\t)]^{-1}=0,\q s\in[\t,T],~ \t\in[t,T],
\end{align*}
which implies the necessary condition \rf{MP-BSDE*} holds.
\end{proof}

We remark that Problem (R) is a very special case of Problem (N)
and the relationship \rf{hat-X-0t} does not hold in general.
When $Y(\cd)$ is multi-dimensional, the problem is time-inconsistent in general,
even if $J(t,\xi;u(\cd))$ is a linear function of $Y(t)$. Here is a simple example.

\begin{example}\label{example-y}
Consider the  (degenerate) FBSDE state equation
$$\left\{\2n\ba{ll}
\ds\dot X(s)=0,\q s\in[t,T],\\
\ns\ds\dot{Y}_1(s)=u(s),\q s\in[t,T],\\
\ns\ds\dot{Y}_2(s)=-Y_1(s)-u(s)-|u(s)|^2,\q s\in[t,T],\\
\ns\ds X(t)=x,\q Y_1(T)=0,\q Y_2(T)=0,\ea\right.$$
with the cost functional
$$J(t,x;u(\cd))=Y_2(t).$$
Then
\begin{align*}
J(t,x;u(\cd))&=\int_t^T [Y_1(s)+u(s)+|u(s)|^2]ds=\int_t^T\[u(s)+|u(s)|^2 -\int_s^T u(r)dr\]ds\\
&=\int_t^T\[(1+t-s)u(s)+|u(s)|^2 \]ds.
\end{align*}
Thus, the unique optimal control $\bar u(\cd\,;t,x)$ for initial pair $(t,x)$ is given by
$$\bar u(s)\equiv\bar u(s;t,x)={s-t-1\over 2},\qq s\in[t,T].$$
And for any $0\les t<\t\les T$, the optimal control at $(\t,\bar X(\t))\equiv(\t,x)$ is given by
$$\ti u(s)\equiv\ti u(s;\t,\bar X(\t))={s-\t-1\over2},\qq s\in[\t,T].$$
Clearly,
$$
\bar u(s)\neq \ti u(s),\q s\in[\t,T],
$$
which implies that the problem is time-inconsistent.
\end{example}

\ms
We now show that in the case of \rf{state-Peng}--\rf{Cost-Peng} (with $m=1$),
the equilibrium HJB equation \rf{HJB-BSVIE-Main} is reduced to the classical HJB equation
associated with  recursive stochastic optimal control problems.
In fact, the associated equilibrium HJB equation is given as follows
\bel{HJB-BSVIE-Peng}\left\{\2n\ba{ll}
\ds\Th_s(s,x)+\Th_x(s,x)b(s,x,\bar\Psi(s,x))+\tr[\Th_{xx}(s,x)a(s,x,
\bar\Psi(s,x))]\\
\ns\ds\q+g(s,x,\bar\Psi(s,x),\Th(s,x),\Th_x(s,x)\si(s,x,\bar\Psi(s,x)))=0,\\
\ns\ds \Th^0_s(t,s,\ti x,x,y)+\Th^0_x(t,s,\ti x,x,y)b(s,x,\bar\Psi(s,x))+\tr[\Th^0_{xx}
(t,s,\ti x,x,y)a(s,x,\bar\Psi(s,x))]=0,\\
\ns\ds\Th(T,x)=h(x),\qq\Th^0(t,y,z,T,x)=y,\ea\right.\ee
where $\bar\Psi(\cd,\cd)$ satisfies the local optimality condition \rf{def-psi-si}.
Clearly, $\Th^0(\cd)\equiv y$ is a classical solution to the second PDE in \rf{HJB-BSVIE-Peng}.
Thus, the local optimality condition \rf{def-psi-si} can be rewritten as follows:
For any $(s,x)\in[0,T]\times\dbR^n$,
\begin{align*}
&\Th_x(s,x)b(s,x,\bar\Psi(s,x))\1n+\1n\tr[\Th_{xx}(s,x)a(s,x,\bar\Psi(s,x))]
\1n+\1ng(s,x,\bar\Psi(s,x),\Th(s,x),\Th_x(s,x)\si(s,x,\bar\Psi(s,x)))\\
&\q=\inf_{u\in U}\Big\{\Th_x(s,x)b(s,x,u)+\tr[\Th_{xx}(s,x)a(s,x,u)]+g
(s,x,u,\Th(s,x),\Th_x(s,x)\si(s,x,u))\Big\}.
\end{align*}
Then the equilibrium value function can be given by
$$\Th^0(s,s,x,x,\Th(s,x))=\Th(s,x),\qq(s,x)\in[0,T]\times\dbR^n,$$
with $\Th(\cd\,,\cd)$ being uniquely determined by
$$\left\{\2n\ba{ll}
\ds\Th_s(s,x)+\inf_{u\in U}\Big\{
\Th_x(s,x)b(s,x,u)+\tr[\Th_{xx}(s,x)a(s,x,u)],\\
\ns\ds\qq\qq\qq+g(s,x,u,\Th(s,x),\Th_x(s,x)\si(s,x,u))\Big\}=0,\q (s,x)\in[0,T]\times\dbR^n,\\
\ns\ds\Th(T,x)=h(x),\q x\in\dbR^n,\ea\right.$$
which is exactly the classical HJB equation derived by Peng \cite{Peng19971}.

\subsection{Comparison with  Yong \cite{Yong2012,Yong2014},  Wei--Yong--Yu \cite{Wei-Yong-Yu2017}, and Wang--Yong \cite{Wang-Yong2021}}

As an equilibrium recursive version of  \cite{Yong2012,Yong2014,Wei-Yong-Yu2017},
Wang--Yong \cite{Wang-Yong2021} considered the optimal control problems with the state equation
\bel{state-Yong}\left\{\2n\ba{ll}
\ds dX(s)=b(s,X(s),u(s))ds+\si(s,X(s),u(s))dW(s),\q s\in[t,T],\\
\ns\ds X(t)=\xi,\ea\right.\ee
and the cost functional
\bel{cost-BSVIE-Yong}J(t,\xi;u(\cd))=Y^0(t),\ee
where $Y^0(\cd)$ is uniquely determined by the following BSVIE:
\bel{cost-BSVIE1-Yong}Y^0(r)=h^0(r,X(T))
+\int_r^Tg^0(r,s,X(s),u(s),Y^0(s),Z^0(r,s))ds-\int_r^TZ^0(r,s)dW(s),\q r\in[t,T].\ee
Then by comparing the above with \rf{cost-BSVIE0} and \rf{cost-BSVIE1},
we see that in our problem, the cost functional can additionally depend on the initial state $X(r)$
and the backward process $(Y(\cd),Z(\cd))$.
If the diffusion term of \rf{state-Yong} does not depend on the control $u(\cd)$,
the associated equilibrium HJB equation admits the followng form:
\bel{HJB-Yong}\left\{\2n\ba{ll}
\ds\Th^0_s(t,s,x)+\tr[\Th^0_{xx}(t,s,x)a(s,x)]+\Th^0_x(t,s,x)\ti b\big(s,x,\Th^0(s,s,x),\Th^0_x(s,s,x)\big)\\
\ns\ds\q+\ti g^0\big(t,s,x,\Th^0(s,s,x),\Th^0_x(s,s,x),\Th^0_x(t,s,x)\big)=0,\q(t,s)\in\D^*[0,T],~x\in\dbR^n,\\
\ns\ds\Th^0(t,T,x)=h^0(t,x),\q(t,x)\in[0,T]\times\dbR^n.\ea\right.\ee
Compared with the equilibrium HJB equation \rf{HJB-Yong} derived in \cite{Yong2012,Wei-Yong-Yu2017,Wang-Yong2021},
\rf{HJB-si-ti} has the following new features:

\ms

$\bullet$
Equilibrium HJB equation \rf{HJB-si-ti} is a coupled system of parabolic PDEs.
It is interesting  that the last PDE in \rf{HJB-si-ti} is coupled with the first $m$ equations
not only through the appearance of $\Th(s,x)$ and $\Th_x(s,x)$ in the function $g^0(\cd)$,
but also through the non-local form $\Th^0(s,s,x,x,\Th(s,x))$,
$\Th_x^0(s,s,x,x,\Th(s,x))$, $\Th_y^0(s,s,x,x,\Th(s,x))$ of the unknown function $\Th^0(\cd)$.

\ms

$\bullet$
Equilibrium HJB equation \rf{HJB-si-ti} depends on the partial derivative $\Th^0_y(t,s,\ti x,x,y)$
along the ``diagonal'' points $(s,s,x,x,\Th(s,x))$,
by which we see that the backward controlled equation has a significant influence on deducing the equilibrium HJB equation.
To be more clear, we take a look at this from a probabilistic viewpoint. By the It\^{o}'s formula,  the stochastic system associated with \rf{HJB-si-ti} reads
$$\left\{\2n\ba{ll}
\ds X(t)=x\1n+\2n\int_0^t\1n\ti b(s,X(s), Y(s),Z(s)\si(s,X(s))^{-1},Y^0(s),Z^0(s,s)\si(s,X(s))^{-1},\h Y^0(s))ds\\
\ns\ds\qq\q+\2n\int_0^t\3n\si(s,X(s))dW(s),\\
\ns\ds Y(t)\1n=\1n h(X(T))\1n+\2n\int_t^T\3n\ti g(s,X(s), Y(s),Z(s)\si(s,X(s))^{-1},Y^0(s),Z^0(s,s)\si(s,X(s))^{-1}\2n,\h Y^0(s))ds\\
\ds\qq\q-\2n\int_t^T\3n Z(s)dW(s),\\
\ns\ds Y^0(t)=h^0(t,X(t),X(T),Y(t))+\int_t^T\wt g^0\big(t,s,X(t),X(s), Y(s),Z(s)\si(s,X(s))^{-1},Y^0(s),\\
\ns\ds\qq\qq Z^0(s,s)\si(s,X(s))^{-1},\h Y^0(s),Z^0(t,s)\si(s,X(s))^{-1}\big)ds-\int_t^TZ^0(t,s)dW(s),\\
\ns\ds\h Y^0(t)=h^0_y(t,X(t),X(T),Y(t))-\int_t^T\h Z^0(t,s)dW(s)+\int_t^T\wt g^0_{p^0}\big(t,s,X(t),X(s), Y(s),Z(s)\si(s,X(s))^{-1},\\
\ns\ds\qq\qq\qq Y^0(s), Z^0(s,s)\si(s,X(s))^{-1},\h Y^0(s),Z^0(t,s)\si(s,X(s))^{-1}\big) Z^0(t,s)\si(s,X(s))^{-1}ds.\ea\right.$$
Compared with \cite[Theorem 5.1]{Wang-Yong2021}, the first backward equation and the third backward equation are new.
The appearance of the first backward equation is natural, because the state system \rf{State} is a controlled FBSDE.
However, the appearance of the third backward equation is surprising.
Indeed, the process $\h Y^0(\cd)$ is introduced for providing a probabilistic representation for the term $\Th^0_y(s,\Th(s,x),x,s,x)$,
which comes from the local optimality condition of the Hamiltonian \rf{def-psi-si}.

\ms

$\bullet$
In \rf{HJB-si-ti}, there are three diagonal/non-local terms $\Th^0(s,s,x,x,\Th(s,x))$, $\Th^0_x(s,s,x,x,\Th(s,x))$ and $\Th^0_y(s,s,x,x,\Th(s,x))$. Equilibrium HJB equation \rf{HJB-si-ti} is non-local not only in the time variables $(t,s)$ (as in \cite{Yong2012,Yong2014,Wei-Yong-Yu2017,Wang-Yong2021}) but also in the space variables $(\ti x,x,y)$.
More interestingly, at the variable $y$, the diagonal/non-local terms are obtained by setting $y=\Th(s,x)$,
instead of by letting $y=x$. Thus, the non-locality of  \rf{HJB-si-ti} is much more complicated than the ones derived in \cite{Yong2012,Yong2014,Wei-Yong-Yu2017,Wang-Yong2021}.
We need to make some very careful analysis  in establishing the well-posedness.

\textcolor[rgb]{0.00,0.00,1.00}{}\subsection{Comparison with Bj\"{o}rk--Khapko--Murgoci \cite{Bjork-Khapko-Murgoci2017}}

In \cite{Bjork-Khapko-Murgoci2017}, Bj\"{o}rk, Khapko, and Murgoci
considered the optimal control problems with the state equation
$$\left\{\2n\ba{ll}
\ds dX(s)=b(s,X(s),u(s))ds+\si(s,X(s),u(s))dW(s),\\
\ns\ds X(t)=\xi,\ea\right.$$
and the cost functional
$$
J(t,x;u(\cd))=\dbE_t[\h F(X(t),X(T))]+\h G (X(t),\dbE_t[X(T)]),
$$
where $\h F(\cd)$ and $\h G(\cd)$ are given deterministic functions.
The so-called {\it extended HJB equation} derived by Bj\"{o}rk--Khapko--Murgoci  \cite{Bjork-Khapko-Murgoci2017} reads
\bel{EHJB-B}
\left\{
\begin{aligned}
&\inf_{u\in U}\Big(({\bf A}^u\h V)(t,x)-({\bf A}^u\h f)(t,x,x)+({\bf A}^u\h f^x)(t,x)-{\bf A}^u(\h G\diamond \h g)(t,x)+({\bf H}^u\h g)(t,x)\Big)=0,\\
&{\bf A}^{\h u} \h f^{\ti x}(t,x)=0,\q {\bf A}^{\h u} \h g(t,x)=0,\q (t,x)\in[0,T]\times\dbR^n,\\
& \h V(T,x)=\h F(x,x)+\h G(x,x),\q\h f^{\ti x}(T,x)=\h F(\ti x,x),\q \h g(T,x)=x,\q x,\ti x\in\dbR^n,
\end{aligned}\right.
\ee
where $\h u(\cd,\cd)$ denotes the strategy which realizes the infimum in the first equation; that is
\begin{align}
&({\bf A}^{\h u}\h V)(t,x)-({\bf A}^{\h u}\h f)(t,x,x)+({\bf A}^{\h u}\h f^x)(t,x)-{\bf A}^{\h u}(\h G\diamond \h g)(t,x)+({\bf H}^{\h u}\h g)(t,x)\nn\\
&\q=\inf_{u\in U}\Big(({\bf A}^u\h V)(t,x)-({\bf A}^u\h f)(t,x,x)+({\bf A}^u\h f^x)(t,x)-{\bf A}^u(\h G\diamond \h g)(t,x)+({\bf H}^u\h g)(t,x)\Big).
\end{align}
In the above, the following notations are used
\begin{align*}
\h f(t,x,\ti x)=\h f^{\ti x} (t,x),\q (\h G\diamond \h g)(t,x)= \h G(x,\h g(t,x)),\q {\bf H}^u\h g(t,x)=\h G_y (s,\h g(t,x)) {\bf A}^u\h g(t,x),
\end{align*}
and the operator ${\bf A}^{ u}$ is determined by
\bel{def-Au}
\dbE_t[k(t+h,X(t+h;t,x,u))]= k(t,x)+h {\bf A}^u k(t,x)+o(h),\q \forall k(\cd\,,\cd)\in C^{1,2}([0,T]\times\dbR^n;\dbR),
\ee
where $X(\cd\,;t,x,u)$ is the unique solution to the forward equation in \rf{State}.

\ms
The associated equilibrium HJB equation \rf{HJB-BSVIE-Main} reads
\bel{HJB-BSVIE-B}\left\{\2n\ba{ll}
\ds\Th_s(s,x)+\Th_x(s,x)b(s,x,\bar\Psi(s,x))+\tr[\Th_{xx}(s,x)a(s,x,\bar\Psi(s,x))]=0,\\
\ns\ds\Th^0_s(s,\ti x,x,y)+\tr[\Th^0_{xx}(s,\ti x,x,y)
a(s,x,\bar\Psi(s,x))]+\Th^0_x(s,\ti x,x,y)b(s,x,\bar\Psi(s,x))=0,\\
\ns\ds\Th(T,x)=x,\q\Th^0(T,\ti x,x,y)=\h F(\ti x,x)+\h G(\ti x,y),\ea\right.\ee
where $\bar\Psi(\cd\,,\cd)$ satisfies the {\it local optimality condition}:
\begin{align}
&\Th^0_x(t,x,x,\Th(t,x))b(t,x,\bar\Psi(t,x))+\tr[ \Th^0_{xx}(t,x,x,\Th(t,x))a(t,x,\bar\Psi(t,x))]\nn\\
&\qq+\Th^0_y(t,x,x,\Th(t,x))\big\{\Th_x(t,x)b(t,x,\bar\Psi(t,x)) +\tr[\Th_{xx}(t,x)a(t,x,\bar\Psi(t,x))]\big\}\nn\\
&\q=\inf_{u\in U}\Big\{\Th^0_x(t,x,x,\Th(t,x))b(t,x,u)+\tr[ \Th^0_{xx}(t,x,x,\Th(t,x))a(t,x,u)]\nn\\
&\qq+\Th^0_y(t,x,x,\Th(t,x))\big\{\Th_x(t,x)b(t,x,u)+\tr[\Th_{xx}(t,x)
a(t,x,u)]\big\}\Big\},\q (t,x)\in[0,T]\times\dbR^n.\label{HJB-BSVIE-OP-B}
\end{align}

\ms

Now we compare the equilibrium HJB equation \rf{HJB-BSVIE-B} derived in the paper
with the extended HJB equation \rf{EHJB-B} obtained by
Bj\"{o}rk--Khapko--Murgoci  \cite{Bjork-Khapko-Murgoci2017} carefully.

\begin{proposition}
Suppose that the equilibrium HJB equation \rf{HJB-BSVIE-B} admits a classical solution $(\Th(\cd),\Th^0(\cd))$.
Then the solution $\h V(\cd\,,\cd)$ of the extended HJB equation \rf{EHJB-B} and the equilibrium control law $\h u(\cd\,,\cd)$ can be given by
\bel{}\h V(t,x)=\Th^0(t,x,x,\Th(t,x)),\q
\h u(t,x)=\bar\Psi(t,x),\q (t,x)\in[0,T]\times\dbR^n.\ee
\end{proposition}

\begin{proof}
By the definition \rf{def-Au} of the operator ${\bf A}^u$, we have
$$({\bf A}^u k)(t,x) = k_t(t,x)+ k_x(t,x) b(t,x,u)+\tr[ k_{xx}(t,x)a(t,x,u)],\q (t,x)\in[0,T]\times\dbR^n,$$
for any $k(\cd\,,\cd)\in C^{1,2}([0,T]\times\dbR^n;\dbR)$.
Then
\begin{align}
&({\bf A}^u\h V)(t,x)-({\bf A}^u\h f)(t,x,x)+({\bf A}^u\h f^x)(t,x)-{\bf A}^u(\h G\diamond\h g)(t,x)+({\bf H}^u\h g)(t,x)\nn\\
&\q =\h V_t(t,x)+\h V_x(t,x)b(t,x,u)+\tr[\h V_{xx}(t,x)a(t,x,u)]-\big\{ \h f_{\ti x}(t,x,x)b(t,x,u)\nn\\
&\qq+2\tr[\h f_{x\ti x}(t,x,x)a(t,x,u)]+\tr[\h f_{\ti x\ti x}(t,x,x)a(t,x,u)]\big\}-\big\{\h G_{\ti x}(x,\h g(t,x))b(t,x,u)\nn\\
&\qq+\tr[\h G_{\ti x\ti x}(x,\h g(t,x))a(t,x,u)] + 2 \tr[ \h G_{\ti x y}(x,\h g(t,x))\h g_x(t,x)a(t,x)]\nn\\
&\qq +  \tr[ \h G_{yy}(x,\h g(t,x))\h g_x(t,x)^2a(t,x)]\big\},\label{A-u-V}
\end{align}
where
\begin{align*}
&\tr[ \h G_{\ti xy}(x,\h g(t,x))\h g_x(t,x)a(t,x)]=\sum _{i,j,k=1}^n \h G_{\ti x_i y_j}(x,\h g(t,x))\h g^j_{x_k}(t,x)\si^i(t,x)\si^k(t,x),\\
& \tr[ \h G_{yy}(x,\h g(t,x))\h g_x(t,x)^2a(t,x)]=\sum _{i,j,k,l=1}^n \h G_{y_i y_j}(x,\h g(t,x))\h g^i_{x_l}(t,x)\h g^j_{x_k}(t,x)\si^l(t,x)\si^k(t,x).
\end{align*}
Let
\bel{def-bar-u}
\bar V(t,x)=\Th^0(t,x,x,\Th(t,x)),\q
\bar u(t,x)=\bar\Psi(t,x),\q \bar g(t,x)=\Th(t,x),\q (t,x)\in[0,T]\times\dbR^n,
\ee
and
\bel{def-bar-f}\bar f(t,\ti x,x)=\Th^0(t,\ti x,x,y)-\h G(\ti x,y),\q (t,\ti x,x,y)\in[0,T]\times\dbR^n\times\dbR^n\times\dbR^n.\ee
We note that in the above, we use the fact that $\Th^0(t,\ti x,x,y)-\h G(\ti x,y)$ is independent of $y$.
Then it is clearly seen that
$$
 {\bf A}^{\bar u} \bar g(t,x)=0,\q (t,x)\in[0,T]\times\dbR^n,\q \bar g(T,x)=x,\q x\in\dbR^n.
$$
and
$$
{\bf A}^{\bar u} \bar f^z(t,x)=0,\q  (t,x)\in[0,T]\times\dbR^n,\q\bar f^z(T,x)=\h F(z,x).
$$
Note that $\Th^0_y(s,\ti x,x,y)=\h G_y(\ti x,y)$.
From \rf{A-u-V}, by some straightforward calculations we get
\begin{align}
\nn&({\bf A}^u\bar V)(t,x)-({\bf A}^u\bar f)(t,x,x)+({\bf A}^u\bar f^x)(t,x)-{\bf A}^u(\h G\diamond\bar g)(t,x)+({\bf H}^u\bar g)(t,x)\\
\nn&\q=\Th^0_t(t,x,x,\Th(t,x))+\Th^0_y(t,x,x,\Th(t,x))\Th_t(t,x)\\
\nn&\qq+\big\{\Th^0_x(t,x,x,\Th(t,x))+\Th^0_{\ti x}(t,x,x,\Th(t,x))+\Th^0_y(t,x,x,\Th(t,x))\Th_x(t,x)\big\}b(t,x,u)\\
\nn&\qq+\tr[\Th^0_{xx}(t,x,x,\Th(t,x))a(t,x,u)]+\tr[\Th^0_y(t,x,x,\Th(t,x))
\Th_{xx}(t,x)a(t,x,u)]\\
\nn&\qq +\tr[\Th^0_{\ti x\ti x}(t,x,x,\Th(t,x))a(t,x,u)]+2\tr[\Th^0_{x\ti x}(t,x,x\Th(t,x))a(t,x,u)]\\
\nn&\qq-\big\{\bar f_{\ti x}z(t,x,x)b(t,x,u)+2\tr[\bar f_{x\ti x}(t,x,x)a(t,x,u)]+\tr[\bar f_{\ti x\ti x}(t,x,x)a(t,x,u)]\big\}\\
\nn&\qq-\big\{\h G_{\ti x}(x,\Th(t,x))b(t,x,u)+\tr[\h G_{\ti x\ti x}(x,\Th(t,x))a(t,x,u)]\big\}\\
\nn&\q=\Th^0_t(t,x,x,\Th(t,x))+\Th^0_x(t,x,x,\Th(t,x))b(t,x,u)
+\tr[\Th^0_{xx}(t,x,x,\Th(t,x))a(t,x,u)]\\
\nn&\qq+\Th^0_y(t,x,x,\Th(t,x))\big\{\Th_t(t,x)+\Th_x(t,x) b(t,x,u)+\tr[\Th_{xx}(t,x)a(t,x,u)]\big\}\\
\nn&\qq+\big\{\Th^0_{\ti x}(t,x,x,\Th(t,x))-\bar f_{\ti x}(t,x,x) -\bar G_{\ti x}(x,\Th(t,x))\big\}b(t,x,u)\\
\nn&\qq +\tr\big\{[\Th_{\ti x\ti x}(t,x,x,\Th(t,x))-\bar f_{\ti x\ti x}(t,x,x)-\bar G_{\ti x\ti x}(x,\Th(t,x))]a(t,x,u)\big\}\\
\nn&\qq+2\tr[\Th_{x\ti x}(t,x,x,\Th(t,x))a(t,x,u)]-2\tr[\bar f_{x\ti x}(t,x,x)a(t,x,u)]\\
\nn&\q=\Th^0_t(t,x,x,\Th(t,x))+\Th^0_x(t,x,x,\Th(t,x))b(t,x,u)
+\tr[\Th^0_{xx}(t,x,x,\Th(t,x))a(t,x,u)]\\
&\qq+\Th^0_y(t,x,x,\Th(t,x))\big\{\Th_t(t,x)+\Th_x(t,x) b(t,x,u)+\tr[\Th_{xx}(t,x)a(t,x,u)]\big\}.\label{V-hat-V}
\end{align}
Then from \rf{HJB-BSVIE-OP-B} and \rf{def-bar-u}, we get
\begin{align*}
&({\bf A}^{\bar u}\bar V)(t,x)-({\bf A}^{\bar u}\bar f)(t,x,x)+({\bf A}^{\bar u}\bar f^x)(t,x)-{\bf A}^{\bar u}(\h G\diamond\bar g)(t,x)+({\bf H}^{\bar u}\bar g)(t,x)\nn\\
&\q=\inf_{u\in U}\big\{({\bf A}^u\bar V)(t,x)-({\bf A}^u\bar f)(t,x,x)+({\bf A}^u\bar f^x)(t,x)-{\bf A}^u(\h G\diamond\bar g)(t,x)+({\bf H}^u\bar g)(t,x)\big\}.
%
\end{align*}
Taking $u=\bar u$ in \rf{V-hat-V} and then by \rf{HJB-BSVIE-B}, we get
\begin{align*}
\inf_{u\in U}\big\{({\bf A}^u\bar V)(t,x)-({\bf A}^u\bar f)(t,x,x)+({\bf A}^u\bar f^x)(t,x)-{\bf A}^u(\h G\diamond\bar g)(t,x)+({\bf H}^u\bar g)(t,x)\big\}=0.
%
\end{align*}
This completes the proof.
\end{proof}

\ms

Compared with  Bj\"{o}rk--Khapko--Murgoci  \cite{Bjork-Khapko-Murgoci2017},
our approach has the following advantages.

\ms

$\bullet$
The equilibrium value function given by $V(t,x)\equiv \Th^0(t,x,x,\Th(t,x))$,
$\Th^0(\cd\,,\ti x,\cd\,,y)$ can be regarded as an auxiliary function  with parameters $(\ti x,y)$.
By introducing this auxiliary function,
the structure of equilibrium HJB equations is much clearer than that of extended HJB equations
(compare \rf{HJB-BSVIE-B} with \rf{EHJB-B}, for example), and
the meaning of the two PDEs in equilibrium HJB equations is also very clear (see  \autoref{EHJB-Prob}).

\ms

$\bullet$
The state term $X(T)$ and the conditional expectation term $\dbE_t[X(T)]$ in the terminal cost of \rf{cost-BSVIE1} could be inseparable,
while in \cite{Bjork-Khapko-Murgoci2017}, they are required to be separable.
The reason is that in our approach, we do not need to introduce a PDE to give an additional representation for
$\h G (X(t),\dbE_t[X(T)])$.
Thus, there are only two PDEs in the equilibrium HJB equation, while the extended HJB equation \rf{EHJB-B} is involved with three PDEs.

\ms

$\bullet$
More importantly, by \autoref{thm:HJB} the well-posedness of equilibrium HJB equations is established  under assumption \ref{ass:H4},
while there is no rigorous argument about the well-posedness of the extended HJB equation \rf{EHJB-B} given in \cite{Bjork-Khapko-Murgoci2017}.
More generally, the problem studied in the paper can depend on a  controlled backward process and
have a recursive cost functional, which is determined by a BSVIE.
In Subsections \ref{sub:example} and \ref{subsec:SG}, two examples are presented to show that
the adoption of backward  controlled processes is necessary in some important applications.

\subsection{Comparison with Peng \cite{Peng1993} and Yong \cite{Yong2010}}
Under different assumptions on the control domain $U$,
Peng \cite{Peng1993} and Yong \cite{Yong2010} studied the Pontryagin's maximum principle for
the optimal control problem of FBSDE \rf{State}
with the Bolza type cost functional \rf{Bolza2}.
This provides a probabilistic approach to the optimal control of Problem (N).
Next,  we are going to formally derive the ``HJB equation" associated with the value function,
which can be regarded as a PDE approach version of \cite{Peng1993,Yong2010}.
By  this approach, we can see how the backward state process $(Y(\cd),Z(\cd))$ in the running cost
affects the  time-consistency and the ``HJB equation" of the problem.

\ms
We assume that the value function $V(\cd\,,\cd)$ is smooth, and
\bel{Bolza-h(x)}
h^0(r,\ti x,x,y)=h^0(x),\qq g^0(r,s,\ti x,x,u,y,z,y^0,z^0)=g^0(s,x,u,y,z).
\ee
Suppose the problem has an optimal control $u^*(\cd)$  with the closed-loop representation $u^*(s)=\Psi^*(s,X^*(s))$.
For any closed-loop strategy $\Psi(\cd\,,\cd)$, with the state $X(\cd)$, we have
\begin{align}
V(t,x)&=\dbE_t\Big\{h(X(T))+\int_t^T g^0(s,X(s),\Psi(s,X(s)),Y(s),Z(s))ds\nn\\
&\q-\int_t^T\[V_s(s,X(s))+V_x(s,X(s))b(s,X(s),\Psi(s,X(s)))+\tr[V_{xx}(s,X(s))a(s,X(s),\Psi(s,X(s)))]\nn\\
&\q+g^0(s,X(s),\Psi(s,X(s)),Y(s),Z(s))\]\Big\}ds\nn\\
&=\dbE_t\Big\{h(X(T))+\int_t^T g^0\big(s,X(s),\Psi(s,X(s)),\Th^\Psi(s,X(s)),\Th_x^\Psi(s,X(s))\si(s,X(s),\Psi(s,X(s)))\big)ds\nn\\
&\q-\int_t^T\[V_s(s,X(s))+V_x(s,X(s))b(s,X(s),\Psi(s,X(s)))+\tr[V_{xx}(s,X(s))a(s,X(s),\Psi(s,X(s)))]\nn\\
&\q+g^0\big(s,X(s),\Psi(s,X(s)),\Th^\Psi(s,X(s)),\Th_x^\Psi(s,X(s))\si(s,X(s),\Psi(s,X(s)))\big)\]ds\Big\},\nn
\end{align}
with
\bel{HJB-PY1}\left\{
\begin{aligned}
&\Th^\Psi_s(s,x)+\Th^\Psi_x(s,x)b(s,x,\Psi(s,x))+\tr[\Th^\Psi_{xx}(s,x)a(s,x,\Psi(s,x))]\\
&\q+g\big(s,x,\Psi(s,x),\Th^\Psi(s,x),\Th_x^\Psi(s,x)\si(s,x,\Psi(s,x))\big)=0,\\
&\Th^\Psi(T,x)=h(x).
\end{aligned}\right.\ee
Note that
$$
V(t,x)=J(t,x;\Psi^*(\cd,\cd))\les J(t,x;\Psi(\cd,\cd)).
$$
By the optimality,  $\Psi^*(\cd\,,\cd)$ should satisfy
\begin{align}
&V_s(s,x)+V_x(s,x)b(s,x,\Psi^*(s,x))+\tr[V_{xx}(s,x)a(s,X(s),\Psi^*(s,x))]\nn\\
&\q+g^0\big(s,x,\Psi^*(s,x),\Th^{\Psi^*}(s,x),\Th_x^{\Psi^*}(s,x)\si(s,x,\Psi^*(s,x))\big)=0,
\label{HJB-PY2}
\end{align}
and
\begin{align}
&V_x(s,x)b(s,x,\Psi^*(s,x))+\tr[V_{xx}(s,x)a(s,x,\Psi^*(s,x))]\nn\\
&+g^0\big(s,x,\Psi^*(s,x),\Th^{\Psi^*}(s,x),\Th_x^{\Psi^*}(s,x)\si(s,x,\Psi^*(s,x))\big)\nn\\
&\q=\inf_{\Psi(\cd,\cd)}\[ V_x(s,x)b(s,x,\Psi(s,x))+\tr[V_{xx}(s,x)a(s,x,\Psi(s,x))]\nn\\
&\qq+g^0\big(s,x,\Psi(s,x),\Th^{\Psi}(s,x),\Th_x^{\Psi}(s,x)\si(s,x,\Psi(s,x))\big)\].\label{HJB-PY3}
\end{align}
On the other hand, under \rf{Bolza-h(x)}, the equilibrium HJB equation \rf{HJB-BSVIE-Main}--\rf{bar Psi} reads
\bel{HJB-BSVIE-Bolza-PY}\left\{\begin{aligned}
&\Th_s(s,x)+\Th_x(s,x)b(s,x,\bar\Psi(s,x))+\tr[\Th_{xx}(s,x)a(s,x,\bar\Psi(s,x))]\\
&\q+g\big(s,x,\bar\Psi(s,x),\Th(s,x),\Th_x(s,x)\si(s,x,\bar\Psi(s,x))\big)=0,\q (s,x)\in[0,T]\times\dbR^n,\\
&\Th^0_s(s,x)+\Th^0_x(s,x)b(s,x,\bar\Psi(s,x))+\tr[\Th^0_{xx}(s,x)a(s,x,\bar\Psi(s,x))]\\
&\q+g^0\big(s,x,\bar\Psi(s,x),\Th(s,x),\Th_x(s,x)\si(s,x,\bar\Psi(s,x))\big)=0,\q (s,x)\in[0,T]\times\dbR^n,\\
&\Th(T,x)=h(x),\q\Th^0(T,x)=h^0(x),\q x\in\dbR^n,
\end{aligned}\right.\ee
where
\bel{bar-Psi-PY}\ba{ll}
\bar\Psi(s,x)=\psi\big(s,x,\Th(s,x),\Th_x(s,x),\Th_{xx}(s,x),\Th^0(s,x),\Th^0_x(s,x),\Th^0_{xx}
(s,x)\big),\qq(s,x)\in[0,T]\times\dbR^n,\ea\ee
with $\psi(\cd)$ being determined by \rf{HH}.
The equilibrium strategy $\bar\Psi(\cd,\cd)$ satisfies the following local optimality condition:
\begin{align}
&\Th^0_x(s,x)b(s,x,\bar\Psi(s,x))+\tr[\Th^0_{xx}(s,x)a(s,x,\bar\Psi(s,x))]
+g^0\big(s,x,\bar\Psi(s,x),\Th(s,x),\Th_x(s,x)\si(s,x,\bar\Psi(s,x))\big)\nn\\
&\q=\inf_{u\in U}\[\Th^0_x(s,x)b(s,x,u)+\tr[\Th^0_{xx}(s,x)a(s,x,u)]
+g^0\big(s,x,u,\Th(s,x),\Th_x(s,x)\si(s,x,u)\big)\].
\end{align}

It is remarkable that if the terminal cost  does not depend on $y$,
the associated equilibrium HJB equation \rf{HJB-BSVIE-Bolza-PY} is reduced
to an $(m+1)$-dimensional fully nonlinear parabolic equation without  nonlocal terms.
By comparing \rf{HJB-PY1}--\rf{HJB-PY3} and \rf{HJB-BSVIE-Bolza-PY}--\rf{bar-Psi-PY},
we see that the difference between the  optimal strategy and the equilibrium strategy
mainly lies on the corresponding optimality conditions \rf{HJB-PY3} and \rf{bar-Psi-PY}.
Note that at $(s,x)$, $\Th^\Psi(s,x)$ depends on the values of $\Psi(\cd\,,\cd)$ on $[s,T]\times\dbR^n$,
\rf{HJB-PY3} is not an optimality condition in the finite dimensional space $U$ and
$\Psi^*(\cd\,,\cd)$ has an aftereffect.
Because of this, Problem (N) is time-inconsistent and the usual DPP/HJB approach does not work.

\section{Linear-Quadratic Problems}\label{sec:LQ}

Consider the controlled linear FBSDEs:
\bel{state-LQ}\left\{\2n\ba{ll}
\ds dX(s)=[A(s)X(s)+B(s)u(s)]ds+[C(s)X(s)+D(s)u(s)]dW(s),\\
\ns\ds dY(s)=-\big[\h A(s)X(s)+\h B(s)u(s)+\h C(s)Y(s)+\h D(s)Z(s)\big]ds+Z(s)dW(s),\\
\ns\ds X(t)=x,\qq Y(T)=HX(T).\ea\right.\ee
We introduce the following cost functional:
\begin{align}
&\cJ(t,x;u(\cd))\1n=\1n
{1\over 2}\dbE_t\Big\{\1n\int_t^T\3n\big[\lan Q(s)X(s),X(s)\ran\1n+\1n\lan M(s)Y(s),Y(s)\ran
\1n+\1n\lan N(s)Z(s),Z(s)\ran\1n+\1n\lan R(s)u(s),u(s)\ran\big]ds\nn\\
&\qq\qq\qq\qq+\lan G_1 X(T), X(T)\ran+\lan  G_2 Y(t) , Y(t)\ran+\lan G_3 X(t) , Y(t)\ran
 +2\lan g,X(T)\ran\Big\}.\label{cost-LQ}
\end{align}
The above problem is  referred to as a {\it linear-quadratic} (LQ,
for short) optimal control problem for FBSDEs, due to the linearity of the
state equation \rf{state-LQ} and the quadratic form of the cost functional \rf{cost-LQ}.
For simplicity, we shall denote the optimal control problem with state equation \rf{state-LQ}
and cost functional \rf{cost-LQ} by Problem (FBLQ).
We refer \cite{Lim-Zhou2001,Wang-Wu-Xiong2015,Huang-Wang-Wu2016,Wang-Xiao-Xiong2018,Li-Sun-Xiong2019,Hu-Ji-Xue2019,Sun-Wang2019,Sun-Wu-Xiong2021}
again for some related results of the LQ control problems for FBSDEs/BSDEs.

\begin{remark}
Note that in the cost functional \rf{cost-LQ}, we introduce a cross term $\lan G_3 X(t) , Y(t)\ran$.
In the literature, the dependence of initial states is motivated by the so-called
{\it state-dependent risk aversions} in finance
(see Bj\"{o}rk--Murgoci--Zhou \cite{Bjork-Murgoci-Zhou2014}).
Indeed,  the initial state $X(t)$, with a form of $\lan X(t),Y(t)\ran$, will also arise naturally when we study
the leader's problem of an LQ Stackelberg game (see  \cite[Subsection 3.2]{Sun-Wang-Wen2021}).
%
\end{remark}

Let us take a look at a special case of the above LQ problem.

\begin{example}\label{example-LQ}
Let $m=n=1$; $A,D\equiv0$, $C,B\equiv1$; $\h A,\h B,\h C,\h D\equiv0$, $H=1$; and $Q,M,N\equiv0$, $R\equiv2I$, $G_1=0$, $G_2=2I$, $G_3=0$, $g=0$.
Note that $Y(t)=\dbE_t[X(T)]$.
Then the state equation \rf{state-LQ} and the cost functional \rf{cost-LQ} are reduced to
$$\left\{\2n\ba{ll}
\ds dX(s)=u(s)ds+X(s)dW(s),\q s\in[t,T],\\
\ns\ds X(t)=x,\ea\right.$$
and
$$J(t,x;u(\cd))=\dbE_t\[\int_t^T|u(s)|^2ds+|\dbE_t[X(T)]|^2\].$$
By Yong \cite{Yong2013}, the unique optimal control $\bar u(\cd;t,x)$ with initial pair $(t,x)$ is given by
$$\bar u(s;t,x)=-{x\over T-t+1},\q s\in[t,T].$$
Then the optimal state process $\bar X(\cd)\equiv\bar X(\cd\,;t,x)$ is given by
$$\bar
X(s)=e^{-{1\over2}(s-t)+W(s)-W(t)}x+{x\over T-t+1}\int_t^s e^{-{1\over2}(s-r)+W(s)-W(r)}dr,\q s\in[t,T].$$
For any $\t\in(t,T)$, the optimal control with initial pair $(\t,\bar X(\t))$ is given by
$$\bar u(s;\t,\bar X(\t))=-{\bar X(\t)\over T-\t+1},\q s\in[\t,T].$$
Thus, on $[\t,T]$,
$$\bar u(\cd\,;t,x)\ne\bar u(\cd\,;\t,\bar X(\t)),$$
which implies that the problem is time-inconsistent.
\end{example}

From the above example and \autoref{example2},
we see that the LQ optimal control problem for FBSDEs  is also time-inconsistent in general.
Recently, an LQ problem for coupled FBSDEs was studied by Hu--Ji--Xue \cite{Hu-Ji-Xue2019},
in which, however, the time-consistency was not realized.
Thus the optimal control obtained in \cite{Hu-Ji-Xue2019} is a pre-committed optimal control.

\ms
In the following, we will mainly look at the corresponding forms of our equilibrium HJB equations.
The well-posedness of the associated Riccati equation is left for our future research.
The associated equilibrium HJB equation reads
$$\left\{\2n\ba{ll}
\ds\Th_s(s,x)+\Th_x(s,x)[A(s)x+B(s)\bar\Psi(s,x)]\\
\ns\ds\q+{1\over 2}\lan\Th_{xx}(s,x)[C(s)x+D(s)\bar\Psi(s,x)],\,C(s)x+D(s)\bar\Psi(s,x)\ran\\
\ns\ds\q+\h A(s)x+\h B(s)\bar\Psi(s,x)+\h C(s)\Th(s,x)+\h D(s)\Th_x(s,x)[C(s)x+D(s)\bar\Psi(s,x)]=0,\\
\ns\ds\Th^0_s(s,\ti x,x,y)+\Th^0_x(s,\ti x,x,y)[A(s)x+B(s)\bar\Psi(s,x)]\\
\ns\ds\q+{1\over 2}\lan \Th^0_{xx}(s,\ti x,x,y)[C(s)x+D(s)\bar\Psi(s,x)],\,C(s)x+D(s)\bar\Psi(s,x)\ran\\
\ns\ds\q+{1\over 2}\big\{\lan Q(s)x,x\ran+\lan M(s)\Th(s,x),\Th(s,x)\ran+\lan N(s)\Th_x(s,x)[C(s)x+D(s)\bar\Psi(s,x)],\\
\ns\ds\qq\Th_x(s,x)[C(s)x+D(s)\bar\Psi(s,x)]\ran+\lan R(s)\bar\Psi(s,x),\bar\Psi(s,x)\ran\big\}=0,\\
\ns\ds\Th(T,x)=Hx,\q\Th^0(T,\ti x,x,y)={1\over 2}\lan G_1x,x\ran+{1\over 2}\lan G_2 y,y\ran+{1\over 2}\lan G_3\ti x,y\ran+\lan g,x\ran,\ea\right.$$
where
\begin{align}
\bar\Psi(s,x)&=-[D^\top \Th^0_{xx}(s, x,x,\Th(s,x))D+R+D^\top\Th_x(s,x)^\top N\Th_x(s,x)D]^{-1}\nn\\
&\q\times \big\{[D^\top\Th^0_{xx}(s,x,x,\Th(s,x))C+ D^\top\Th_x(s,x)^\top N\Th_x(s,x)C]x+B^\top\Th^0_x(s,x,x,\Th(s,x))\nn\\
&\q+[B^\top\Th_x(s,x)^\top +\h B^\top+ D^\top\Th_x(s,x)^\top\h D^\top]\Th^0_y(s,x,x,\Th(s,x)\big\}.\nn
\end{align}
In the above, we have taken the ansatz $\Th_{xx}(\cd)\equiv 0$.
Now let us take the following ansatz for $\Th^0(s,\ti x,x,y)$ and $\Th(s,x)$:
\begin{align*}
&\Th^0(s,\ti x,x,y)={1\over2}\lan\F_1(s)x,x\ran+{1\over2}\lan\F_2(s)y,y\ran+{1\over2}\lan
\F_3(s)\ti x,y\ran+\F_4(s)x+\F_5(s),\\
&\Th(s,x)=\F_6(s)x+\F_7(s),
\end{align*}
where $\F_i(\cd)$ $(i=1,...,7)$ are undetermined functions (of proper dimensions).
Then the equilibrium strategy is given by
\begin{align}
\bar\Psi(s,x)&=-[D^\top\F_1D+R+D^\top\F_6^\top N\F_6D]^{-1}\Big\{D^\top\F_1C+D^\top\F_6^\top N\F_6C+ B^\top\F_1\nn\\
&\q+B^\top\F_6^\top\F_2\F_6+\h B^\top\F_2\F_6+D^\top\F_6^\top\h D^\top\F_2\F_6
+{1\over 2}B^\top\F_6^\top\F_3+{1\over 2}\h B^\top\F_3+{1\over 2}D^\top\F_6^\top\h D^\top\F_3\Big\}x\nn\\
&\q-[D^\top\F_1D+R+D^\top\F_6^\top N\F_6D]^{-1}\big\{B^\top\F_4+[B^\top\F_6^\top+\h B^\top+D^\top
\F_6^\top\h D^\top]\F_2\F_7\big\}\nn\\
&=:\bar\Psi(s)x+\bar v(s),\label{Psi-star-LQ2}
\end{align}
where $\F_i(\cd)$ is determined by the following {\it Riccati-type} ordinary differential equation
(ODE, for short):
\bel{Ri-LQ1}\left\{\2n\ba{ll}
\ds\dot\F_1+\F_1(A+B\bar\Psi)+(A+B\bar\Psi)^\top\F_1+(C+D\bar\Psi)^\top\F_1(C+D\bar\Psi)\\
\ns\ds\q+Q+\F_6^\top M\F_6+(C+D\bar\Psi)^\top\F_6^\top N\F_6(C+D\bar\Psi)+\bar\Psi^\top R\bar\Psi=0,\\
\ns\ds\dot\F_2=0,\q\dot\F_3=0,\q \dot\F_4+\bar v^\top B^\top\F_1+\F_4(A+B\bar\Psi)+\bar v^\top D^\top\F_1(C+D\bar\Psi)\\
\ns\ds\q+\F_7^\top M\F_6+\bar v^\top\F_6^\top N\F_6(C+D\bar\Psi)+\bar v^\top R\bar\Psi=0,\\
\ns\ds\dot\F_5+\F_4 B\bar v+{1\over 2}\bar v^\top D^\top\F_1 D\bar v+{1\over2}\bar v^\top\F_6^\top N\F_6\bar v+{1\over 2}\bar v^\top R\bar v
+{1\over 2}\F_7^\top M\F_7=0,\\
\ns\ds\dot\F_6+\F_6(A+B\bar\Psi)+\h A+\h B\bar\Psi+\h C\F_6+\h D\F_6(C+D\bar\Psi)=0,\\
\ns\ds\dot\F_7+\F_6 B\bar v+\h B\bar v+\h C\F_7+\h D\F_6 D\bar v=0,\\
\ns\ds\F_1(T)=G_1,\q\F_2(T)=G_2,\q\F_3(T)=G_3,\q\F_4(T)=g,\\
\ns\ds\F_5(T)=0,\q\F_6(T)=H,\q\Phi_7(T)=0.\ea\right.\ee

\begin{proposition}\label{Prop:LQ}
Suppose that the Riccati equation \rf{Ri-LQ1} admits a unique solution.
Then the strategy $\bar\Psi(\cd,\cd)$ given by \rf{Psi-star-LQ2}
is an equilibrium strategy of {\rm Problem (FBLQ)}.
\end{proposition}

When the weighting matrices $G_3=0$ and $g=0$, then $\Phi_i\equiv0$ for $i=3,4,5,7$,
and the Riccati equation \rf{Ri-LQ1} can be simplified.
The following result shows that the LQ problem for FBSDEs is closely related to
the so-called {\it mean-field LQ optimal control problems} (see \cite{Yong2013,Yong2017,Sun-Wang-Wu2021}, for example).
Let $\h A,\, \h B,\,\h C,\,\h D,\,M,\,N\equiv 0$, $G_3=0$, $g=0$ and $H=I_n$.
Then the state equation \rf{state-LQ} and the cost functional \rf{cost-LQ} are reduced to
$$\left\{\begin{aligned}
dX(s)&=[A(s)X(s)+B(s)u(s)]ds+[  C(s)X(s)+D(s)u(s)]dW(s),\\
X(t)&=x,
\end{aligned}\right.
$$
and
\begin{align*}
 J(t,x;u(\cd))&={1\over 2}\dbE_t\Big\{\int_t^T \big[\lan Q(s)X(s),X(s)\ran+\lan R(s)u(s),u(s)\ran \big]ds\\
&\qq\qq+\lan G_1 X(T), X(T)\ran+\lan G_2\dbE_t[X(T)] , \dbE_t[X(T)]\ran\Big\},
\end{align*}
respectively. The associated Riccati equation \rf{Ri-LQ1} reads
\bel{Ri-LQ1-S}\left\{\2n\ba{ll}
\ds\dot\F_1+\F_1(A+B\bar\Psi)+(A+B\bar\Psi)^\top\F_1
+(C+D\bar\Psi)^\top\F_1(C+D\bar\Psi)+Q+\bar\Psi^\top R\bar\Psi=0,\\
\ns\ds\dot\F_2=0,\q\dot\F_6+\F_6(A+B\bar\Psi)=0,\\
\ns\ds\F_1(T)=G_1,\q\F_2(T)=G_2,\q\F_6(T)=I_n,\ea\right.\ee
with
\bel{Ri-LQ12-S}\bar\Psi=-[D^\top\F_1D+R]^{-1}[D^\top\F_1C+ B^\top\F_1+B^\top\F_6^\top\F_2\F_6].\ee
Denote $\F=\F_1$ and $\h\F=\F_1+\F_6^\top\F_2\F_6$.
Then we can rewrite \rf{Ri-LQ1-S}--\rf{Ri-LQ12-S} as follows:
\bel{Ri-LQ2-S}\left\{\2n\ba{ll}
\ds\dot\F+\F(A+B\bar\Psi)+(A+B\bar\Psi)^\top\F+(C+D
\bar\Psi)^\top\F(C+D\bar\Psi)+Q+\bar\Psi^\top R\bar\Psi=0,\\
\ns\ds\dot{\h\F}+\h\F(A+B\bar\Psi)+(A+B\bar\Psi)^\top\h\F
+(C+D\bar\Psi)^\top\F(C+D\bar\Psi)+Q+\bar\Psi^\top R\bar\Psi=0,\\
\ns\ds\F(T)=G_1,\q\h\F(T)=G_1+G_2,\ea\right.\ee
with
\bel{}\bar\Psi=-[D^\top\F D+R]^{-1}[D^\top\F C+B^\top\h\F].\ee
We emphasize that \rf{Ri-LQ2-S} is exactly a special case of the Riccati-type equation derived by Yong \cite{Yong2017}.
Thus, under some positivity conditions, one can obtain
the well-posedness of \rf{Ri-LQ2-S}  from \cite[Theorem 4.6]{Yong2017} easily. The well-posedness of Riccati equation \rf{Ri-LQ1} will be published separately.

\section{Applications}\label{sec:SC}

In this section, we shall investigate three important applications, which are also the main motivations
of studying forward-backward optimal control problems mentioned in Introduction.

\subsection{Mean-variance Models}\label{sub:MV}

Consider a Black--Scholes market model in which there is one bond with the riskless interest rate $r>0$ and one stock with the appreciation rate $\m>0$ and volatility $\si>0$. Then a standard argument leads to the following SDE for the wealth process $X(\cd)$:
\bel{MV-state}\left\{\2n\ba{ll}
\ds dX(s)=[rX(s)+(\mu-r)u(s)]ds+\si u(s)dW(s),\q s\in[t,T],\\
\ns\ds X(t)=\xi,\ea\right.\ee
where $u(\cd)$ is the dollar amount invested in the stock.
The investor wishes to minimize the following functional:
\bel{MV-utility}
J(t,\xi;u(\cd))= -\dbE_t[X(T)]+{\g\over 2}\var_t[X(T)]= -\dbE_t[X(T)]+{\g\over 2}\dbE_t[|X(T)|^2]-{\g\over 2}|\dbE_t[X(T)]|^2.
\ee
It is known (see Basak--Chabakauri \cite{Basak-Chabakauri2010}) that the optimal control of the above mean-variance model is time-inconsistent.
We shall apply \autoref{theorem-equilibrium-BSVIE-Main}  to find a time-consistent equilibrium.

\ms

Note that the  cost functional \rf{MV-utility} can be rewritten as
$$J(t,\xi;u(\cd))= -\dbE_t[X(T)]+{\g\over 2}\dbE_t[|X(T)|^2]-{\g\over 2}|Y(t)|^2,$$
with
$$\left\{\2n\ba{ll}
\ns\ds dX(s)=[rX(s)+(\m-r)u(s)]ds+\si u(s)dW(s),\q s\in[t,T],\\
\ns\ds dY(s)=Z(s)dW(s),\q s\in[t,T],\\
\ns\ds X(t)=\xi,\q Y(T)=X(T).\ea\right.
$$
Thus, the mean-variance model is a special case of the linear-quadratic problems for FBSDEs.
By  \autoref{Prop:LQ}, the equilibrium strategy $\bar\Psi(\cd,\cd)$ can be given by
$$\bar\Psi(t,x)=\bar\Psi(s)x+\bar v(s),\q(t,x)\in[0,T]\times\dbR,$$
where
\begin{align*}
&\bar\Psi(s)=-[\si^2\F_1]^{-1}\big[(\m-r)\F_1+(\m-r)\F_2\F^2_6
+{1\over 2}(\m-r)\F_3\F_6\big],\\
&\bar v(s)=-[\si^2\F_1]^{-1}\big[(\m-r)\F_4+(\m-r)\F_6\F_2\F_7\big],
\end{align*}
with
$$\left\{\2n\ba{ll}
\ds\dot\F_1+2\F_1(r+(\m-r)\bar\Psi)+\si\bar\Psi\F_1\si\bar\Psi=0,\\
\ns\ds\dot\F_2=0,\q\dot\F_3=0,\\
\ns\ds\dot\F_4+\bar v(\m-r)\F_1+\F_4(r+(\m-r)\bar\Psi)+\bar v\si\F_1\si\bar\Psi=0,\\
\ns\ds\dot\F_5+\F_4(\m-r)\bar v+{1\over 2}\bar v\si\F_1\si\bar v=0,\\
\ns\ds\dot\F_6+\F_6(r+(\m-r)\bar\Psi)=0,\\
\ns\ds\dot\F_7+\F_6(\m-r)\bar v=0,\\
\ns\ds\F_1(T)=\g,\q\F_2(T)=-\g,\q\F_3(T)=0,\q\F_4(T)=-1,\\
\ns\ds\F_5(T)=0,\q\F_6(T)=1,\q\F_7(T)=0.\ea\right.
$$
From the above, it is easily seen that $\F_2\equiv-\g$, $\F_3\equiv 0$ and $\F_1+\F_6\F_2\F_6\equiv 0$. Thus,
\bel{Psi-star-MV1}\bar\Psi=0,\q\bar v=-[\si^2 \F_1]^{-1}\big[(\m-r)\F_4-(\m-r)\F_6\g\F_7\big],\ee
and
\bel{Ri-MV1}\left\{\2n\ba{ll}
\ds\dot\F_1+2r\F_1=0,\\
\ns\ds\dot\F_4-\si^{-2}(\m-r)^2\big[\F_4-\F_6\g\F_7\big]\F_1+r\F_4=0,\\
\ns\ds\dot\F_6+r\F_6=0,\\
\ns\ds\dot\F_7-\F_6(\m-r)[\si^2\F_1]^{-1}\big[(\m-r)\F_4-(\m-r)\F_6\g\F_7\big]=0,\\
\ns\ds\F_1(T)=\g,\q\F_4(T)=-1,\\
\ns\ds\F_6(T)=1,\q\F_7(T)=0.\ea\right.\ee
By first solving the unknown variables $\F_1$ and $\F_6$, equation \rf{Ri-MV1} becomes a linear equation. By the variation of
constants formula, the unique solution $(\F_1,\F_4,\F_6,\F_7)$ of equation \rf{Ri-MV1} can be explicitly solved. Then the equilibrium strategy can be given by  \rf{Psi-star-MV1}. Indeed, we can observe that
$$
{d[\F_4-\F_6\g\F_7\big]\over dt}=-\g[\F_4-\F_6\g\F_7\big],\q [\F_4-\F_6\g\F_7\big](T)=-1,
$$
which implies that
$$[\F_4-\F_6\g\F_7\big]=-e^{-\g(t-T)}.$$
Substituting the above and $\F_1(t)=e^{-2\g(t-T)} $ into \rf{Psi-star-MV1}, the equilibrium strategy $\bar\Psi(\cd,\cd)$ is  explicitly given by
$$\bar\Psi(t,x)={\m-r\over\g\si^2}e^{-r(t-T)},\q (t,x)\in[0,T]\times\dbR.$$

\ms

From the above, we see that the optimal control problem of FBSDEs is a natural extension of the conditional mean-variance problem,
with the dynamic risk measure $\var\1n_t[X(T)]$ replaced by some more general ones,
which can be described by some process $Y(\cd)$ satisfying a BSDE.
We refer the reader to Riedel \cite{Riedel2004}, Barrieu--El Karoui \cite{Barrieu2005} and Detlefsen--Scandolo \cite{Detlefsen-Scandolo2005}
for the theory of risk measures.
Mathematically, the dynamic risk measure $Y(t)$ could depend on the whole path $X(s);t\les s\les T$ of the state $X(\cd)$,
while $\var\1n_t[X(T)]$ is only affected by the terminal value $X(T)$. Moreover, Problem (N) can also be regarded as an extension of
the dynamic mean-variance models with the conditional expectation operator $\dbE_t[\,\cd\,]$ replaced by the so-called {\it g-expectation} operator $\cE_{g,t}[\,\cd\,]$,
which was introduced by Peng \cite{Peng1997} and then was widely applied in finance; see Chen--Epstein \cite{Chen-Epstein2002},
Coquet \cite{Coquet2002} and Chen--Chen--Davison \cite{Chen-Chen-Davison2005}, for example.

\subsection{Social Planner Problems
with Heterogeneous Epstein-Zin Utilities}\label{sub:example}

In this subsection, we shall consider a social planner problem for Merton's investment-consumption models,
in which  each agent's objective is given by an  Epstein--Zin utility.
The social planner would like to maximize the utility of the coalition,
which is a convex combination of each agent's utility.
The main feature of our model is that the discount rate in each agent's utility can be different.
We will reveal two interesting facts: (1) the model is time-inconsistent;
(ii) the situation of controlled backward state equations is not avoidable in this model.

\ms

Consider the following SDE for the wealth process $X(\cd)$:
$$\left\{\2n\ba{ll}
\ds dX(s)=\big\{rX(s)+(\m-r)[u_1(s)+u_2(s)]-[c_1(s)+c_2(s)]\big\}ds+\si \big[u_1(s)+u_2(s)\big]dW(s),\\
\ns\ds X(t)=\xi,\ea\right.
$$
where $u_i(\cd)$ and $c_i(\cd)$ are the dollar amount invested in the stock and the consumption of agent $i$ ($i$=1,2), respectively.
Naturally, agent $i$ wants to maximize his/her utility functional
$$J_i(t,\xi;u_1(\cd),u_2(\cd),c_1(\cd),c_2(\cd))= Y_i(t),$$
where $Y_i(\cd)$, called an Epstein--Zin utility (see \cite{Duffie-Epstein1992,El Karoui-Peng-Quenez1997}, for example),
is determined by
\bel{M-BSDE}Y_i(s)=\dbE_s\[\int_s^T g_i(c_1(r)+c_2(r),Y_i(r))dr+h_i(X(T))\],\q s\in[t,T],\ee
with
$$g_i(c,y)=\a^{-1}((1-\g)y)^{1-{\a\over 1-\g}}[c^\a-\rho_i((1-\g)y)^{{\a\over 1-\g}}],\q h_i(x)={x^{1-\g}\over 1-\g}.$$
The parameter  $\g>0$ controls the
risk aversion of the agents,  $ {\a\over 1-\g}>0$ gives the agents' IES,
and $\rho_i$ is the discount rate of agent $i$ (which could be different for different $i$).

\ms

Such kind of models was initially studied by Duffie--Geoffard--Skiadas \cite{Duffie-Geoffard-Skiadas-1994}
(also see Ma--Yong \cite[Page 6]{Ma-Yong1999}),
however, the time-inconsistency issue was not realized.
If the agents decide to cooperate, then the social planer would try to maximize
\begin{align*}
&J^\l (t,\xi;u_1(\cd),u_2(\cd),c_1(\cd),c_2(\cd))\\
&\q= \l J_1(t,\xi;u_1(\cd),u_2(\cd),c_1(\cd),c_2(\cd))+(1-\l)J_2(t,\xi;u_1(\cd),u_2(\cd),c_1(\cd),c_2(\cd)),
\end{align*}
where $\l\in[0,1]$  is a weighting parameter of the two agents.
Denote $c(\cd)=c_1(\cd)+c_2(\cd)$ and $u(\cd)=u_1(\cd)+u_2(\cd)$.
Then the state equation and the utility functional of the social planner (or called the group decision-maker) become
\bel{M-state1}\left\{\begin{aligned}
dX(s)&=[rX(s)+(\mu-r)u(s)-c(s)]ds+\si u(s)dW(s),\\
dY_i(s)&=- g_i(c(s),Y_i(s))ds+Z_i(s)dW(s);\q i=1,2,\\
X(t)&=\xi,\q Y_i(T)=h_i(X(T)),
\end{aligned}\right.
\ee
and
\begin{align}
&J^\l (t,\xi;u(\cd),c(\cd))\equiv J^\l (t,\xi;u_1(\cd),u_2(\cd),c_1(\cd),c_2(\cd))\nn\\
&\q=\dbE_t\Big\{{X(T)^{1-\g}\over 1-\g}+\int_t^T \[\l \a^{-1}((1-\g)Y_1(r))^{1-{\a\over 1-\g}}\big(c(r)^\a-\rho_1((1-\g)Y_1(r))^{{\a\over 1-\g}}\big)\nn\\
&\qq+(1-\l) \a^{-1}((1-\g)Y_2(r))^{1-{\a\over 1-\g}}\big(c(r)^\a-\rho_2((1-\g)Y_2(r))^{{\a\over 1-\g}}\big)\]dr\Big\}.
\label{M-utility1}
\end{align}

\begin{remark}
Note that when $\a=1-\g$, the  Epstein--Zin utility $Y_i(\cd)$ is reduced to the standard
constant relative risk aversion  utility case,
because the corresponding BSDE \rf{M-BSDE} is linear with respect to the unknown process $Y_i(\cd)$.
Then the corresponding  utility \rf{M-utility1} becomes
\begin{align}
J^\l (t,\xi;u(\cd),c(\cd))&=\dbE_t\Big\{\big[\l e^{- \rho_1(T-t)}+(1-\l)e^{-\rho_2(T-t)}\big]  {X(T)^{\a}\over \a}\nn\\
&\q\qq+\int_t^T \big[\l e^{- \rho_1(r-t)}+(1-\l)e^{-\rho_2(r-t)}\big] c(r)^\a dr\Big\}.\label{cost-qe}
\end{align}
The control problem with state equation \rf{M-state1} and utility functional \rf{cost-qe}
is exactly the Merton's problem with a quasi-exponential discounting function $\l e^{- \rho_1\cd}+(1-\l)e^{-\rho_2\cd}$.
We refer the reader to \cite{Ekeland2008,Ekeland2010,Marin-Solano2010,Marin-Solano2011,Yong2012,Breton2014} for
more  results on this special case. In the general case, that is $\a$ could not equal $1-\g$,
the  Epstein--Zin utility $Y_i(\cd)$ is described by the solution to a nonlinear BSDE.
Then the situation of controlled BSDEs is not avoidable.
\end{remark}

It is clearly seen that the  control problem with state equation \rf{M-state1} and utility functional \rf{M-utility1} is time-inconsistent.
Thus, the group decision-maker should look for an equilibrium strategy for the coalition.
The associated equilibrium HJB equation reads
\bel{M-HJB0}
\left\{
\begin{aligned}
&\Th^1_t(t,x)+\Th^1_x(t,x)[rx+(\mu-r)\dbU(t,x)-\dbC(t,x)]+{1\over 2}\Th^1_{xx}(t,x)[\si\dbU(t,x)]^2\\
&\qq+\a^{-1}((1-\g)\Th^1(t,x))^{1-{\a\over 1-\g}}[\dbC(t,x)^\a-\rho_1((1-\g)\Th^1(t,x))^{{\a\over 1-\g}}]=0;\\
&\Th^2_t(t,x)+\Th^2_x(t,x)[rx+(\mu-r)\dbU(t,x)-\dbC(t,x)]+{1\over 2}\Th^2_{xx}(t,x)[\si\dbU(t,x)]^2\\
&\qq+\a^{-1}((1-\g)\Th^2(t,x))^{1-{\a\over 1-\g}}[\dbC(t,x)^\a-\rho_2((1-\g)\Th^2(t,x))^{{\a\over 1-\g}}]=0;\\
&\Th^0_t(t,x)+\Th^0_x(t,x)[rx+(\mu-r)\dbU(t,x)-\dbC(t,x)]+{1\over 2}\Th^0_{xx}(t,x)[\si\dbU(t,x)]^2\\
&\qq+\l\a^{-1}((1-\g)\Th^1(t,x))^{1-{\a\over 1-\g}}[\dbC(t,x)^\a-\rho_1((1-\g)\Th^1(t,x))^{{\a\over 1-\g}}]\\
&\qq +(1-\l)\a^{-1}((1-\g)\Th^2(t,x))^{1-{\a\over 1-\g}}[\dbC(t,x)^\a-\rho_2((1-\g)\Th^2(t,x))^{{\a\over 1-\g}}]=0;\\
&\Th^1(T,x)={x^{1-\g}\over 1-\g},\q \Th^2(T,x)={x^{1-\g}\over 1-\g},\q \Th^0(T,x)={x^{1-\g}\over 1-\g},
\end{aligned}\right.
\ee
with the equilibrium investment strategy:
\bel{EIS}
\dbU(t,x)={(r-\mu)\Th^0_x(t,x)\over \si^2\Th^0_{xx}(t,x)},
\ee
and the equilibrium consumption strategy:
\bel{ECS}
\dbC(t,x)={\Th^0_x(t,x)^{1\over\a-1}\over[\l((1-\g)\Th^1(t,x))^{1-\g-\a\over 1-\g}+(1-\l)((1-\g)\Th^2(t,x))^{1-\g-\a\over 1-\g}]^{1\over\a-1}}.
\ee
Let us make the ansatz:
\begin{align}
&\Th^1(t,x)={1\over 1-\g}x^{1-\g}\th_1(t),\q \Th^2(t,x)={1\over 1-\g} x^{1-\g}\th_2(t),\nn\\
& \Th^0(t,x)={\th_0(t)\over 1-\g} x^{1-\g}={\l\th_1(t)+(1-\l)\th_2(t)\over 1-\g} x^{1-\g}.\nn
\end{align}
Then
\begin{align}
&\Th^i_x(t,x)=\th_i(t) x^{-\g},\q \Th^i_{xx}(x)=-\g \th_i(t)x^{-\g-1},\nn\\
&\Th^0_x(t,x)=\l\th_1(t) x^{-\g}+(1-\l)\th_2(t) x^{-\g},\nn\\
&\Th^0_{xx}(x)=-\l\g \th_1(t)x^{-\g-1}-(1-\l)\g \th_2(t)x^{-\g-1}.\nn
\end{align}
The equilibrium investment strategy \rf{EIS} and the equilibrium consumption strategy \rf{ECS} become
\begin{align}
\mathbb{U}(t,x)&={(\mu-r)\over \g\si^2}x,\q x\in\dbR,\label{U}
\end{align}
and
\bel{C}
\mathbb{C}(t,x)={[\l\th_1(t) +(1-\l)\th_2(t) ]^{1\over \a-1}
\over [\l\th_1(t)^{1-\g-\a\over 1-\g}+(1-\l)\th_2(t)^{1-\g-\a\over 1-\g}]^{1\over\a-1}} x,\q x\in\dbR^n,
\ee
with
\bel{M-HJB}
\left\{
\begin{aligned}
&\dot{\th}_1(t)+(1-\g)\th_1(t)\[r+{(\mu-r)^2\over 2\g\si^2}-{[\l\th_1(t) +(1-\l)\th_2(t) ]^{1\over \a-1}\over [\l\th_1(t)^{1-\g-\a\over 1-\g}+(1-\l)\th_2(t)^{1-\g-\a\over 1-\g}]^{1\over\a-1}}\]\\
&\qq-(1-\g)\rho_1\a^{-1}\th_1(t)
+{\a^{-1}(1-\g)\th_1(t)^{1-\g-\a\over1-\g}[\l\th_1(t) +(1-\l)\th_2(t) ]^{\a\over \a-1}
\over [\l\th_1(t)^{1-\g-\a\over 1-\g}+(1-\l)\th_2(t)^{1-\g-\a\over 1-\g}]^{\a\over\a-1}}=0;\\
&\dot{\th}_2(t)+(1-\g)\th_2(t)\[r+{(\mu-r)^2\over 2\g\si^2}-{[\l\th_1(t) +(1-\l)\th_2(t) ]^{1\over \a-1}\over [\l\th_1(t)^{1-\g-\a\over 1-\g}+(1-\l)\th_2(t)^{1-\g-\a\over 1-\g}]^{1\over\a-1}}\]\\
&\qq-(1-\g)\rho_2\a^{-1}\th_2(t)
+{\a^{-1}(1-\g)\th_2(t)^{1-\g-\a\over1-\g}[\l\th_1(t) +(1-\l)\th_2(t) ]^{\a\over \a-1}
\over [\l\th_1(t)^{1-\g-\a\over 1-\g}+(1-\l)\th_2(t)^{1-\g-\a\over 1-\g}]^{\a\over\a-1}}=0;\\
&\th_1(T)=\th_2(T)=1.
\end{aligned}\right.
\ee
%
%
%
%
%
%

\begin{proposition}
The system \rf{M-HJB} of ODEs admits a unique solution $(\th_1(\cd),\th_2(\cd))$.
The strategies $\mathbb{U}(\cd,\cd)$ and $\mathbb{C}(\cd,\cd)$, given by \rf{U}--\rf{C},
is an equilibrium investment strategy and an equilibrium consumption strategy, respectively.
\end{proposition}

\begin{proof}
It suffices to show that if $(\th_1(\cd),\th_2(\cd))$ is a postive solution of \rf{M-HJB}
on $[t_0,T]$, then
$$
{\d}\les\th_i(s)\les \kappa,\q s\in[t_0,T];\q i=1,2,
$$
for some positive constants $\d,\kappa>0$ independent of $t_0$.
Without loss of generality, let $\rho_1\ges\rho_2$.
Denote $\D\th(\cd)=\th_1(\cd)-\th_2(\cd)$.
Note that
$$
\left\{
\begin{aligned}
&\D\dot{\th}(t)+(1-\g)\D\th(t)\[r+{(\mu-r)^2\over 2\g\si^2}-{[\l\th_1(t) +(1-\l)\th_2(t) ]^{1\over \a-1}\over [\l\th_1(t)^{1-\g-\a\over 1-\g}+(1-\l)\th_2(t)^{1-\g-\a\over 1-\g}]^{1\over\a-1}}\]\\
&\qq-(1-\g)\rho_1\a^{-1}\D\th(t)+(1-\g)(\rho_2-\rho_1)\a^{-1}\th_2(t)\\
&\qq+{\a^{-1}(1-\g-\a)[\l\th_1(t) +(1-\l)\th_2(t) ]^{\a\over \a-1}
\over [\l\th_1(t)^{1-\g-\a\over 1-\g}+(1-\l)\th_2(t)^{1-\g-\a\over 1-\g}]^{\a\over\a-1}}\int_0^1[l\th_1(t)+(1-l)\th_2(t)]^{-\a\over 1-\g}dl\D\th(t)=0,\\
&\D\th(T)=0.
\end{aligned}\right.
$$
and
$$(1-\g)(\rho_2-\rho_1)\a^{-1}\th_2(t)\les 0.$$
Then, $\th_1(\cd)\les\th_2(\cd)$.
Thus,
\begin{align}
&{\th_1(t)^{1-\g-\a\over1-\g}[\l\th_1(t) +(1-\l)\th_2(t) ]^{\a\over \a-1}
\over [\l\th_1(t)^{1-\g-\a\over 1-\g}+(1-\l)\th_2(t)^{1-\g-\a\over 1-\g}]^{\a\over\a-1}}\nn\\
&\q={\th_1(t)\th_1(t)^{-\a\over1-\g}[\l\th_1(t) +(1-\l)\th_2(t) ][\l\th_1(t) +(1-\l)\th_2(t) ]^{1\over \a-1}
\over [\l\th_1(t)^{1-\g-\a\over 1-\g}+(1-\l)\th_2(t)^{1-\g-\a\over 1-\g}] [\l\th_1(t)^{1-\g-\a\over 1-\g}+(1-\l)\th_2(t)^{1-\g-\a\over 1-\g}]^{1\over\a-1}}\nn\\
&\q={\th_1(t)[\l\th_1(t)^{1-\g-\a\over 1-\g} +(1-\l)\th_1(t)^{-\a\over1-\g}\th_2(t) ][\l\th_1(t) +(1-\l)\th_2(t) ]^{1\over \a-1}
\over [\l\th_1(t)^{1-\g-\a\over 1-\g}+(1-\l)\th_2(t)^{1-\g-\a\over 1-\g}] [\l\th_1(t)^{1-\g-\a\over 1-\g}+(1-\l)\th_2(t)^{1-\g-\a\over 1-\g}]^{1\over\a-1}}\nn\\
&\q\ges{\th_1(t)[\l\th_1(t)^{1-\g-\a\over 1-\g} +(1-\l)\th_2(t)^{1-\g-\a\over 1-\g} ][\l\th_1(t) +(1-\l)\th_2(t) ]^{1\over \a-1}
\over [\l\th_1(t)^{1-\g-\a\over 1-\g}+(1-\l)\th_2(t)^{1-\g-\a\over 1-\g}] [\l\th_1(t)^{1-\g-\a\over 1-\g}+(1-\l)\th_2(t)^{1-\g-\a\over 1-\g}]^{1\over\a-1}}\nn\\
&\q={\th_1(t)[\l\th_1(t) +(1-\l)\th_2(t) ]^{1\over \a-1}
\over  [\l\th_1(t)^{1-\g-\a\over 1-\g}+(1-\l)\th_2(t)^{1-\g-\a\over 1-\g}]^{1\over\a-1}}\nn,
\end{align}
which implies
\begin{align*}
&(1-\g)\th_1(t)\[-{[\l\th_1(t) +(1-\l)\th_2(t) ]^{1\over \a-1}\over [\l\th_1(t)^{1-\g-\a\over 1-\g}+(1-\l)\th_2(t)^{1-\g-\a\over 1-\g}]^{1\over\a-1}}\]\nn\\
&
+{\a^{-1}(1-\g)\th_1(t)^{1-\g-\a\over1-\g}[\l\th_1(t) +(1-\l)\th_2(t) ]^{\a\over \a-1}
\over [\l\th_1(t)^{1-\g-\a\over 1-\g}+(1-\l)\th_2(t)^{1-\g-\a\over 1-\g}]^{\a\over\a-1}}\\
&\q\ges{(\a^{-1}-1)(1-\g)\th_1(t)[\l\th_1(t) +(1-\l)\th_2(t) ]^{1\over \a-1}
\over  [\l\th_1(t)^{1-\g-\a\over 1-\g}+(1-\l)\th_2(t)^{1-\g-\a\over 1-\g}]^{1\over\a-1}}\ges 0.
\end{align*}
If follows that
$$
\th_1(t)\ges e^{\int_t^T(1-\g)[r-\rho_1\a^{-1}+{(\mu-r)^2\over 2\g\si^2}]ds}\ges  e^{-T|1-\g||r-\rho_1\a^{-1}+{(\mu-r)^2\over 2\g\si^2}|}=:\d>0.
$$
By $\th_2\ges\th_1$, we get $\th_2\ges\d$. Then
\begin{align}
&\th_i(t)^{1-\g-\a\over1-\g} \les \d^{1-\g-\a\over1-\g};\,i=1,2,\q [\l\th_1(t) +(1-\l)\th_2(t) ]^{\a\over \a-1}\les \d^{\a\over \a-1},\nn\\
& [\l\th_1(t)^{1-\g-\a\over 1-\g}+(1-\l)\th_2(t)^{1-\g-\a\over 1-\g}]^{1\over 1-\a}\les\d^{1-\g-\a\over (1-\a)(1-\g)}.\nn
\end{align}
From the above, we get that there exists a constant $\kappa>0$, independent of $t_0$, such that
$$
\th_i\les\kappa.
$$
Then a routine argument applies to get the well-posedness of the equation.
\end{proof}

\subsection{Stackelberg Games}\label{subsec:SG}
In this subsection, we consider a specific Stackelberg game (also called a leader-follower game).
We will show that the leader's problem in this Stackelberg game is an optimal control problem for FBSDEs,
whose optimal control is time-inconsistent. By applying \autoref{theorem-equilibrium-BSVIE-Main},
we can find a time-consistent equilibrium for the leader.  This will give a very good illustration.

\begin{example}
Consider the following one-dimensional state equation
\bel{state-LF-E}\left\{\2n\ba{ll}
\ds\dot X(s)=u_1(s)-u_2(s),\q s\in[t,1],\\
\ns\ds X(t)=x,\ea\right.\ee
and the cost functionals
\begin{align}
J_1(t,x;u_1(\cd),u_2(\cd))&=|X(1)|^2+\int_t^1\big[|u_1(s)|^2-|u_2(s)|^2\big]ds,\label{Cost-LF-E1}\\
J_2(t,x;u_1(\cd),u_2(\cd))&=\int_t^1\[-u_1(s)+{X(s)\over s-2}+ u_2(s)+|u_2(s)|^2\]ds.\label{Cost-LF-E2}
\end{align}
In the above, Player $2$ is the leader (or the principal), who announces his/her control $u_2(\cd)$ first, and Player $1$ is the follower (or the agent), who chooses his/her control $u_1(\cd)$ accordingly. Whatever the leader announces, the follower will select a control $\bar u_1(\cd\,;t,x,u_2(\cd))$ (depending the control $u_2(\cd)$ announced by the leader as well as the initial pair $(t,x)$) such that $u_1(\cd)\mapsto J_1(t,x;u_1(\cd),u_2(\cd))$ is minimized.
Knowing this, the leader will choose a $\bar u_2(\cd)$ a priori so that $u_2(\cd)\mapsto J_2(t,x;\bar u_1(\cd\,;t,x,u_2(\cd)),u_2(\cd))$ is minimized.
For any given initial pair $(t,x)$ and control $u_2(\cd)$ of the leader,
by the standard results of LQ control problems (see \cite[Chapter 6]{Yong-Zhou1999}),
the follower admits a unique optimal strategy $\bar u_1(\cd\,;t,x,u_2(\cd))$.
Then by some straightforward calculations, the leader's problem can be stated as follows:
Find a control $u_2(\cd)$ to minimize
\bel{Cost-LF-E2-1}
J_2(t,x;\bar u_1(\cd\,;t,x,u_2(\cd)),u_2(\cd))=\int_t^1\big[Y(s)+u_2(s)+|u_2(s)|^2\big]ds,\ee
with the backward evolution equation
\bel{state-LF-E1}\left\{\ba{ll}
\ds\dot Y(s)={1\over 2-s}Y(s)+{1\over 2-s}u_2(s),\q s\in[t,1],\\
\ns\ds Y(1)=0.\ea\right.\ee
Note that
$$
Y(s)={1\over s-2}\int_s^1 u_2(r)dr.
$$
Then
\begin{align}
J_2(t,x;\bar u_1(\cd),u_2(\cd))&=\int_t^1\[{1\over s-2}\int_s^1 u_2(r)dr+u_2(s)+|u_2(s)|^2\]ds\nn\\
&=\int_t^1\Big\{\big[\ln(2-s)-\ln(2-t)+1\big]u_2(s)+|u_2(s)|^2\Big\}ds.\nn
\end{align}
It follows that the unique optimal control of the leader is given by
$$
\bar u_2(s;t,x)={\ln(2-t)-\ln(2-s)-1\over 2},\q s\in[t,1].
$$
In particular, at the initial pair $(0,x)$, the unique optimal control of the leader is
$$
\bar u_2(s;0,x)={\ln2-\ln(2-s)-1\over 2},\q s\in[0,1].
$$
Let $\bar X(\cd)\equiv \bar X(\cd;0,x)$ be the state process with initial pair $(0,x)$
and optimal controls $(\bar u_1(s;t,x,\bar u_2),\,\bar u_2(\cd))$.
For any given $t\in(0,1)$, at the initial pair $(t,\bar X(t;0,x))$, the unique optimal control of the leader is
$$
\bar u_2(s;t,\bar X(t;0,x))={\ln(2-t)-\ln(2-s)-1\over 2},\q s\in[t,1].
$$
Thus, on the time interval $[t,1]$,
$$
\bar u_2(\cd;0,x)\neq \bar u_2(\cd;t,\bar X(t;0,x)),
$$
which implies that the leader's problem  is time-inconsistent.
By \autoref{theorem-equilibrium-BSVIE-Main},
we can easily obtain the time-consistent equilibrium strategy of the leader,
which is given by
$$
\bar\Psi(s,x)=-{1\over 2},\q (s,x)\in[0,1]\times\dbR^n.
$$
\end{example}

\begin{remark}
We refer the reader to \cite{Stackelberg1952,Yong2002,Cvitanic-Zhang2012, Bensoussan2018,Sun-Wang-Wen2021} for
some  theoretical  results and financial applications of Stackelberg games.
It is worthy of pointing out that the well-known {\it principal-agent problem}  (see \cite{Cvitanic-Zhang2012})
can be regarded as a special case.
\end{remark}

\section{Verification Theorem}\label{sec:Verification}

In this section, we shall show that the function $\bar\Psi(\cd\,,\cd)$, determined by \rf{bar Psi},
is an equilibrium strategy of Problem (N).
In other words, we would like to rigorously prove the verification theorem
(i.e., \autoref{theorem-equilibrium-BSVIE-Main}).
To do this, we assume that the equilibrium HJB equation \rf{HJB-BSVIE-Main}
admits a classical solution and the function $\bar \Psi(\cd\,,\cd)$ defined by \rf{bar Psi} is a feedback strategy.
We also assume that all the involved functions are bounded and differentiable with bounded derivatives.

\ms

Let $(\bar X(\cd),\bar Y(\cd),\bar Z(\cd))$ and $(\bar Y^0(\cd),\bar Z^0(\cd,\cd))$ be the solutions to FBSDE \rf{State} and BSVIE \rf{cost-BSVIE1},
respectively,  corresponding to the strategy $\bar \Psi(\cd\,,\cd)$ and the initial pair $(0,\xi)$.
For any $t\in[0,T)$, $u\in L_{\cF_t}^2(\Om;U)$ and $\e\in[0,T-t)$,
define the strategy $\Psi^\e(\cd\,,\cd)$ by \rf{def-equilibrium2}.
With the initial pair $(t,\bar X(t))\in\sD$, take the strategy $\Psi^\e(\cd\,,\cd)$,
then the corresponding state equation \rf{State} and cost functional \rf{cost-BSVIE0}--\rf{cost-BSVIE1} become
\bel{state-e}
\left\{\begin{aligned}
dX^\e(s)&=b(s,X^\e (s),\bar\Psi(s,X^\e(s)))ds+\si(s,X^\e (s),\bar\Psi(s,X^\e(s)))dW(s),\q s\in[t+\e,T];\\
dX^\e(s)&=b(s,X^\e (s),u )ds+\si(s,X^\e (s),u)dW(s),\q s\in[t,t+\e),\\
dY^\e (s)&=-g(s,X^\e(s),\bar\Psi(s,X^\e(s)),Y^\e(s),Z^\e(s))ds+Z^\e(s)dW(s),\q s\in[t+\e,T];\\
dY^\e (s)&=-g(s,X^\e(s),u,Y^\e(s),Z^\e(s))ds+Z^\e(s)dW(s),\q s\in[t,t+\e),\\
X^\e(t)&=\bar X(t),\q Y^\e(T)=h(X^\e(T)),
\end{aligned}\right.
\ee
and
\bel{cost-e}
J(t,\bar X(t);\Psi^\e(\cd))=Y^{0,\e}(t),
\ee
respectively, with
\bel{BSVIE-e11}
\left\{\begin{aligned}
&Y^{0,\e}(r)=h^0(r,X^\e(r),X^\e(T),Y^\e(r))-\int_r^T Z^{0,\e}(r,s)dW(s)\\
&\q+\int_r^T g^0(r,s,X^\e(r),X^\e(s),\bar\Psi(s,X^\e(s)),Y^\e(s),Z^\e(s),Y^{0,\e}(s),Z^{0,\e}(r,s))dr,\q r\in[t+\e,T];\\
&Y^{0,\e}(r)=h^0(r,X^\e(r),X^\e(T),Y^\e(r))-\int_r^TZ^{0,\e}(r,s)dW(s)\\
&\q+\int_{t+\e}^Tg^0(r,s,X^\e(r),X^\e(s),\bar\Psi(s,X^\e(s)),Y^\e(s),Z^\e(s),Y^{0,\e}(s),Z^{0,\e}(r,s))ds\\
&\q +\int_r^{t+\e}g^0(r,s,X^\e(r),X^\e(s),u,Y^\e(s),Z^\e(s),Y^{0,\e}(s),Z^{0,\e}(r,s))ds,\q r\in[t,t+\e).
\end{aligned}\right.
\ee

Next, let us deduce the PDEs associated with the FBSDE \rf{state-e} and BSVIE \rf{BSVIE-e11}.

\ms
By the Feynman--Kac formula for BSDEs  (see Pardoux--Peng \cite{Pardoux--Peng1992}, for example),
we get
\begin{align}
&Y^\e(s)=\Th( s,X^\e(s)),\q Z^\e(s)=\Th_x(s,X^\e(s))\si(s,X^\e(s),\bar\Psi(s,X^\e(s))),\q s\in[t+\e,T],\nn
\end{align}
where $\Th(\cd,\cd)$ is the unique solution to the first PDE in \rf{HJB-BSVIE-Main}.
Then on the time interval $[t,t+\e]$, we can rewrite \rf{state-e} as follows:
$$
\left\{\begin{aligned}
dX^\e(s)&=b(s,X^\e (s),u )ds+\si(s,X^\e (s),u)dW(s),\q s\in[t,t+\e],\\
dY^\e (s)&=-g(s,X^\e(s),u,Y^\e(s),Z^\e(s))ds+Z^\e(s)dW(s),\q s\in[t,t+\e],\\
X^\e(t)&=\bar X(t),\q Y^\e(t+\e)=\Th(t+\e,X^\e(t+\e)).
\end{aligned}\right.
$$
Note that the control $u\in L^2_{\cF_t}(\Om;U)$ is $\cF_t$-measurable.
Then by the Feynman--Kac formula for BSDEs again, we get
\bel{Y-e-Th-e}
Y^\e(s)=\Th^\e( s,X^\e(s)),\q Z^\e(s)=\Th^\e_x(s,X^\e(s))\si(s,X^\e(s),u),\q s\in[t,t+\e],
\ee
where $\Th^\e(\cd,\cd)$ is the unique classical solution to the following {\it perturbation} PDE:
\bel{Th-e}\left\{
\begin{aligned}
&\Th^\e_s(s,x)+\Th^\e_x(s,x)b(s,x,u)+\tr[\Th^\e_{xx}(s,x)a(s,x,u)]\\
&\q+g(s,x,u,\Th^\e(s,x),\Th^\e_x(s,x)\si(s,x,u))=0,\q s\in[t,t+\e],\\
&\Th^\e(t+\e,x)=\Th(t+\e,x).
\end{aligned}\right.
\ee

\begin{remark}
Indeed, \rf{Th-e} is a PDE with random parameters, because $u\in L^2_{\cF_t}(\Om;U)$ is a random variable.
However, note that $u$ is $\cF_t$-measurable and \rf{Th-e} is considered only on $[t,t+\e]$.
The random PDE \rf{Th-e}  can  be treated as a deterministic one.
\end{remark}

By \autoref{Prop:FK-BSVIE},  on the time interval $[t+\e,T]$, we get
\bel{cY-e-t+e}
Y^{0,\e}(s)= \Th^0(s,s,x,x,\Th(s,X^\e(s))),
\ee
where $\Th^0(\cd)$ is the solution of the second PDE in \rf{HJB-BSVIE-Main}.
Motivated by Wang--Yong--Zhang \cite{Wang-Yong-Zhang2021}, we introduce the following auxiliary processes with two time variables:
\bel{BSVIE-e}
\left\{\begin{aligned}
d Y^{0,\e}(r;s)&=-g^0\big(r,s,X^\e(r),X^\e (s),\bar\Psi(s,X^\e(s)),Y^\e(s), Z^\e(s),Y^{0,\e}(s),Z^{0,\e}(r;s)\big)ds\\
&\q+Z^{0,\e}(r;s)dW(s),\q s\in[(t+\e)\vee r,T],~r\in[t,T];\\
dY^{0,\e} (r;s)&=-g^0\big(r,s,X^\e(r),X^\e (s),u,Y^\e(s), Z^\e(s),Y^{0,\e}(s),Z^{0,\e}(r;s)\big)ds\\
&\q+Z^{0,\e}(r;s)dW(s),\q (r,s)\in\D^*[t,t+\e];\\
Y^{0,\e}(r;T)&= h^0(r,X^\e(r),X^\e(T),Y^\e(r)),\q r\in[t,T],
\end{aligned}\right.
\ee
which can give the unique solution of BSVIE \rf{BSVIE-e11} by
\bel{BSVIE-AP-Re}
Y^{0,\e}(s)=Y^{0,\e}(s;s),\q Z^{0,\e}(r,s)=Z^{0,\e}(r;s),\q (r,s)\in\D^*[t,T].
\ee
Notice that for any fixed $r\in[t,T]$, \rf{BSVIE-e} is a BSDE.
Recall the representations \rf{Y-e-Th-e} and \rf{cY-e-t+e}.
Then by  the Feynman--Kac formula for BSDEs again,  we get that for any $r\in[t,t+\e]$,
\begin{align}
& Y^{0,\e}(r;s)=\Th^0(r,s,X^\e(r),X^\e(s),\Th^\e(r,X^\e(r))),\q s\in [t+\e,T],\nn\\
& Z^{0,\e}(r;s)=\Th^0_x(r,s,X^\e(r),X^\e(s),\Th^\e(r,X^\e(r)))\si(s,X^\e(s),\bar\Psi(s,X^\e(s))), \q s\in [t+\e,T].\label{cY-e-Th}
\end{align}
On the other hand, by the flow property of the auxiliary process $Y^{0,\e}(\cd;\cd)$, we have
\begin{align}
Y^{0,\e}(r;r)&=Y^{0,\e}(r;t+\e)+\int_r^{t+\e}g^0(r,s,X^\e(r),X^\e(s),u,Y^\e(s),Z^\e(s),Y^{0,\e}(s),Z^{0,\e}(r;s))ds\nn\\
&\q-\int_r^{t+\e} Z^{0,\e}(r;s)dW(s),\q r\in[t,t+\e].\nn
\end{align}
Substituting  \rf{cY-e-Th} into the above and noting \rf{BSVIE-AP-Re}, we get
\begin{align}
Y^{0,\e}(r)&=\Th^0(r,t+\e,X^\e(r),X^\e(t+\e),\Th^\e(r,X^\e(r)))-\int_r^{t+\e} Z^{0,\e}(r,s)dW(s)\nn\\
&\q+\int_r^{t+\e}g^0(r,s,X^\e(r),X^\e(s),u,Y^\e(s),Z^\e(s),Y^{0,\e}(s),Z^{0,\e}(r,s))ds,\q r\in[t,t+\e].\label{BSVIE-e-t+e}
\end{align}
Then by \autoref{Prop:FK-BSVIE} (recalling \rf{Y-e-Th-e}--\rf{Th-e}),
we have the following representation:
\begin{align}
& Y^{0,\e}(r)= \Th^{0,\e}(r,r,X^\e(r),X^\e(r),\Th^\e(r,X^\e(r))),\nn\\
& Z^{0,\e}(r,s)= \Th_x^{0,\e}(r,s,X^\e(r),X^\e(s),\Th^\e(r,X^\e(r)))\si(s,X^\e(s),u),\q (r,s)\in\D^*[t,t+\e],
\label{cY-t-te}
\end{align}
where $\Th^{0,\e}(\cd)$ is the unique solution to the following PDE:
\bel{V-e-t}\left\{
\begin{aligned}
&\Th_s^{0,\e}(r,s,\ti x,x,y)+\Th_x^{0,\e}(r,s,\ti x,x,y)b(s,x,u)+\tr[\Th_{xx}^{0,\e}(r,s,\ti x,x,y)a(s,x,u)]\\
&\q+g^0\big(r,s,\ti x,x,u,\Th^\e(s,x),\Th^\e_x(s,x)\si(s,x,u),\Th^{0,\e}(s,s,x,x,\Th^\e(s,x)),\\
&\qq\q \Th_x^{0,\e}(r,s,\ti x,x,y)\si(s,x,u)\big)=0,\q (r,s)\in\D^*[t,t+\e],\\
& \Th^{0,\e} (r,t+\e,\ti x,x,y)=\Th^{0} (r,t+\e,\ti x,x,y),\q r\in[t,t+\e].
\end{aligned}\right.
\ee

\begin{remark}
Note that both \rf{Th-e} and \rf{V-e-t} are  semilinear parabolic equations.
To guarantee the well-posedness of PDEs \rf{Th-e} and \rf{V-e-t},
we assume that the following non-degenerate condition holds:
There exist two constants $\l_0,\l_1>0$ such that
\bel{non-degen-condition}
\l_0 I\les a(t,x,u)\les \l_1 I,\q \forall(t,x,u)\in[0,T]\times\dbR^n\times U.
\ee
\end{remark}

%

Under the assumption \rf{non-degen-condition},
we have the following convergence result of the families $\{\Th^\e(\cd)\}_{\e>0}$ and $\{\Th^{0,\e}(\cd)\}_{\e>0}$.

\begin{proposition}\label{lem:convergence}
Let \rf{non-degen-condition} hold. Then the PDEs \rf{Th-e} and \rf{V-e-t} admit unique classical solutions
$\Th^\e(\cd)$ and $\Th^{0,\e}(\cd)$, respectively.
Moreover, there exists a constant $K>0$, only depending on $\|\Th(\cd)\|_{C^{{\a\over 2},2+\a}}$ and $\|\Th^0(\cd)\|_{C^{{\a\over 2},{\a\over 2},\a,\a,1+\a}}$, such that
\bel{lem:convergence-main}
\|\Th^\e(\cd)-\Th(\cd)\|_{C^{0,2}[t,t+\e]}+\|\Th^{0,\e}(\cd)-\Th^0(\cd)\|_{C^{0,0,0,0,1}[t,t+\e]}\les K\e^{\a\over 2},
\ee
where $\a\in(0,1)$ is a constant.
\end{proposition}

\begin{remark}
The proof of \autoref{lem:convergence} is sketched in \autoref{sec:Proofs} as a byproduct of \autoref{thm:HJB}.
We emphasize that for \autoref{lem:convergence}, the assumption \rf{non-degen-condition} should not be necessary,
because one could replace the analytic approach by a probabilistic argument
(see \cite{Pardoux--Peng1992,Wang-Yong-Zhang2021}).
\end{remark}

\begin{remark}
The estimate \rf{lem:convergence-main} plays the same role as the convergence  assumption (H3)  in
Wei--Yong--Yu \cite{Wei-Yong-Yu2017}, which was proved only for some special cases
(see \cite[Theorem 6.2]{Wei-Yong-Yu2017}).
In \autoref{lem:convergence}, we can show that \rf{lem:convergence-main} holds in general.
The key point is that \rf{lem:convergence-main} is only a byproduct of the stability of semilinear parabolic equations,
while the assumption (H3) in \cite{Wei-Yong-Yu2017} is concerned with the fully nonlinear PDEs.
The deeper reason is that in our paper the main technique is the Feynman--Kac formula for BSVIEs/BSDEs,
while in \cite{Wei-Yong-Yu2017} they heavily rely on  the HJB equation approach.
\end{remark}

\subsection{Proof of \autoref{theorem-equilibrium-BSVIE-Main}}

For any fixed $t\in[0,T)$, $\e\in[0,T-t]$ and $u\in L_{\cF_t}^2(\Om;U)$,
let $\Th^\e(\cd)$ and $\Th^{0,\e}(\cd)$ be the unique classical solution to PDEs \rf{Th-e} and \rf{V-e-t}, respectively.
With the representations \rf{Y-e-Th-e} and \rf{cY-t-te},
by \rf{BSVIE-e-t+e} we can represent $Y^{0,\e}(t)$ as follows:
\begin{align}
Y^{0,\e}(t)=\dbE_t\[\Th^0(t,t+\e,X^\e(t),X^\e(t+\e),\Th^\e(t,X^\e(t)))+\int_t^{t+\e}g^{0,\e}(t,s,u)ds\],
\end{align}
where
\begin{align}
&g^{0,\e}(t,s,u):= g^0\Big(t,\,s,\,X^\e(t),\,X^\e(s),\,u,\,\Th^\e(X^\e(s),s),\Th^\e_x(s,X^\e(s))\si(s,X^\e(s),u),\nn\\
&\qq\qq\q \Th^{0,\e}(s,s,X^\e(s),X^\e(s),\Th^\e(s,X^\e(s))),\, \Th_x^{0,\e}(t,s,X^\e(t),X^\e(s),\Th^\e(t,X^\e(t)))\si(s,X^\e(s),u)\Big).
\label{g-0-e-(t,s,u)}
\end{align}
Note that $X^\e(t)=\bar X(t)$.
Applying It\^{o}'s formula to the mapping $s\mapsto \Th^0(t,s,X^\e(t),X^\e(s),\Th^\e(t,X^\e(t)))$ yields that
\begin{align*}
Y^{0,\e}(t)&=\dbE_t\Big\{\Th^0(t,t,\bar X(t),\bar X(t),\Th^\e(t,\bar X(t)))+\int_t^{t+\e}
\[\Th_s^0(t,s,\bar X(t),X^\e(s),\Th^\e(t,\bar X(t))) \\
&\qq+ \Th_x^0(t,s,\bar X(t),X^\e(s),\Th^\e(t,\bar X(t)))b(s,X^\e(s),u)+g^{0,\e}(t,s,u)\\
&\qq+\tr\big[\Th_{xx}^0(t,s,\bar X(t),X^\e(s),\Th^\e(t,\bar X(t)))a(s,X^\e(s),u)\big]\]ds\Big\}.
\end{align*}
Using the fact $\Th^\e(t+\e,\cd)=\Th(t+\e,\cd)$, we get
\begin{align}
Y^{0,\e}(t)&=\dbE_t\Big\{  \Th^0(t,t,\bar X(t),\bar X(t),\Th^\e(t,\bar X(t)))-\Th^0(t,t,\bar X(t),\bar X(t),\Th^\e(t+\e,\bar X(t+\e)))\nn\\
&\qq +\Th^0(t,t,\bar X(t),\bar X(t),\Th(t+\e,\bar X(t+\e)))+\int_t^{t+\e}\[ \Th_s^0(t,s,\bar X(t),X^\e(s),\Th^\e(t,\bar X(t))) \nn\\
&\qq+ \Th_x^0(t,s,\bar X(t),X^\e(s),\Th^\e(t,\bar X(t)))b(s,X^\e(s),u)+g^{0,\e}(t,s,u)\nn\\
&\qq+\tr\big[\Th_{xx}^0(t,s,\bar X(t),X^\e(s),\Th^\e(t,\bar X(t)))a(s,X^\e(s),u)\big]\]ds\Big\}.\nn
\end{align}
Then by applying the It\^{o}'s formula to the mapping
$s\mapsto \Th^0(t,t,\bar X(t),\bar X(t),\Th^\e(s,\bar X(s)))$, we have
\begin{align}
Y^{0,\e}(t)&=\dbE_t\Big\{\Th^0(t,t,\bar X(t),\bar X(t),\Th(t+\e,\bar X(t+\e)))\nn\\
&\qq-{1\over 2}\int_t^{t+\e}\tr\[\Th^0_{yy}(t,t,\bar X(t),\bar X(t),\Th^\e(s,\bar X(s)))
\Th_x^\e(s,\bar X(s))\bar\si(s)[\Th_x^\e(s,\bar X(s))\bar \si(s)]^\top\]ds\nn\\
&\qq-\int_t^{t+\e}\Th^0_y(t,t,\bar X(t),\bar X(t),\Th^\e(s,\bar X(s)))\[\Th_s^\e(s,\bar X(s)) +\Th_x^\e(s,\bar X(s))\bar b(s)\nn\\
&\qq+\tr[\Th_{xx}^\e(s,\bar X(s))\bar a(s)]\]ds+\int_t^{t+\e}\[ \Th_s^0(t,s,\bar X(t),X^\e(s),\Th^\e(t,\bar X(t))) \nn\\
&\qq+ \Th_x^0(t,s,\bar X(t),X^\e(s),\Th^\e(t,\bar X(t)))b(s,X^\e(s),u)+g^{0,\e}(t,s,u)\nn\\
&\qq+\tr[\Th_{xx}^0(t,s,\bar X(t),X^\e(s),\Th^\e(t,\bar X(t)))a(s,X^\e(s),u)]\]ds\Big\},\label{cY-tt}
\end{align}
where
\begin{align}
\bar \f(s):=\f(s,\bar X(s),\bar\Psi(s,\bar X(s))),\,\, s\in[t,t+\e],\q \hbox{for}\q \f(\cd)=b(\cd),\si(\cd),a(\cd).
\label{bar-f}
\end{align}
Recalling \rf{Th-e} and \rf{HJB-BSVIE-Main}, we get that on $[t,t+\e]$,
\begin{align}
&\Th_s^\e(s,\bar X(s)) +\Th_x^\e(s,\bar X(s))\bar b(s)+\tr[\Th_{xx}^\e(s,\bar X(s))\bar a(s)]\nn\\
&\q=\Th_x^\e(s,\bar X(s))[\bar b(s)-b(s,\bar X(s),u)]+\tr\big\{\Th_{xx}^\e(s,\bar X(s))[\bar a(s)-a(s,\bar X(s),u)]\big\}
-g^{\e}(s,u),\label{Th-e-s}
\end{align}
and
\begin{align}
&\Th_s^0(t,s,\bar X(t),X^\e(s),\Th^\e(t,\bar X(t))) + \Th_x^0(t,s,\bar X(t),X^\e(s),\Th^\e(t,\bar X(t)))b(s,X^\e(s),u)\nn\\
&+\tr[\Th_{xx}^0(t,s,\bar X(t),X^\e(s),\Th^\e(t,\bar X(t)))a(s,X^\e(s),u)]\nn\\
&\q=\Th_x^0(t,s,\bar X(t),X^\e(s),\Th^\e(t,\bar X(t)))[b(s,X^\e(s),u)-\bar b^{\e}(s)]\nn\\
&\qq+\tr\big\{\Th_{xx}^0(t,s,\bar X(t),X^\e(s),\Th^\e(t,\bar X(t)))[a(s,X^\e(s),u)-\bar a^{\e}(s)]\big\}-\bar g^{0,\e}(t,s),\label{V-e-s}
\end{align}
where
\begin{align}
&\bar\f^{\e}(s):= \f(s,X^\e(s),\bar\Psi(s,X^\e(s))),\q \hbox{for}\q \f(\cd)=b(\cd),\,\si(\cd),\,a(\cd);\nn\\
&g^{\e}(s,u):= g\big(s,\bar X(s),u,\Th^\e(s,\bar X(s)),\Th^\e_x(s,\bar X(s))\si(s,\bar X(s),u)\big);\nn\\
&\bar g^{0,\e}(t,s):= g^0\big(t,s,\bar X(t),X^\e(s),\bar\Psi(s,X^\e(s)),\Th(s,X^\e(s)), \Th_x(s,X^\e(s))\bar\si^{\e}(s),\nn\\
&\q \Th^0(s,s,X^\e(s),X^\e(s),\Th(s,X^\e(s))), \Th_x^0(t,s,\bar X(t),X^\e(s),\Th(t,\bar X(t)))\bar\si^{\e}(s)\big).\label{g-e-(s,u)}
\end{align}
Substituting \rf{Th-e-s} and \rf{V-e-s} into \rf{cY-tt} yields that
\begin{align}
 Y^{0,\e}(t)& =\dbE_t\Big\{\Th^0(t,t,\bar X(t),\bar X(t),\Th(t+\e,\bar X(t+\e)))\nn\\
&\qq+\int_t^{t+\e}\Th^0_y(t,t,\bar X(t),\bar X(t),\Th^\e(s,\bar X(s)))\[g^\e(s,u)+\Th_x^\e(s,\bar X(s))\nn\\
&\qq\q\times[b(s,\bar X(s),u)-\bar b(s)]+\tr\big\{\Th_{xx}^\e(s,\bar X(s))[a(s,\bar X(s),u)-\bar a(s)]\big\}\]ds\nn\\
&\qq-{1\over 2}\int_t^{t+\e}\tr\[\Th^0_{yy}(t,t,\bar X(t),\bar X(t),\Th^\e(s,\bar X(s)))
\Th_x^\e(s,\bar X(s))\bar\si(s)[\Th_x^\e(s,\bar X(s))\bar\si(s)]^\top\]ds\nn\\
&\qq+\int_t^{t+\e}\[g^{0,\e}(t,s,u)-\bar g^{0,\e}(t,s)+\Th_x^0(t,s,\bar X(t),X^\e(s),\Th^\e(t,\bar X(t)))[b(s,X^\e(s),u)-\bar b^\e(s)]\nn\\
&\qq\q+\tr\big\{\Th_{xx}^0(t,s,\bar X(t),X^\e(s),\Th^\e(t,\bar X(t)))[a(s,X^\e(s),u)-\bar a^\e(s)]\big\}\]ds\Big\}.\label{cYe}
\end{align}
Applying the above arguments to $\bar Y^0(t)$, we have
\begin{align}
J(t,\bar X(t);\bar\Psi(\cd))&\equiv \bar Y^0(t)=\dbE_t\Big\{\Th^0(t,t,\bar X(t),\bar X(t),\Th(t+\e,\bar X(t+\e)))\nn\\
&\q-{1\over 2}\int_t^{t+\e}
\tr\big\{\Th^0_{yy}(t,t,\bar X(t),\bar X(t),\Th(s,\bar X(s)))
\Th_x(s,\bar X(s))\bar\si(s)[\Th_x(s,\bar X(s))\bar\si(s)]^\top\big\}ds\nn\\
&\q+\int_t^{t+\e}\Th^0_y(t,t,\bar X(t),\bar X(t),\Th(s,\bar X(s)))\bar g(s,\bar\Psi(s,\bar X(s)))ds\Big\},\label{cY}
\end{align}
where
\begin{align}
\bar g(s,u)&:= g\big(s,\bar X(s),u,\Th(s,\bar X(s)),\Th_x(s,\bar X(s))\si(s,\bar X(s),u)\big),\q (s,u)\in[t,t+\e]\times U.
\label{bar-g-(s,u)}
\end{align}
Combining \rf{cYe} with \rf{cY} together, we get
\begin{align}
Y^{0,\e}(t)-\bar Y(t)&=\dbE_t\Big\{\int_t^{t+\e}\Th^0_y(t,t,\bar X(t),\bar X(t),\Th^\e(s,\bar X(s)))\[g^\e(s,u)+\Th_x^\e(s,\bar X(s))\nn\\
&\q\times[b(s,\bar X(s),u)-\bar b(s)]+\tr\big[\Th_{xx}^\e(s,\bar X(s))[a(s,\bar X(s),u)-\bar a(s)]\big]\]ds\nn\\
&\q-{1\over 2}\int_t^{t+\e}\tr\[\Th^0_{yy}(t,t,\bar X(t),\bar X(t),\Th^\e(s,\bar X(s)))
\Th_x^\e(s,\bar X(s))\bar\si(s)[\Th_x^\e(s,\bar X(s))\bar\si(s)]^\top\]ds\nn\\
&\q+\int_t^{t+\e}\[g^{0,\e}(t,s,u)-\bar g^{0,\e}(t,s)+\Th_x^0(t,s,\bar X(t),X^\e(s),\Th^\e(t,\bar X(t)))[b(s,X^\e(s),u)-\bar b^\e(s)]\nn\\
&\q+\tr\big\{\Th_{xx}^0(t,s,\bar X(t),X^\e(s),\Th^\e(t,\bar X(t)))[a(s,X^\e(s),u)-\bar a^\e(s)]\big\}\]ds\nn\\
&\q+{1\over 2}\int_t^{t+\e}\tr\big[\Th^0_{yy}(t,t,\bar X(t),\bar X(t),\Th(s,\bar X(s)))
\Th_x(s,\bar X(s))\bar\si(s)[\Th_x(s,\bar X(s))\bar\si(s)]^\top\big]ds\nn\\
&\q-\int_t^{t+\e}\Th^0_y(t,t,\bar X(t),\bar X(t),\Th(s,\bar X(s)))\bar g(s,\bar\Psi(s,\bar X(s)))ds\Big\}.\label{cYe-cY}
\end{align}
By the standard results of SDEs, we get
\begin{align}
&\dbE_t\[\sup_{s\in[t,t+\e]}\big(|\bar X(s)|^2+|X^\e(s)|^2\big)\]\les K(1+|\bar X(t)|^2),\nn\\
&\dbE_t\[\sup_{s\in[t,t+\e]}|\bar X(s)-X^\e(s)|^2\]\les K(1+|\bar X(t)|^2)\e.\label{X-Xe}
\end{align}
By \autoref{lem:convergence}, we have
\begin{align}
&\dbE_t\big[|\Th^\e(s,\bar X(s))-\Th(s,\bar X(s))|+|\Th_x^\e(s,\bar X(s))-\Th_x(s,\bar X(s))|\nn\\
&\q+|\Th_{xx}^\e(s,\bar X(s))-\Th_{xx}(s,\bar X(s))|\big]\les K\e^{\a\over 2}.
\end{align}
It follows that
\begin{align}
&\dbE_t\[\big|\Th^0_y(t,t,\bar X(t),\bar X(t),\Th^\e(s,\bar X(s)))-\Th^0_y(t,t,\bar X(t),\bar X(t),\Th(s,\bar X(s)))\big|\nn\\
&\q+\big|\Th^0_{yy}(t,t,\bar X(t),\bar X(t),\Th^\e(s,\bar X(s)))-\Th^0_{yy}(t,t,\bar X(t),\bar X(t),\Th(s,\bar X(s)))\big|\nn\\
&\q+\big|\Th_x^0(t,s,\bar X(t),X^\e(s),\Th^\e(t,\bar X(t)))-\Th_x^0(t,s,\bar X(t),\bar X(s),\Th(t,\bar X(t)))\big|\nn\\
&\q+\big|\Th_{xx}^0(t,s,\bar X(t),X^\e(s),\Th^\e(t,\bar X(t)))-\Th_{xx}^0(t,s,\bar X(t),\bar X(s),\Th(t,\bar X(t)))\big|\]\nn\\
&\qq\les K \e^{\a\over 2}+ K\e^{1\over 2}(1+|\bar X(t)|)\les K \e^{\a\over 2}(1+|\bar X(t)|),
\end{align}
and
\begin{align}
&\dbE_t\[|g^\e(s,u)-\bar g(s,u)|+|\bar g^{0,\e}(t,s)-\bar g^0(t,s,\bar\Psi(s,\bar X(s)))|\nn\\
&\q+|\bar b^{\e}(s)-\bar b(s)|+|\bar\si^{\e}(s)-\bar\si(s)|\]\les K\e^{\a\over2}(1+|\bar X(t)|),\label{g-e-g}
\end{align}
where
\begin{align}
\bar g^{0}(t,s,u)&:=g^0\Big(t,\,s,\,\bar X(t),\,\bar X(s),\,u,\,\Th(s,\bar X(s)),\,\Th_x(s,\bar X(s))\si(s,\bar X(s),u),\nn\\
&\qq\q \Th^0(s,s,\bar X(s),\bar X(s),\Th(s,\bar X(s))),\,\Th_x^0(t,s,\bar X(t),\bar X(s),\Th(t,\bar X(t)))\si(s,\bar X(s),u)\Big),
\end{align}
and the form of other functions is given in \rf{bar-f}, \rf{g-e-(s,u)}, and \rf{bar-g-(s,u)}.
Moreover, by \autoref{lem:convergence} and \rf{X-Xe}, we have
\begin{align}
&\dbE_t\[|g^{0,\e}(t,s,u)-\bar g^0(t,s,u)|\]\les K\e^{\a\over2}(1+|\bar X(t)|),\label{g-e-g11}
\end{align}
where $g^{0,\e}(t,s,u)$ is given by \rf{g-0-e-(t,s,u)}.
With the above estimates \rf{X-Xe}--\rf{g-e-g11}, from \rf{cYe-cY} we have
\begin{align}
 Y^{0,\e}(t)-\bar Y(t)&=\dbE_t\Big\{\int_t^{t+\e}\Th^0_y(t,t,\bar X(t),\bar X(t),\Th(s,\bar X(s)))
\[\bar g(s,u)-\bar g(s,\bar\Psi(s,\bar X(s)))\nn\\
&\qq+\Th_x(s,\bar X(s))[b(s,\bar X(s),u)-\bar b(s)]+\tr\big\{\Th_{xx}(s,\bar X(s))[a(s,\bar X(s),u)-\bar a(s)]\big\}\]ds\nn\\
&\qq+\int_t^{t+\e}\[\bar g^0(t,s,u)-\bar g(t,s,\bar\Psi(s,\bar X(s)))+\Th^0_x(t,s,\bar X(t),\bar X(s),\Th(s,\bar X(s)))\nn\\
&\qq\times[b(s,\bar X(s),u)-\bar b(s)]+\tr\big\{\Th^0_{xx}(t,s,\bar X(t),\bar X(s),\Th(s,\bar X(s)))[a(s,\bar X(s),u)-\bar a(s)]\big\}\]ds\Big\}\nn\\
&\q+o(\e)(1+|\bar X(t)|). \nn
\end{align}
Thus,
\begin{align}
&\liminf_{\e\to 0^+} {J(t,\bar X(t);\Psi^\e(\cd))-J(t,\bar X(t);\bar\Psi(\cd))\over\e}
=\liminf_{\e\to 0^+} {Y^{0,\e}(t)-\bar Y^0(t)\over\e}\nn\\
&\q=\bar\Th^0_y(t)\Big\{\bar\Th_x(t)[b(t,\bar X(t),u)-\bar b(t)]+\tr\big\{\bar\Th_{xx}(t)[a(t,\bar X(t),u)-\bar a(t)]\big\}\nn\\
&\qq+\bar g(t,t,u)-\bar g(t,t,\bar \Psi(t,\bar X(t)))\Big\}+ \bar\Th^0_x(t)[b(t,\bar X(t),u)-\bar b(t)]\nn\\
 &\qq+\tr\big\{\bar\Th^0_{xx}(t)[a(t,\bar X(t),u)-\bar a(t)]\big\}+\bar g^{0}(t,t,u)-\bar g(t,t,\bar\Psi(t,\bar X(t))),\nn
\end{align}
where
\begin{align}
\bar\Th(t):=\Th(t,\bar X(t)),\q \bar\Th^0(t):= \Th^0(t,t,\bar X(t),\bar X(t),\Th(t,\bar X(t))),\q t\in[0,T].\nn
\end{align}
Then by the local optimality condition \rf{local-opt} of $\bar\Psi(\cd,\cd)$, we have
$$
\liminf_{\e\to 0^+} {J(t,\bar X(t);\Psi^\e(\cd))-J(t,\bar X(t);\bar\Psi(\cd))\over\e}\ges 0,
$$
which completes the proof.

\section{Some Proofs}\label{sec:Proofs}
\subsection{Proof of \autoref{thm:HJB}}

For the ease of presentation, in the rest of the paper we restrict to the case with $m=1$ only.
However, all our results hold true in the multiple dimensional situation.
To begin with, let us first adopt some notations.

\ms\no
\textbf{Some Notations:}
For any functions $\varsigma:[S,T]\to\dbR$ and $\nu:\dbR^n\to\dbR$, with $\a\in(0,1)$ and $S\in[0,T)$, let
$$
\|\varsigma(\cd)\|_{\a\over 2}=\sup_{s_1,s_2\in[S,T],\,s_1\neq s_2}{|\varsigma(s_1)-\varsigma(s_2)|\over|s_1-s_2|^{\a\over 2}},\q
\|\nu(\cd)\|_{\a}=\sup_{x_1,x_2\in\dbR^n,\,x_1\neq x_2}{|\nu(x_1)-\nu(x_2)|\over|x_1-x_2|^\a}.
$$
For any $\f:[S,T]\times\dbR^n\to\dbR$, let
\begin{align*}
\|\f(\cd\,,\cd)\|_{C^{{\a\over 2},\a}([S,T]\times\dbR^n;\dbR)}&=\|\f(\cd\,,\cd)\|_{L^\i([S,T]\times\dbR^n;\dbR)}
+\sup_{x\in\dbR^n}\|\f(\cd\,,x)\|_{\a\over 2}+\sup_{s\in[S,T]}\|\f(s,\cd)\|_\a,\\
\|\f(\cd\,,\cd)\|_{C^{{\a\over 2},1+\a}([S,T]\times\dbR^n;\dbR)}&=\|\f(\cd\,,\cd)\|_{C^{0,1}([S,T]\times\dbR^n;\dbR)}
+\|\f(\cd\,,\cd)\|_{C^{{\a\over 2},\a}([S,T]\times\dbR^n;\dbR)}\\
&\q+\|\f_x(\cd\,,\cd)\|_{C^{{\a\over 2},\a}([S,T]\times\dbR^n;\dbR)}.
\end{align*}
We will often simply write $C^{{\a\over 2},1+\a}([S,T]\times\dbR^n;\dbR)$ as $C^{{\a\over 2},1+\a}$ when there is no confusion.
Similarly, we can define  $C^{{\a\over 2},{\a\over 2},\a,1+\a,2}([S,T]\times[S,T]\times\dbR^n\times\dbR^n\times\dbR;\dbR)$, etc.

\ms

For any $\th(\cd)\in C^{{\a\over 2},1+\a}$ and $\th^0(\cd)\in C^{{\a\over 2},{\a\over 2},\a,1+\a,2}$,
let us consider the following PDE:
\bel{Th-v}
\left\{\begin{aligned}
& \Th_s(s,x)+ \cL[s;\th(\cd),\th^0(\cd)]\Th(s,x)+\ti g\big(s,x,\th(s,x),\th_x(s,x),\th^0(s,x,s,x,\th(s,x)),\\
&\q \th_x^0(s,x,s,x,\th(s,x)),\th_y^0(s,x,s,x,\th(s,x))\big)=0,\\
&\Th^0_s(t,s,\ti x,x,y)+ \cL[s;\th(\cd),\th^0(\cd)]\Th^0(t,s,\ti x,x,y)+\ti g^0\big(t,s,\ti x,x,\th(s,x),\th_x(s,x),\\
&\q\th^0(s,x,s,x,\th(s,x)), \th_x^0(s,x,s,x,\th(s,x)),\th_y^0(s,x,s,x,\th(s,x)),\Th^0_x(t,s,\ti x,x,y)\big)=0,\\
&\Th(T,x)=h(x),\q  \Th^0_s(t,T,\ti x,x,y)=h^0(t,\ti x,x,y),
\end{aligned}\right.\ee
where the  differential operator $\cL[s;\th(\cd),\th^0(\cd)]$ is defined by the following:
\begin{align}
\cL[s;\th(\cd),\th^0(\cd)]\varphi(x)&=\tr[\varphi_{xx}(x)a(s,x)]+\varphi_{x}(x)
\ti b\big(s,x,\th(s,x),\th_x(s,x),\th^0(s,x,s,x,\th(s,x)), \nn\\
&\q\, \th^0_x(s,x,s,x,\th(s,x)), \th^0_y(s,x,s,x,\th(s,x))\big),\q \forall \varphi(\cd)\in C^2(\dbR^n;\dbR).
\label{def-def-operator}
\end{align}
We first present a result for the well-posedness of  PDE \rf{Th-v}.

\begin{lemma}\label{lem:HJB-linear}
Fix a $(t,\ti x, y)\in[0,T]\times\dbR^n\times\dbR$.
Then for any $\th(\cd)\in C^{{\a\over 2},1+\a}$ and $\th^0(\cd)\in C^{{\a\over 2},{\a\over 2},\a,1+\a,2}$,
the PDE \rf{Th-v} admits a unique classical solution $(\Th(\cd,\cd),\Th^0(t,\cd,\ti x,\cd,y))
\in C^{1+{\a\over 2},2+\a}\times C^{1+{\a\over 2},2+\a}$. Moreover, the following relationship holds:
\begin{align}
\nn &\Th(s,x)=\int_{\dbR^n}\Xi(s,x,T,\mu)h(\mu)d\mu+\int_s^T\int_{\dbR^n}\Xi(s,x,r,\mu)\[\Th_{x}(r,\mu)\nn\\
&\q\times\ti b\big(r,\mu,\th(r,\mu),\th_x(r,\mu),\th^0(r,\mu,r,\mu,\th(r,\mu)), \th_x^0(r,\mu,r,\mu,\th(r,\mu)),\th_y^0(r,\mu,r,\mu,\th(r,\mu))\big)\nn\\
&\q+\ti g\big(r,\mu,\th(r,\mu),\th_x(r,\mu),\th^0(r,\mu,r,\mu,\th(r,\mu)), \th_x^0(r,\mu,r,\mu,\th(r,\mu)),\th_y^0(r,\mu,r,\mu,\th(r,\mu))\big)\]d\mu dr;\label{lemma-PDE-main1}\\
\nn &\Th^0(t,s,\ti x,x,y)=\int_{\dbR^n}\Xi(s,x,T,\mu)L(t,\ti x,\mu,y)d\mu+\int_s^T\int_{\dbR^n}\Xi(s,x,r,\mu)\nn\\
&\q\times\[\Th^0_{x}(t,r,\ti x, \mu,y)\ti  b\big(r,\mu,\th(r,\mu),\th_x(r,\mu),\th^0(r,\mu,r,\mu,\th(r,\mu)), \th_x^0(r,\mu,r,\mu,\th(r,\mu)),\nn\\
&\qq\th_y^0(r,\mu,r,\mu,\th(r,\mu))\big)+\ti g^0\big(t,r,\ti x,\mu,\th(r,\mu),\th_x(r,\mu), \th^0(r,\mu,r,\mu,\th(r,\mu)),\nn\\
&\qq \th_x^0(r,\mu,r,\mu,\th(r,\mu)),\th_y^0(r,\mu,r,\mu,\th(r,\mu)),\Th^0_{x}(t,r,\ti x, \mu,y)\big)\]d\mu dr\label{lemma-PDE-main2},
\end{align}
where $\Xi(\cd\,,\cd\,,\cd\,,\cd)$ is given by the following explicitly:
\begin{align}
\Xi(s,x,r,\mu)={1\over (4\pi(r-s))^{{n\over 2}}(\det[a(r,\mu)])^{{1\over 2}}}e^{-{\lan a(r,\mu)^{-1}(x-\mu),(x-\mu)\ran\over 4(r-s)}},
\q (s,x),(r,\mu)\in[0,T]\times\dbR^n.\label{def-Xi}
\end{align}
\end{lemma}

\begin{proof}
For any fixed  $\th(\cd)\in C^{{\a\over 2},1+\a}$ and $\th^0(\cd)\in C^{{\a\over 2},{\a\over 2},\a,1+\a,2}$, denote
\begin{align*}
&\nu_1(s,x)=\th(s,x),\q \nu_2(s,x)=\th_x(s,x),\q \nu_3(s,x)=\th^0(s,s,x,x,\th(s,x)),\\
& \nu_4(s,x)=\th_x^0(s,s,x,x,\th(s,x)),\q \nu_5(s,x)=\th_y^0(s,s,x,x,\th(s,x)),\q (s,x)\in[0,T]\times\dbR^n.
\end{align*}
Then we have $\nu_i(\cd)\in C^{{\a\over 2},\a}$, for $i=1,...,5$.
Taking $(t,\ti x,y)$ as parameters, by the standard results of parabolic equations
(see \cite[Theorem 12, Chapter 1]{Friedman1964}  or \cite[Chapter IV, Sections 13--14]{Ladyzenskaja1968}, for example),
we get that PDE \rf{Th-v} admits a unique classical solution and \rf{lemma-PDE-main1}--\rf{lemma-PDE-main2} hold.
\end{proof}

Direct computations show that (see \cite{Friedman1964,Ladyzenskaja1968}, for example)
\bel{G-x-z}
\Xi_\mu(s,x,r,\mu)=-\Xi_x(s,x,r,\mu)-\Xi(s,x,r,\mu)\rho(s,x,r,\mu),
\ee
where
\bel{rho-def}
\left\{\begin{aligned}
& \rho(s,x,r,\mu)={(\det[a(r,\mu)])_\mu\over 2\det[a(r,\mu)]}+{\lan[a(r,\mu)^{-1}]_\mu(x-\mu),(x-\mu)\ran\over 4(r-s)},\\
& \lan[a(r,\mu)^{-1}]_\mu(x-\mu),(x-\mu)\ran=\begin{pmatrix}\lan [a(r,\mu)^{-1}]_{\mu_1}(x-\mu),(x-\mu)\ran \\
                                                    \lan [a(r,z)^{-1}]_{\mu_2}(x-\mu),(x-\mu)\ran\\
                                                    \vdots\\
                                                    \lan [a(r,z)^{-1}]_{\mu_n}(x-\mu),(x-\mu)\ran\end{pmatrix}.
\end{aligned}\right.
\ee
Moreover, under assumption \ref{ass:H4}, it is easy to check that
\bel{G-G-x-est}
\left\{\begin{aligned}
& |\Xi(s,x,r,\mu)|\les K{1\over (r-s)^{{n\over 2}}}e^{{-\l|x-\mu|^2\over 4(r-s)}},\\
& |\Xi_x(s,x,r,\mu)|\les K{1\over (r-s)^{{n+1\over 2}}}e^{{-\l|x-\mu|^2\over 4(r-s)}},\\
& |\rho(s,x,r,\mu)|\les K\left(1+{|x-\mu|^2\over (r-s)}\right),
\end{aligned}\right.
\ee
for some $0<\l<\l_0$. In what follows, we denote
$$
\ti\f(s,x;\th,\th^0):=\ti \f\big(s,\,x,\,\th(s,x),\,\th_x(s,x),\,\th^0(s,x,s,x,\th(s,x)),\,
\th_x^0(s,x,s,x,\th(s,x)),\,\th_y^0(s,x,s,x,\th(s,x))\big),
$$
for  $\f(\cd)=\ti b(\cd),\ti g(\cd)$, and
\begin{align}
\ti g^0(t,s,\ti x,x,y;\th,\th^0,\Th^0)&:=\ti g^0\big(t,\,s,\,\ti x,\,x,\,\th(s,x),\,\th_x(s,x),\,\th^0(s,x,s,x,\th(s,x)),\,\th_x^0(s,x,s,x,\th(s,x)),\nn\\
&\qq\q \th_y^0(s,x,s,x,\th(s,x)),\, \Th_x^0(t,s,\ti x,x,y)\big).
\end{align}

\ms

First, we establish a  $C^{0,1}$-norm estimate for $(\Th(\cd,\cd),\Th^0(t,\cd,\ti x,\cd,y))$.

\begin{lemma}\label{lem:HJB-Th-x}
There exists a constant $\kappa>0$, independent of  $\th(\cd)$ and $\th^0(\cd)$, such that
\begin{align}
&\|\Th(\cd,\cd)\|_{C^{0,1}}+\sup_{t,\ti x,y\in[0,T]\times\dbR^n\times\dbR}\|\Th^0(t,\cd,\ti x,\cd,y))\|_{C^{0,1}}\nn\\
&\q \les \k\[1+\| h(\cd)\|_{C^1}+\sup_{t,\ti x,y\in[0,T]\times\dbR^n\times\dbR}\|h^0(t,\ti x,\cd,y)\|_{C^1}\].
\label{lem:HJB-Th-x1}
\end{align}
\end{lemma}

\begin{proof}
By \rf{lemma-PDE-main1} and \rf{G-x-z}, using the  method of integration by parts, we have
%
%
%
%
%
%
\begin{align}
\nn \Th_x(s,x)&=\int_{\dbR^n}\big[\Xi(s,x,T,\mu)h_x(\mu)-\Xi(s,x,T,\mu)\rho(s,x,T,\mu)h(\mu)\big]d\mu\\
&\q+\int_s^T\int_{\dbR^n}\Xi_x(s,x,r,\mu)\[\Th_{x}(r,\mu)\ti b(r,\mu;\th,\th^0)
+\ti g(r,\mu;\th,\th^0)\]d\mu dr.\label{proof-step11-Th-x}
\end{align}
Then from the estimate \rf{G-G-x-est}, we get
\begin{align*}
|\Th_x(s,x)|&\les \int_{\dbR^n}K{1\over (r-s)^{{n\over 2}}}e^{{-\l|x-\mu|^2\over 4(r-s)}}
                                  \[|h_x(\mu)|+\(1+{|x-\mu|^2\over (r-s)}\)|h(\mu)|\]d\mu\\
&\q+\int_s^T\int_{\dbR^n} K{1\over (r-s)^{{n+1\over 2}}}e^{{-\l|x-\mu|^2\over 4(r-s)}}(1+|\Th_x(r,\mu)|)d\mu dr\\
&\les K(1+\|h(\cd)\|_{C^1})+\int_s^T\int_{\dbR^n} K{1\over (r-s)^{{n+1\over 2}}}e^{{-\l|x-\mu|^2\over 4(r-s)}}|\Th_x(r,\mu)|d\mu dr.
\end{align*}
By Gr\"{o}nwall's inequality, we obtain
$$
|\Th_x(s,x)|\les  K(1+\|h(\cd)\|_{C^1}),\q\forall (s,x)\in[0,T]\times\dbR^n.
$$
Substituting the above into \rf{lemma-PDE-main1} and then by \rf{G-G-x-est} again, we have
$$
|\Th(s,x)|\les  K(1+\|h(\cd)\|_{C^1}),\q\forall (s,x)\in[0,T]\times\dbR^n.
$$
It follows that
\bel{Th-C(0,1)-estimate}
\|\Th(\cd\,,\cd)\|_{C^{0,1}}\les  K(1+\|h(\cd)\|_{C^1}).
\ee
Similar to \rf{proof-step11-Th-x}, we get
\begin{align}
\nn\Th_x^0(t,s,\ti x,x,y)&=\int_{\dbR^n}\big[\Xi(s,x,T,\mu)h^0_x(t,\ti x,\mu,y)-\Xi(s,x,T,\mu)\rho(s,x,T,\mu)h^0(t,\ti x,\mu,y)\big]d\mu\\
&\q+\int_s^T\int_{\dbR^n}\Xi_x(s,x,r,\mu)\[\Th_x^0(t,r,\ti x,\mu,y)\ti b(r,\mu;\th,\th^0)
+\ti g^0(t,r,\ti x,\mu,y;\th,\th^0,\Th^0)\]d\mu dr\label{proof-step1-Vx}.
\end{align}
Note that $(t,\ti x,y)$  serve only as parameters in \rf{lemma-PDE-main2}.
By the same argument as the above, we get
$$
\sup_{t,\ti x,y\in[0,T]\times\dbR^n\times\dbR}\|\Th^0(t,\cd\,,\ti x,\cd\,,y))\|_{C^{0,1}}
\les  K\sup_{t,\ti x,y\in[0,T]\times\dbR^n\times\dbR}\big[1+\|h^0(t,\ti x,\cd\,,y)\|_{C^1}\big].
$$
Combining the above with \rf{Th-C(0,1)-estimate}, the desired estimate \rf{lem:HJB-Th-x1} is obtained.
\end{proof}

The following  gives the regularity estimate of
$\Th^0(t,s,\ti x,x,y)$ with respect to the parameters $t$, $\ti x$ and $y$.

\begin{lemma}\label{lem:HJB-par}
There exists a constant $\k>0$, independent of  $\th(\cd)$ and $\th^0(\cd)$, such that
\bel{lem:HJB-par-estimate}
\|\Th^0(\cd)\|_{C^{{\a\over 2},0,\a,1,2}}\les\k\big[1+\|h^0(\cd)\|^2_{C^{{\a\over 2},\a,1,2}}\big].
\ee
\end{lemma}

\begin{proof}
From \rf{lemma-PDE-main2} and \rf{proof-step1-Vx}, it is easily seen that
both $\Th^0(t,s,\ti x,x,y)$ and $\Th_x^0(t,s,\ti x,x,y)$ are differentiable with respect to the parameter $y$.
Moreover, the derivatives are given by
\begin{align}
\nn \Th_y^0(t,s,\ti x,x,y)&=\int_{\dbR^n}\Xi(s,x,T,\mu)h^0_y(t,\ti x,\mu,y)d\mu\\
&\q+\int_s^T\int_{\dbR^n}\Xi(s,x,r,\mu)\Th_{xy}^0(t,r,\ti x,\mu,y)\big[\ti b(r,\mu;\th,\th^0)
+\ti g_{p^0}^0(t,r,\ti x,\mu,y;\th,\th^0,\Th^0)\big] d\mu dr,\nn\\
\nn\Th_{xy}^0(t,s,\ti x,x,y)
&=\int_{\dbR^n}\big[\Xi(s,x,T,\mu)h^0_{xy}(t,\ti x,\mu,y)-\Xi(s,x,T,\mu)\rho(s,x,T,\mu)h^0_y(t,\ti x,\mu,y)\big]d\mu\\
\nn&\q+\int_s^T\int_{\dbR^n}\Xi_x(s,x,r,\mu)\Th_{xy}^0(t,r,\ti x,\mu,y)\big[\ti b(r,\mu;\th,\th^0)
+\ti g_{p^0}^0(t,r,\ti x,\mu,y;\th,\th^0,\Th^0)\big] d\mu dr.
\end{align}
Applying the arguments employed in the proof of \autoref{lem:HJB-Th-x}, we have
\bel{V-y-estimate}
\|\Th^0_y(\cd)\|_{L^\i}+\|\Th^0_{xy}(\cd)\|_{L^\i}\les K\big(1+\|h^0(\cd)\|_{C^{0,0,1,1}}\big).
\ee
By continuing the above argument, we get
\begin{align}
\nn\Th_{yy}^0(t,s,\ti x,x,y)&=\int_{\dbR^n}\Xi(s,x,T,\mu)h^0_{yy}(t,\ti x,\mu,y)d\mu+\int_s^T\int_{\dbR^n}\Xi(s,x,r,\mu)\\
&\q\times\Big\{\Th_{xyy}^0(t,r,\ti x,\mu,y)\big[\ti b(r,\mu;\th,\th^0)
+\ti g_{p^0}^0(t,r,\ti x,\mu,y;\th,\th^0,\Th^0)\big]\nn\\
&\q+\big\lan\ti g_{p^0p^0}^0(t,r,\ti x,\mu,y;\th,\th^0,\Th^0)\Th_{xy}^0(t,r,\ti x,\mu,y),\,
 \Th_{xy}^0(t,r,\ti x,\mu,y)\big\ran \Big\}d\mu dr.\nn\\
\nn\Th_{xyy}^0(t,s,\ti x,x,y)&=\int_{\dbR^n}
\big[\Xi(s,x,T,\mu)h^0_{xyy}(t,\ti x,\mu,y)-\Xi(s,x,T,\mu)\rho(s,x,T,\mu)h^0_{yy}(t,\ti x,\mu,y)\big]d\mu\\
\nn&\q+\int_s^T\int_{\dbR^n}\Xi_x(s,x,r,\mu)
\Big\{\Th_{xyy}^0(t,r,\ti x,\mu,y)\big[\ti b(r,\mu;\th,\th^0)+\ti g_{p^0}^0(t,r,\ti x,\mu,y;\th,\th^0,\Th^0)\big]\nn\\
&\q+\big\lan\ti g_{p^0p^0}^0(t,r,\ti x,\mu,y;\th,\th^0,\Th^0)\Th_{xy}^0(t,r,\ti x,\mu,y),\,
 \Th_{xy}^0(t,r,\ti x,\mu,y)\big\ran \Big\}d\mu dr.\nn
\end{align}
Note from \rf{V-y-estimate} that
$$
\big|\big\lan\ti g_{p^0p^0}^0(t,r,\ti x,\mu,y;\th,\th^0,\Th^0)\Th_{xy}^0(t,r,\ti x,\mu,y),\,
 \Th_{xy}^0(t,r,\ti x,\mu,y)\big\ran\big|\les K\big(1+\|h^0(\cd)\|^2_{C^{0,0,1,1}}\big).
$$
Then by the arguments employed in the proof of \autoref{lem:HJB-Th-x} again, we get
\bel{V-C(1,2)-estimate}
\|\Th^0_{yy}(\cd)\|_{L^\i}+\|\Th^0_{xyy}(\cd)\|_{L^\i}\les  K\big(1+\|h^0(\cd)\|^2_{C^{0,0,1,2}}\big).
\ee
For any $\ti x_1,\ti x_2\in\dbR^n$, denote
$$
\d \Th^0(t,s,x,y)=\Th^{0,1}(t,s,x,y)-\Th^{0,2}(t,s,x,y)\q \hbox{with}\q \Th^{0,i}(t,s,x,y)=\Th^0(t,s,\ti x_i,x,y),\q i=1,2.
$$
Then, we have
\begin{align*}
\d\Th^0(t,s,x,y)&=\int_{\dbR^n}\Xi(s,x,T,\mu)[h^0(t,\ti x_1,\mu,y)-h^0(t,\ti x_2,\mu,y)]d\mu\nn\\
&\q+\int_s^T\int_{\dbR^n}\Xi(s,x,r,\mu)\[\d\Th^0_{x}(t,r,\mu,y)\ti b(r,\mu;\th,\th^0)\nn\\
&\q+\ti g^0(t,r,\ti x_1,\mu,y;\th,\th^0,\Th^{0,1})-\ti g^0(t,r,\ti x_2,\mu,y;\th,\th^0,\Th^{0,2})\]d\mu dr,\\
\d\Th_y^0(t,s,x,y)&=\int_{\dbR^n}\Xi(s,x,T,\mu)[h_y^0(t,\ti x_1,\mu,y)-h_y^0(t,\ti x_2,\mu,y)]d\mu\nn\\
&\q+\int_s^T\int_{\dbR^n}\Xi(s,x,r,\mu)
\Big\{\d\Th^0_{xy}(t,r,\mu,y)\big[\ti b(r,\mu;\th,\th^0)+\ti g_{p^0}^0(t,r,\ti x_1,\mu,y;\th,\th^0,\Th^{0,1})\big]\\
&\q+\Th^{0,2}_{xy}(t,r,\mu,y)\big[\ti g_{p^0}^0(t,r,\ti x_1,\mu,y;\th,\th^0,\Th^{0,1})
-\ti g_{p^0}^0(t,r,\ti x_2,\mu,y;\th,\th^0,\Th^{0,2})\big]\Big\}d\mu dr,\nn
\end{align*}
and
\begin{align*}
\d\Th_x^0(t,s,x,y)&=\int_{\dbR^n}\Xi(s,x,T,\mu)[h_x^0(t,\ti x_1,\mu,y)-h_x^0(t,\ti x_2,\mu,y)]d\mu\nn\\
&\q+\int_{\dbR^n}\Xi(s,x,T,\mu)\rho(s,x,T,\mu)[h^0(t,\ti x_1,\mu,y)-h^0(t,\ti x_2,\mu,y)]d\mu\\
&\q+\int_s^T\int_{\dbR^n}\Xi_x(s,x,r,\mu)\[\d\Th^0_{x}(t,r,\mu,y)\ti b(r,\mu;\th,\th^0)\\
&\q+\ti g^0(t,r,\ti x_1,\mu,y;\th,\th^0,\Th^{0,1})-\ti g^0(t,r,\ti x_2,\mu,y;\th,\th^0,\Th^{0,2})\]d\mu dr,\\
\d\Th_{xy}^0(t,s,x,y)&=\int_{\dbR^n}\Xi(s,x,T,\mu)[h_{xy}^0(t,\ti x_1,\mu,y)-h_{xy}^0(t,\ti x_2,\mu,y)]d\mu\\
&\q+\int_{\dbR^n}\Xi(s,x,T,\mu)\rho(s,x,T,\mu)[h_y^0(t,\ti x_1,\mu,y)-h_y^0(t,\ti x_2,\mu,y)]d\mu\\
&\q+\int_s^T\int_{\dbR^n}\Xi_x(s,x,r,\mu)
\Big\{\d\Th^0_{xy}(t,r,\mu,y)\big[\ti b(r,\mu;\th,\th^0)+\ti g_{p^0}^0(t,r,\ti x_1,\mu,y;\th,\th^0,\Th^{0,1})\big]\\
&\q+\Th^{0,2}_{xy}(t,r,\mu,y)\big[\ti g_{p^0}^0(t,r,\ti x_1,\mu,y;\th,\th^0,\Th^{0,1})
-\ti g_{p^0}^0(t,r,\ti x_2,\mu,y;\th,\th^0,\Th^{0,2})\big]\Big\}d\mu dr.
\end{align*}
Using the estimates \rf{G-G-x-est} and \rf{V-y-estimate}, we get
\begin{align*}
|\d\Th^0(t,s,x,y)|&\les K\int_{\dbR^n}{1\over (r-s)^{{n\over 2}}}e^{{-\l|x-\mu|^2\over 4(r-s)}}
   \|h^0(\cd)\|_{C^{0,\a,0,0}}d\mu|\ti x_1-\ti x_2|^\a\\
&\q+K\int_s^T\int_{\dbR^n} {1\over (r-s)^{n\over 2}}e^{{-\l|x-\mu|^2\over 4(r-s)}}\[|\ti x_1-\ti x_2|^\a+|\d\Th^0_x(t,r,\mu,y)|\]d\mu dr,\\
| \d\Th_y^0(t,s,x,y)|&\les K\int_{\dbR^n}{1\over (r-s)^{{n\over 2}}}e^{{-\l|x-\mu|^2\over 4(r-s)}}
  \|h^0(\cd)\|_{C^{0,\a,0,1}}d\mu|\ti x_1-\ti x_2|^\a\\
&\q+K\int_s^T\int_{\dbR^n} {1\over (r-s)^{n\over 2}}e^{{-\l|x-\mu|^2\over 4(r-s)}}\Big\{|\d\Th^0_{xy}(t,r,\mu,y)|\\
&\q+\big[|\ti x_1-\ti x_2|^\a+|\d\Th^0_x(t,r,\mu,y)|\big]|\Th^{0,2}_{xy}(t,r,\mu,y)|\Big\}d\mu dr,
\end{align*}
and
\begin{align*}
|\d\Th_x^0(t,s,x,y)|&\les K\int_{\dbR^n}{1\over (r-s)^{{n\over 2}}}e^{{-\l|x-\mu|^2\over 4(r-s)}}
   \[\|h^0(\cd)\|_{C^{0,\a,1,0}}+\(1+{|x-\mu|^2\over (r-s)}\)\|h^0(\cd)\|_{C^{0,\a,0,0}}\]d\mu|\ti x_1-\ti x_2|^\a\\
&\q+K\int_s^T\int_{\dbR^n} {1\over (r-s)^{n+1\over 2}}e^{{-\l|x-\mu|^2\over 4(r-s)}}\[|\ti x_1-\ti x_2|^\a+|\d\Th^0_x(t,r,\mu,y)|\]d\mu dr,\\
|\d\Th_{xy}^0(t,s,x,y)|&\les K\int_{\dbR^n}{1\over (r-s)^{{n\over 2}}}e^{{-\l|x-\mu|^2\over 4(r-s)}}
   \[\|h^0(\cd)\|_{C^{0,\a,1,1}}+\(1+{|x-\mu|^2\over (r-s)}\)\|h^0(\cd)\|_{C^{0,\a,0,1}}\]d\mu|\ti x_1-\ti x_2|^\a\nn\\
&\q+K\int_s^T\int_{\dbR^n} {1\over (r-s)^{n+1\over 2}}e^{{-\l|x-\mu|^2\over 4(r-s)}}\Big\{|\d\Th^0_{xy}(t,r,\mu,y)|\\
&\q+\big[|\ti x_1-\ti x_2|^\a+|\d\Th^0_x(t,r,\mu,y)|\big]|\Th^{0,2}_{xy}(t,r,\mu,y)|\Big\}d\mu dr.
\end{align*}
Note from \rf{V-y-estimate} that  $\Th^{0,2}_{xy}(\cd)$  is globally bounded,
and the estimate  of $\|\Th^{0,2}_{xy}(\cd)\|_{L^\i}$ is independent of $(\th(\cd),\th^0(\cd))$.
Thus, by the definition of the seminorm $\|\cd \|_\a$ and the Gr\"{o}nwall's inequality, we obtain
\begin{align}
&\sup_{(t,s,x,y)}\left[\|\Th^0(t,s,\cd\,,x,y)\|_{\a}+\|\Th_{x}^0(t,s,\cd\,,x,y)\|_{\a}
+\|\Th_{y}^0(t,s,\cd\,,x,y)\|_{\a}+\|\Th_{xy}^0(t,s,\cd\,,x,y)\|_{\a}\right]\nn\\
&\qq\les K\big(1+\|h^0(\cd)\|_{C^{0,\a,1,1}}\big).\nn
\end{align}
By continuing the above arguments, we can also have
\begin{align}
\sup_{(t,s,x,y)}\big[\|\Th_{yy}^0(t,s,\cd\,,x,y)\|_{\a}
+\|\Th_{xyy}^0(t,s,\cd\,,x,y)\|_{\a}\big]\les K\big(1+\|h^0(\cd)\|^2_{C^{0,\a,1,2}}\big).\nn
\end{align}
Similarly, for the parameter $t$, we have
\begin{align}
&\sup_{(s,\ti x,x,y)}\[\|\Th^0(\cd\,,s,\ti x,x,y)\|_{\a\over 2}+\|\Th_{x}^0(\cd,s,\ti x,x,y)\|_{\a\over 2}
+\|\Th_{y}^0(\cd,s,\ti x,x,y)\|_{\a\over 2}+\|\Th_{xy}^0(\cd\,,s,\ti x,x,y)\|_{\a\over 2}\nn\\
&\qq\q+\|\Th_{yy}^0(\cd\,,s,\ti x,x,y)\|_{\a\over 2}+\|\Th_{xyy}^0(\cd\,,s,\ti x,x,y)\|_{\a\over 2}\]\les K\big(1+\|h^0(\cd)\|_{C^{{\a\over 2},0,1,2}}\big).\nn
\end{align}
Combining the above together,
we get the estimate \rf{lem:HJB-par-estimate} immediately.
\end{proof}

\begin{remark}
Note that \rf{lem:HJB-par-estimate} is a global prior estimate and the constant $\k>0$ is independent of the choice of
$(\th(\cd),\th^0(\cd))$. Thus, we could always assume that
\bel{lem:v-par-estimate}
\|\th(\cd)\|_{C^{0,1}}+\|\th^0(\cd)\|_{C^{{\a\over 2},0,\a,1,2}}
\les\k\big[1+\|h^0(\cd)\|^2_{C^{{\a\over 2},\a,1,2}}\big],
\ee
where $\k>0$ is same as that given in \autoref{lem:HJB-par}.
\end{remark}

\ms

Next, we are going to establish the $C^{1+\a}$-norm estimate for $\Th^0(t,s,\ti x,\cd,y)$ and $\Th(s,\cd)$.
To achieve this, we need to make some preparations.
By making the transforms $x-\mu=(\sqrt{T-s})\ti \mu$  and $x-\mu=(\sqrt{r-s})\ti \mu$
in the first integral term and the second integral term of \rf{proof-step11-Th-x}--\rf{proof-step1-Vx},
respectively,  we have
\begin{align}
\nn\Th_x(s,x)&=\int_{\dbR^n}\ti\Xi(s,x,T,x-\sqrt{T-s}\ti\mu)\big[h_x(x-\sqrt{T-s}\ti\mu)-\rho(s,x,T,x-\sqrt{T-s}\ti\mu)h(x-\sqrt{T-s}\ti\mu)\big]d\ti\mu\\
\nn&\q+\int_s^T\int_{\dbR^n}\ti\Xi_x(s,x,r,x-\sqrt{r-s}\ti\mu)
\[\Th_{x}(r,x-\sqrt{r-s}\ti\mu)\ti b(r,x-\sqrt{r-s}\ti\mu;\th,\th^0)\nn\\
&\q+\ti g(r,x-\sqrt{r-s}\ti\mu;\th,\th^0)\]d\ti\mu dr,\label{proof-step1-Th-x-change}
\end{align}
and
\begin{align}
\nn\Th^0_x(t,s,\ti x,x,y)&=\int_{\dbR^n}\ti \Xi(s,x,T,x-\sqrt{T-s}\ti\mu)\[h^0_x(t,\ti x,x-\sqrt{T-s}\ti\mu,y)\\
\nn&\q-\rho(s,x,T,x-\sqrt{T-s}\ti\mu)h^0(t,\ti x,x-\sqrt{T-s}\ti\mu,y)\]d\ti\mu
\nn\\
&\q+\int_s^T\int_{\dbR^n}\ti\Xi_x(s,x,r,x-\sqrt{r-s}\ti\mu)\[\Th^0_{x}(t,r,\ti x,x-\sqrt{r-s}\ti\mu,y)\ti b(r,x-\sqrt{r-s}\ti\mu;\th,\th^0)\nn\\
&\q+\ti g^0(t,r,\ti x,x-\sqrt{r-s}\ti\mu;\th,\th^0,\Th^0)\]d\ti\mu dr,\label{proof-step1-V-x-change}
\end{align}
where
\begin{align}
\nn\wt \Xi(s,x,r,x-\sqrt{r-s}\ti\mu)&:= -\Xi(s,x,r,x-\sqrt{r-s}\ti\mu)(r-s)^{n\over 2}\\
\nn&=-{1 \over (4\pi)^{n\over 2}(\det[a(r,x-\sqrt{r-s}\ti\mu)])^{1\over 2}}
e^{-{\lan a(r,x-\sqrt{r-s}\ti\mu)^{-1}\ti\mu,\ti\mu\ran\over 4}},\\
\nn\rho(s,x,r,x-\sqrt{r-s}\ti\mu)&={(\det[a(r,x-\sqrt{r-s}\ti\mu)])_\mu\over 2\det[a(r,x-\sqrt{r-s}\ti\mu)]}
+{\lan[a(r,x-\sqrt{r-s}\ti\mu)^{-1}]_\mu\ti\mu,\,\ti\mu\ran\over 4},\\
\nn\wt\Xi_x(s,x,r,x-\sqrt{r-s}\ti\mu)&:=-\Xi_x(s,x,r,x-\sqrt{r-s}\ti\mu)(r-s)^{n\over 2}\\
\label{ti-G}&={1 \over (4\pi)^{n\over 2}(\det[a(r,x-\sqrt{r-s}\ti\mu)])^{1\over 2}}
e^{-{\lan a(r,x-\sqrt{r-s}\ti\mu)^{-1}\ti\mu,\ti\mu\ran\over 4}}{a(r,x-\sqrt{r-s}\ti\mu)^{-1}\over 2\sqrt{r-s}}\ti\mu.
\end{align}
By some straightforward calculations, it is clearly seen that
\begin{align}
\nn&|\wt\Xi(s,x,r,x-\sqrt{r-s}\ti\mu)|\les Ke^{-\l|\ti\mu|^2},\q &&|\rho(s,x,r,x-\sqrt{r-s}\ti\mu)|\les K(1+|\ti\mu|^2),\\
\nn& |\wt\Xi_x(s,x,r,x-\sqrt{r-s}\ti\mu)|\les {K\over \sqrt{r-s}}e^{-\l|\ti\mu|^2},
\q&&\|\wt\Xi(s,\cd,r,\cd-\sqrt{r-s}\ti\mu)\|_\a\les Ke^{-\l|\ti\mu|^2},\\
 &
\|\rho(s,\cd,r,\cd-\sqrt{r-s}\ti\mu)\|_\a\les K(1+|\ti\mu|^2),\q &&\|\wt\Xi_x(s,\cd,r,\cd-\sqrt{r-s}\ti\mu)\|_\a\les {K\over \sqrt{r-s}}e^{-\l|\ti\mu|^2},
\label{Th-x-alpha}
\end{align}
for some $0<\l<\l_0$.

\begin{proposition}\label{Prop:Holder}
There exist two constants $0<\e\les T$ and $\bar\k>0$ such for any $\th(\cd)\in C^{{\a\over 2},1+\a}$
and $\th^0(\cd)\in C^{{\a\over 2},{\a\over 2},\a,1+\a,2}$ with
\bel{Prop:Holder1}
\|\th_x(\cd)\|_{C^{0,\a}([T-\e,T])}
+\|\th^0_x(\cd)\|_{C^{0,0,0,\a,0}([T-\e,T])}\les 2\bar\k\big[1+\|h(\cd)\|_{C^{1+\a}}+\|h^0(\cd)\|_{C^{0,0,1+\a,0}}\big],\ee
the unique solution $(\Th(\cd),\Th^0(\cd))$ of PDE \rf{Th-v} satisfies
\begin{align}
\|\Th_x(\cd)\|_{C^{0,\a}([T-\e,T])}
+\|\Th^0_x(\cd)\|_{C^{0,0,0,\a,0}([T-\e,T])}\les 2\bar\k\big[1+\|h(\cd)\|_{C^{1+\a}}+\|h^0(\cd)\|_{C^{0,0,1+\a,0}}\big].
\label{Prop:Holder2}
\end{align}
Moreover, there exists a constant $\h\k>0$, which depends on $\bar\k$, such that
\begin{align}
\|\Th^0_{xy}(\cd)\|_{C^{0,0,0,\a,0}([T-\e,T])}\les \h\k\big[1+\|h(\cd)\|_{C^{1+\a}}+\|h^0(\cd)\|_{C^{0,0,1+\a,1}}\big].
\label{Prop:Holder3}
\end{align}
\end{proposition}

\begin{proof}
For any $x_1,\,x_2\in\dbR^n$, from \rf{proof-step1-Th-x-change} and then by the estimate \rf{Th-x-alpha}, we have
\begin{align}
\nn |\Th_x(s,x_1)-\Th_x(s,x_2)|&\les \int_{\dbR^n} Ke^{-\l|\ti\mu|^2}(1+|\ti\mu|^2) d\ti\mu\|h(\cd)\|_{C^{1}}|x_1-x_2|^\a\\
\nn&\q
+\int_s^T\int_{\dbR^n} {K\over \sqrt{r-s}}e^{-\l|\ti\mu|^2}d\ti\mu dr\big[1+\|\Th_x(\cd)\|_{L^\i}\big]|x_1-x_2|^\a \\
\nn&\q
+\int_{\dbR^n} Ke^{-\l|\ti\mu|^2}\big|h_x(x_1-\sqrt{T-s}\ti\mu)-h_x(x_2-\sqrt{T-s}\ti\mu)\big|d\ti\mu \\
\nn&\q
+\int_{\dbR^n} Ke^{-\l|\ti\mu|^2}(1+|\ti\mu|^2)\big|h(x_1-\sqrt{T-s}\ti\mu)-h(x_2-\sqrt{T-s}\ti\mu)\big|d\ti\mu \\
\nn&\q+\int_s^T\int_{\dbR^n}{K\over \sqrt{r-s}}e^{-\l|\ti\mu|^2}
\Big[\big| \Th_x(r,x_1-\sqrt{r-s}\ti\mu)-\Th_x(r,x_2-\sqrt{r-s}\ti\mu)\big|\\
\nn&\qq+\big| \ti b(r,x_1-\sqrt{r-s}\ti\mu;\th,\th^0)-\ti b(r,x_2-\sqrt{r-s}\ti\mu;\th,\th^0)\big|\|\Th_x(\cd)\|_{L^\i}\\
&\qq+\big| \ti g(r,x_1-\sqrt{r-s}\ti\mu;\th,\th^0)-\ti g(r,x_2-\sqrt{r-s}\ti\mu;\th,\th^0)\big|\Big]d\ti\mu dr.
\label{proof-step1-Th-x-sub}
\end{align}
Note that for $\ti\f(\cd)=\ti b(\cd),\ti g(\cd)$,
\begin{align}
&\ti\f(r,x_1-\sqrt{r-s}\ti\mu;\th,\th^0)-\ti\f(r,x_2-\sqrt{r-s}\ti\mu;\th,\th^0)\nn\\
&\q=\ti \f\Big(r,x_1-\sqrt{r-s}\ti\mu,\th(r,x_1-\sqrt{r-s}\ti\mu),\th_x(r,x_1-\sqrt{r-s}\ti\mu),\nn\\
&\qq\q\, \th^0(r,x_1-\sqrt{r-s}\ti\mu,r,x_1-\sqrt{r-s}\ti\mu,\th(r,x_1-\sqrt{r-s}\ti\mu)),\nn\\
&\qq\q\, \th_x^0(r,x_1-\sqrt{r-s}\ti\mu,r,x_1-\sqrt{r-s}\ti\mu,\th(r,x_1-\sqrt{r-s}\ti\mu)),\nn\\
&\qq\q\, \th_y^0(r,x_1-\sqrt{r-s}\ti\mu,r,x_1-\sqrt{r-s}\ti\mu,\th(r,x_1-\sqrt{r-s}\ti\mu)) \Big)\nn\\
&\qq-\ti \f\Big(r,x_2-\sqrt{r-s}\ti\mu,\th(r,x_2-\sqrt{r-s}\ti\mu),\th_x(r,x_2-\sqrt{r-s}\ti\mu),\nn\\
&\qq\q\, \th^0(r,x_2-\sqrt{r-s}\ti\mu,r,x_2-\sqrt{r-s}\ti\mu,\th(r,x_2-\sqrt{r-s}\ti\mu)),\nn\\
&\qq\q\, \th_x^0(r,x_2-\sqrt{r-s}\ti\mu,r,x_2-\sqrt{r-s}\ti\mu,\th(r,x_2-\sqrt{r-s}\ti\mu)),\nn\\
&\qq\q\, \th_y^0(r,x_2-\sqrt{r-s}\ti\mu,r,x_2-\sqrt{r-s}\ti\mu,\th(r,x_2-\sqrt{r-s}\ti\mu)) \Big).\nn
\end{align}
Then by \autoref{lem:HJB-par} (or \rf{lem:v-par-estimate}), we get
\begin{align}
&|\ti\f(r,x_1-\sqrt{r-s}\ti\mu;\th,\th^0)-\ti\f(r,x_2-\sqrt{r-s}\ti\mu;\th,\th^0)|\nn\\
&\q\les K\big(1+\|\th_x(r,\cd)\|_{C^\a}+\|\th^0_x(r,\cd,r,\cd,\cd)\|_{C^{0,\a,0}}\big)|x_1-x_2|^\a,\nn
\end{align}
where $K>0$ depends on $h(\cd)$ and $h^0(\cd)$.
Substituting the above into \rf{proof-step1-Th-x-sub} and then by \autoref{lem:HJB-Th-x}, we have
\begin{align}
\nn|\Th_x(s,x_1)-\Th_x(s,x_2)|&\les K\big(1+\|h(\cd)\|_{C^{1+\a}}\big)|x_1-x_2|^\a\\
\nn&\q+\int_s^T{K\over \sqrt{r-s}}
\[\|\Th_x(r,\cd)\|_\a+\|\th_x(r,\cd)\|_{C^{\a}}+\|\th^0_x(r,\cd,r,\cd,\cd)\|_{C^{0,\a,0}}\] dr|x_1-x_2|^\a,
\end{align}
which implies that
\bel{Th-x-holder}
\|\Th_x(s,\cd)\|_{\a}\les K\big(1+\|h(\cd)\|_{C^{1+\a}}\big)
+\int_s^T{K\over \sqrt{r-s}}\[\| \Th_x(r,\cd)\|_{\a}+\|\th_x(r,\cd)\|_{C^{\a}}+\|\th^0_x(r,\cd,r,\cd,\cd)\|_{C^{0,\a,0}}\]dr.\ee
By the same argument as the above (noting \rf{proof-step1-V-x-change}), we also have
\begin{align}
\nn \|\Th^0_x(t,s,\ti x,\cd,y)\|_{\a}&\les K\big(1+\|h^0(t,\ti x,\cd,y)\|_{C^{1+\a}}\big)\\
&\q+\int_s^T{K\over \sqrt{r-s}}
\[ \|\Th^0_x(t,s,\ti x,\cd,y)\|_{\a}+\|\th_x(r,\cd)\|_{\a}+\|\th^0_x(r,\cd,r,\cd,\cd)\|_{C^{0,\a,0}}\]dr,\label{V-x-holder}\\
\nn \|\Th^0_{xy}(t,s,\ti x,\cd,y)\|_{\a}&\les K\big(1+\|h^0(t,\ti x,\cd,\cd)\|_{C^{1+\a,1}}\big)\\
&\q+\int_s^T{K\over \sqrt{r-s}}
\[ \|\Th^0_{xy}(t,s,\ti x,\cd,y)\|_{\a}+\|\th_x(r,\cd)\|_{\a}+\|\th^0_x(r,\cd,r,\cd,\cd)\|_{C^{0,\a,0}}\]dr.
\label{V-y-holder}
\end{align}
Combining \rf{Th-x-holder} and \rf{V-x-holder} yields that
\begin{align}
\nn& \|\Th_x(\cd)\|_{C^{0,\a}([T-\e,T])}+\|\Th^0_x(\cd)\|_{C^{0,0,0,\a,0}([T-\e,T])}\nn\\
&\q\les \bar\k\big[1+\|h(\cd)\|_{C^{1+\a}}+\|h^0(\cd)\|_{C^{0,0,1+\a,0}}\big]
+\bar\k\sqrt{\e}\big[ \|\Th_x(\cd)\|_{C^{0,\a}([T-\e,T])}\nn\\
&\qq+\|\Th^0_x(\cd)\|_{C^{0,0,0,\a,0}([T-\e,T])}+ \|\th_x(\cd)\|_{C^{0,\a}([T-\e,T])}
+\|\th^0_x(\cd)\|_{C^{0,0,0,\a,0}([T-\e,T])}\big],
\label{V-x-holder1}
\end{align}
where $\bar\k>\k$, only depending on $( h(\cd),\ti g(\cd), h^0(\cd),\ti g^0(\cd))$, is a fixed constant.
Let $\e$ be small enough such that $\bar\k\sqrt{\e}\les{1\over 4}$ and
\bel{v-x-esti}
 \|\th_x(\cd)\|_{C^{0,\a}([T-\e,T])}+\|\th^0_x(\cd)\|_{C^{0,0,0,\a,0}([T-\e,T])}
\les 2\bar\k\big[1+\|h(\cd)\|_{C^{1+\a}}+\| h^0(\cd)\|_{C^{0,0,1+\a,0}}\big].
\ee
Then from \rf{V-x-holder1}, we get
\begin{align}
 \|\Th_x(\cd)\|_{C^{0,\a}([T-\e,T])}+\|\Th^0_x(\cd)\|_{C^{0,0,0,\a,0}([T-\e,T])}
\les 2\bar\k\big[1+\|h(\cd)\|_{C^{1+\a}}+\| h^0(\cd)\|_{C^{0,0,1+\a,0}}\big].\nn
\end{align}
Substituting \rf{v-x-esti}  into \rf{V-y-holder} also yields \rf{Prop:Holder3} immediately.
The proof is complete.
\end{proof}

The following gives  the $C^{\a\over 2}$-norm estimate for $\Th^0(t,\cd,\ti x,x,y)$ and $\Th(\cd,x)$.

\begin{proposition}\label{Prop:Holder-t}
There exists a constant $\check{\k}>0$ such for any  $\th(\cd)\in C^{{\a\over 2},1+\a}$
and $\th^0(\cd)\in C^{{\a\over 2},{\a\over 2},\a,1+\a,2}$ satisfying \rf{Prop:Holder1},
the unique solution of  PDE \rf{Th-v} satisfies
\begin{align}
&\|\Th(\cd,x)\|_{\a\over2}+\|\Th_x(\cd,x)\|_{\a\over2}+\|\Th^0(t,\cd,\ti x,x,y)\|_{\a\over 2}
+\|\Th_x^0(t,\cd,\ti x,x,y)\|_{\a\over 2}+\|\Th_y^0(t,\cd,\ti x,x,y)\|_{\a\over 2}\nn\\
&+\|\Th_{xy}^0(t,\cd,\ti x,x,y)\|_{\a\over 2}\,\les\, \check{\k}\big(1+\|h(\cd)\|_{C^{1+\a}}+\| h^0(\cd)\|_{C^{0,0,1+\a,0}}\big),
 \q\forall (t,\ti x,x,y)\in[0,T]\times\dbR^n\times\dbR^n\times\dbR,\label{Prop:Holder-t1}
\end{align}
on the time interval $[T-\e,T]$.
\end{proposition}

\begin{proof}
For any $T-\e\les s_1\les s_2\les T$, by \rf{proof-step1-Th-x-change} we have
\begin{align}
\nn |\Th_x(s_1,x)-\Th_x(s_2,x)|&\les \int_{\dbR^n}\[|\ti\Xi(s_1,x,T,x-\sqrt{T-s_1}\ti\mu)|\|h_x(\cd)\|_\a|\ti\mu|^\a|s_1-s_2|^{\a\over 2}\nn\\
&\q +\big|\ti\Xi(s_1,x,T,x-\sqrt{T-s_1}\ti\mu)-\ti\Xi(s_2,x,T,x-\sqrt{T-s_2}\ti\mu)\big|\|h_x(\cd)\|_{L^\i}\nn\\
&\q+|\ti\Xi(s_1,x,T,x-\sqrt{T-s_1}\ti\mu)\rho(s_1,x,T,x-\sqrt{T-s_1}\ti\mu)|\|h(\cd)\|_\a|\ti\mu|^\a|s_1-s_2|^{\a\over 2}\nn\\
\nn&\q+\big|\ti\Xi(s_1,x,T,x-\sqrt{T-s_1}\ti\mu)\rho(s_1,x,T,x-\sqrt{T-s_1}\ti\mu)\nn\\
\nn&\qq-\ti\Xi(s_2,x,T,x-\sqrt{T-s_2}\ti\mu)\rho(s_2,x,T,x-\sqrt{T-s_2}\ti\mu)\big|\|h(\cd)\|_{L^\i}\]d\ti\mu\\
\nn&\q+\int_{s_1}^{s_2}\int_{\dbR^n}|\ti\Xi_x(s_1,x,r,x-\sqrt{r-s_1}\ti\mu)|\|\Th_{x}(\cd)\ti b(\cd;\th,\th^0)+\ti g\big(\cd;\th,\th^0\big)\|_{L^\i}d\ti\mu dr\\
\nn&\q+\int_{s_2}^{T}\int_{\dbR^n}\big|\ti\Xi_x(s_1,x,r,x-\sqrt{r-s_1}\ti\mu)-\ti\Xi_x(s_2,x,r,x-\sqrt{r-s_2}\ti\mu)\big|\nn\\
\nn&\qq\times\|\Th_{x}(\cd)\ti b(\cd;\th,\th^0)+\ti g\big(\cd;\th,\th^0\big)\|_{L^\i}d\ti\mu dr\\
\nn&\q+\int_{s_2}^{T}\int_{\dbR^n}|\ti\Xi_x(s_1,x,r,x-\sqrt{r-s_1}\ti\mu)|\[\big
 |\Th_{x}(r,x-\sqrt{r-s_1}\ti\mu)\ti b(r,x-\sqrt{r-s_1}\ti\mu;\th,\th^0)\nn\\
&\qq -\Th_{x}(r,x-\sqrt{r-s_2}\ti\mu)\ti b(r,x-\sqrt{r-s_2}\ti\mu;\th,\th^0)\big|+\big|\ti g(r,x-\sqrt{r-s_1}\ti\mu;\th,\th^0)\nn\\
&\qq-\ti g(r,x-\sqrt{r-s_2}\ti\mu;\th,\th^0)\big|\]d\ti\mu dr.\label{Th-Holder-t}
\end{align}
Note that (recalling \rf{ti-G})
\begin{align}
\nn&\int_{s_2}^{T}\int_{\dbR^n}\big|\ti\Xi_x(s_1,x,r,x-\sqrt{r-s_1}\ti\mu)-\ti\Xi_x(s_2,x,r,x-\sqrt{r-s_2}\ti\mu)\big|d\ti\mu dr\\
\nn&\q\les K|s_1-s_2|^{\a\over 2}+K \int_{s_2}^{T}\Big|{1\over\sqrt{r-s_2}}-{1\over\sqrt{r-s_1}}\Big|dr\les K|s_1-s_2|^{\a\over 2}.
\end{align}
Then by \rf{lem:v-par-estimate}, \autoref{Prop:Holder} and \rf{Th-x-alpha},
from \rf{Th-Holder-t} we obtain
\begin{align}
|\Th_x(s_1,x)-\Th_x(s_2,x)|\les K\big(1+\|h(\cd)\|_{C^{1+\a}}+\|h^0(\cd)\|_{C^{0,0,1+\a,0}}\big)|s_1-s_2|^{\a\over 2},\nn\\
\forall T-\e\les s_1\les s_2\les T,\, x\in\dbR^n,\label{Prop:Holder-t-11}
\end{align}
where $K>0$ depends on  $\bar\k$.
Recall \rf{proof-step1-V-x-change} and note that
\begin{align}
\nn\Th^0_y(t,s,\ti x,x,y)&=\int_{\dbR^n}\ti \Xi(s,x,T,x-\sqrt{T-s}\ti\mu)h^0_y(t,\ti x,x-\sqrt{T-s}\ti\mu,y)d\ti\mu\\
\nn&\q+\int_s^T\int_{\dbR^n}\ti\Xi(s,x,r,x-\sqrt{r-s}\ti\mu)
\[\Th^0_{xy}(t,r,\ti x,x-\sqrt{r-s}\ti\mu,y)\ti b(r,x-\sqrt{r-s}\ti\mu;\th,\th^0)\nn\\
&\q+\ti g_{p^0}^0(t,r,\ti x,x-\sqrt{r-s}\ti\mu;\th,\th^0,\Th^0)\Th^0_{xy}(t,r,\ti x,x-\sqrt{r-s}\ti\mu,y)\]d\ti\mu dr,\nn
\end{align}
and
\begin{align}
\nn\Th^0_{xy}(t,s,\ti x,x,y)&=\int_{\dbR^n}\ti \Xi(s,x,T,x-\sqrt{T-s}\ti\mu\[h^0_{xy}(t,\ti x,x-\sqrt{T-s}\ti\mu,y)\nn\\
\nn&\q-\rho(s,x,T,x-\sqrt{T-s}\ti\mu)h^0_y(t,\ti x,x-\sqrt{T-s}\ti\mu,y)\]d\ti\mu\\
\nn&\q+\int_s^T\int_{\dbR^n}\ti\Xi_x(s,x,r,x-\sqrt{r-s}\ti\mu)
\[\Th^0_{xy}(t,r,\ti x,x-\sqrt{r-s}\ti\mu,y)\ti b(r,x-\sqrt{r-s}\ti\mu;\th,\th^0)\nn\\
&\q+\ti g_{p^0}^0(t,r,\ti x,x-\sqrt{r-s}\ti\mu;\th,\th^0,\Th^0)\Th^0_{xy}(t,r,\ti x,x-\sqrt{r-s}\ti\mu,y)\]d\ti\mu dr.\nn
\end{align}
By the same argument as the above, we also have
\begin{align}
&|\Th^0_x(t,s_1,\ti x,x,y)-\Th^0_x(t,s_2,\ti x,x,y)|\les K\big(1+\|h(\cd)\|_{C^{1+\a}}+\|h^0(\cd)\|_{C^{0,0,1+\a,0}}\big)|s_1-s_2|^{\a\over 2},\nn\\
&|\Th^0_y(t,s_1,\ti x,x,y)-\Th^0_y(t,s_2,\ti x,x,y)|+|\Th^0_{xy}(t,s_1,\ti x,x,y)-\Th^0_{xy}(t,s_2,\ti x,x,y)|\nn\\
&\q\les K\big(1+\|h(\cd)\|_{C^{1+\a}}+\|h^0(\cd)\|_{C^{0,0,1+\a,1}}\big)|s_1-s_2|^{\a\over 2},\nn\\
&\qq\qq\qq\qq\qq\forall T-\e\les s_1\les s_2\les T,\, (t,y,\ti x,x)\in[0,T]\times\dbR\times\dbR^n\times\dbR^n.\label{V-x-y-s-holder}
\end{align}
The rest can be obtained easily.
\end{proof}

\begin{remark}
Note that the constant $\check{\k}>0$ in \autoref{Prop:Holder-t} is independent of
the ${\a\over 2}$-H\"{o}lder norm of $(\th(\cd,\cd),\th^0(\cd,\cd,\cd,\cd,\cd))$ with respect to the second argument $s$.
We assume that
\begin{align}
&\|\th(\cd,x)\|_{\a\over2}+\|\th_x(\cd,x)\|_{\a\over2}+\|\th^0(t,\cd,\ti x,x,y)\|_{\a\over 2}
+\|\th_x^0(t,\cd,\ti x,x,y)\|_{\a\over 2}+\|\th_y^0(t,\cd,\ti x,x,y)\|_{\a\over 2}\nn\\
&+\|\th_{xy}^0(t,\cd,\ti x,x,y)\|_{\a\over 2}
\les \check{\k}\big[1+\|h(\cd)\|_{C^{1+\a}}+\| h^0(\cd)\|_{C^{0,0,1+\a,0}}\big],\nn\\
&\qq\qq\qq\qq\qq\qq\qq\qq \forall (t,\ti x,x,y)\in[0,T]\times\dbR^n\times\dbR^n\times\dbR,\label{Prop:Holder-t-v}
\end{align}
on the time interval $[T-\e,T]$.
\end{remark}

The following is concerned with the local solvability of the equilibrium HJB equation \rf{HJB-si-ti}.

\begin{proposition}
There exists a constant $\bar\e\in(0,\e]$ such that the equilibrium HJB equation \rf{HJB-si-ti}
admits a unique classical solution on the time interval $[T-\bar\e,T]$.
\end{proposition}

\begin{proof}
Denote
$$
\cB_\e=\big\{ (\th(\cd),\th^0(\cd))\in  C^{{\a\over 2},1+\a}\times C^{{\a\over 2},{\a\over 2},\a,1+\a,2}~\big|~
 (\th(\cd),\th^0(\cd)) \hbox{ satisfies } \rf{lem:v-par-estimate}, \rf{Prop:Holder1} \hbox{ and } \rf{Prop:Holder-t-v}\big\}.
$$
For any $(\th(\cd),\th^0(\cd))\in\cB_\e$,  by \autoref{lem:HJB-linear},
PDE \rf{Th-v} admits a unique classical solution $(\Th(\cd),\Th^0(\cd))$.
Moreover, from  \autoref{lem:HJB-Th-x}, \autoref{lem:HJB-par}, \autoref{Prop:Holder} and \autoref{Prop:Holder-t},
we know that $(\Th(\cd),\Th^0(\cd))\in\cB_\e$.
Thus, the mapping $\G(\cd,\cd):\cB_\e\to\cB_\e$, given by
$$
\G(\th(\cd),\th^0(\cd))=(\Th(\cd),\Th^0(\cd)),
$$
is well-defined. For any $(\th^i(\cd),\th^{0,i}(\cd))\in \cB_\e$ ($i=1,2$), let
$$
(\Th^i(\cd),\Th^{0,i}(\cd))=\G(\th^i(\cd),\th^{0,i}(\cd)),\q i=1,2.
$$
Denote
\begin{align}
&\d\th(s,x)=\th^1(s,x)-\th^2(s,x),\q \d\th^0(t,s,\ti x, x,y)=\th^{0,1}(t,s,\ti x, x,y)-\th^{0,2}(t,s,\ti x, x,y),\nn\\
&\d\Th(s,x)=\Th^1(s,x)-\Th^2(s,x),\q \d\Th^0(t,s,\ti x, x,y)=\Th^{0,1}(t,s,\ti x, x,y)-\Th^{0,2}(t,s,\ti x, x,y),\nn\\
&\d\ti \f(s,x)=\ti \f(s,x;\th^1,\th^{0,1})-\ti \f(s,x;\th^2,\th^{0,2}),\q\hbox{for}\q \f(\cd)=b(\cd),g(\cd),\nn\\
&\d\ti g^0(t,s,\ti x,x,y)=\ti g^0(t,s,\ti x,x, y;\th^1,\th^{0,1},\Th^{0,1})-\ti g^0(t,s,\ti x,x, y;\th^1,\th^{0,2},\Th^{0,2}).\nn
\end{align}
We hope to show that
\bel{Contraction}
\|\d\Th(\cd)\|_{C^{{\a\over 2},1+\a}}+\|\d\Th^0(\cd)\|_{C^{{\a\over 2},{\a\over 2},\a,1+\a,2}}
\les {1\over 2}\big[\|\d\th(\cd)\|_{C^{{\a\over 2},1+\a}}+\|\d\th^0(\cd)\|_{C^{{\a\over 2},{\a\over 2},\a,1+\a,2}}\big],
\ee
on some time interval $[T-\bar\e,T]\subseteq[T-\e,T]$.
Thus, $\G(\cd)$ is a contraction mapping and then it admits a unique fixed point
$ (\Th(\cd),\Th^0(\cd))\in  C^{{\a\over 2},1+\a}\times C^{{\a\over 2},{\a\over 2},\a,1+\a,2}$.
Then we get the well-posedness of the following  equation:
\begin{align*}
\nn &\Th(s,x)=\int_{\dbR^n}\Xi(s,x,T,\mu)h(\mu)d\mu+\int_s^T\int_{\dbR^n}\Xi(s,x,r,\mu)\[\Th_{x}(r,\mu)\ti b\big(r,\mu,\Th(r,\mu),\Th_x(r,\mu),\\
&\q\,\,\Th^0(r,\mu,r,\mu,\Th(r,\mu)), \Th_x^0(r,\mu,r,\mu,\Th(r,\mu)),\Th_y^0(r,\mu,r,\mu,\Th(r,\mu))\big)\\
&\q+\ti g\big(r,\mu,\Th(r,\mu),\Th_x(r,\mu),\Th^0(r,\mu,r,\mu,\Th(r,\mu)), \Th_x^0(r,\mu,r,\mu,\Th(r,\mu)),\Th_y^0(r,\mu,r,\mu,\Th(r,\mu))\big)\]d\mu dr;\\
\nn &\Th^0(t,s.\ti x,x,y)=\int_{\dbR^n}\Xi(s,x,T,\mu)L(t,\ti x,\mu,y)d\mu+\int_s^T\int_{\dbR^n}\Xi(s,x,r,\mu)\\
&\q\times\[\Th^0_{x}(t,r,\ti x, \mu,y)\ti  b\big(r,\mu,\Th(r,\mu),\Th_x(r,\mu),\Th^0(r,\mu,r,\mu,\Th(r,\mu)), \Th_x^0(r,\mu,r,\mu,\Th(r,\mu)),\\
&\q\,\,\Th_y^0(r,\mu,r,\mu,\Th(r,\mu))\big)+\ti g^0\big(t,r,\ti x,\mu,\Th(r,\mu),\Th_x(r,\mu), \Th^0(r,\mu,r,\mu,\Th(r,\mu)),\\
&\q\,\, \Th_x^0(r,\mu,r,\mu,\Th(r,\mu)),\Th_y^0(r,\mu,r,\mu,\Th(r,\mu)),\Th^0_{x}(t,r,\ti x, \mu,y)\big)\]d\mu dr.
\end{align*}
By \autoref{lem:HJB-linear},  $(\Th(\cd),\Th^0(\cd))$ is the unique classical solution of
equilibrium HJB equation \rf{HJB-si-ti} on $[T-\bar\e,T]$.

\ms

In the following, let us show that \rf{Contraction} really holds for some $\bar\e>0$.

\ms
\textbf{Step 1.}
From \rf{lemma-PDE-main1}--\rf{lemma-PDE-main2}, we have
\begin{align}
 \d\Th(s,x)&=\int_s^T\int_{\dbR^n}\Xi(s,x,r,\mu)\[ \Th^1_x(r,\mu)\d\ti b(r,\mu)+\d\Th_x(r,\mu)\ti b(r,\mu;\th^2,\th^{0,2})+\d\ti g(r,\mu)\]d\mu dr,\label{D-Th}\\
\d\Th^0(t,s,\ti x,x,y)&=\int_s^T\int_{\dbR^n}\Xi(s,x,r,\mu)\[\Th_x^{0,1}(t,r,\ti x,\mu,y) \d\ti b(r,\mu)
+\Th_x^{0,1}(t,r,\ti x,\mu,y)\ti b(r,\mu;\th^2,\th^{0,2})\nn\\
&\q+\d\ti g^0(t,r,\ti x,\mu,y)\]d\mu dr.\label{D-V}
\end{align}
By  \autoref{lem:HJB-Th-x} and the estimate \rf{G-G-x-est}, from \rf{D-Th} we get
\bel{D-Th-1}
 |\d\Th(s,x)|\les K\int_s^T\int_{\dbR^n}{e^{-{\l|x-\mu|^2\over 4(r-s)}}\over (r-s)^{{n\over 2}}}
\Big[|\d\ti b(r,\mu)|+|\d\Th_x(r,\mu)|+|\d\ti g(r,\mu)|\Big]d\mu dr.
\ee
For $\f(\cd)=b(\cd),g(\cd)$, by \autoref{lem:HJB-par} (or \rf{lem:v-par-estimate}) we have
\begin{align}
|\d\ti \f(s,x)|&\les K\[|\d\th(s,x)|+|\d\th_x(s,x)|+|\th^{0,1}(s,x,s,x,\th^1(s,x))-\th^{0,2}(s,x,s,x,\th^2(s,x))|
\nn\\
&\q +|\th_x^{0,1}(s,x,s,x,\th^1(s,x))-\th_x^{0,2}(s,x,s,x,\th^2(s,x))|\nn\\
&\q+|\th_y^{0,1}(s,x,s,x,\th^1(s,x))-\th_y^{0,2}(s,x,s,x,\th^2(s,x))|\]\nn\\
&\les K\[|\d\th(s,x)|+|\d\th_x(s,x)|+|\d\th^{0}(s,x,s,x,\th^1(s,x))|\nn\\
&\q+|\d\th_x^{0}(s,x,s,x,\th^1(s,x))|+|\d\th_y^{0}(s,x,s,x,\th^1(s,x))|\].\label{D-b}
\end{align}
Substituting the above into \rf{D-Th-1} yields that
\begin{align}
\nn |\d\Th(s,x)|&\les K(T-s)\big[ \|\d\th(\cd)\|_{L^\i}+\|\d\th_x(\cd)\|_{L^\i}+\|\d \th^0(\cd)\|_{L^\i}
+\|\d\th^0_x(\cd)\|_{L^\i}+\|\d\th^0_y(\cd)\|_{L^\i}\big] \\
\label{D-Th-esti}
&\q+K\int_s^T\int_{\dbR^n}{e^{-{\l|x-\mu|^2\over 4(r-s)}}\over (r-s)^{{n\over 2}}}|\d\Th_x(r,\mu)|d\mu dr.
\end{align}
Similar to \rf{D-b},  using \autoref{lem:HJB-par} (or \rf{lem:v-par-estimate}) again, we have
\begin{align}
|\d\ti g^0(t,s,\ti x,x,y)|&\les K\[|\d\th(s,x)|+|\d\th_x(s,x)|+|\d\th^{0}(s,x,s,x,\th^1(s,x))|
+|\d\th_x^{0}(s,x,s,x,\th^1(s,x))|\nn\\
&\q+|\d\th_y^{0}(s,x,s,x,\th^1(s,x))|+\d\Th^0(t,s,\ti x,x,y)\].\label{D-f}
\end{align}
Substituting the above into \rf{D-V}, we get
\begin{align}
\nn |\d\Th^0(t,s,\ti x,x,y)|&\les K(T-s)\big[ \|\d\th(\cd)\|_{L^\i}+\|\d\th_x(\cd)\|_{L^\i}+\|\d \th^0(\cd)\|_{L^\i}
+\|\d\th^0_x(\cd)\|_{L^\i}+\|\d\th^0_y(\cd)\|_{L^\i}\big]\\
&\q +K\int_s^T\int_{\dbR^n}{e^{-{\l|x-\mu|^2\over 4(r-s)}}\over (r-s)^{{n\over 2}}}|\d\Th^0_x(t,r,\ti x,\mu,y)|d\mu dr.\label{D-V-esti}
\end{align}
From \rf{proof-step11-Th-x} and \rf{proof-step1-Vx}, we have%
\begin{align}
\d\Th_x(s,x)&=\int_s^T\int_{\dbR^n}\Xi_x(s,x,r,\mu)\[ \Th^1_x(r,\mu)\d\ti b(r,\mu)+\d\Th_x(r,\mu)\ti b(r,\mu;\th^2,\th^{0,2})+\d\ti g(r,\mu)\]d\mu dr,\label{D-Th-x}\\
\nn \d\Th_x^0(t,s,\ti x,x,y)&=\int_s^T\int_{\dbR^n}\Xi_x(s,x,r,\mu)\[\Th_x^{0,1}(t,r,\ti x,\mu,y) \d\ti b(r,\mu)+\Th_x^{0,1}(t,r,\ti x,\mu,y)\ti b(r,\mu;\th^2,\th^{0,2})\\
&\qq\qq\qq\qq\qq +\d\ti g^0(t,r,\ti x,\mu,y)\]d\mu dr.\label{D-V-x}
\end{align}
By \autoref{lem:HJB-Th-x} and the estimates \rf{G-G-x-est}, \rf{D-b} and \rf{D-f}, we get
\begin{align}
\nn |\d\Th_x(s,x)|&\les K\sqrt{T-s}\big[ \|\d\th(\cd)\|_{L^\i}+\|\d\th_x(\cd)\|_{L^\i}+\|\d \th^0(\cd)\|_{L^\i}
+\|\d\th^0_x(\cd)\|_{L^\i}+\|\d\th^0_y(\cd)\|_{L^\i}\big]\\
\label{D-Th-x-esti}
&\q+K\int_s^T\int_{\dbR^n}{e^{-{\l|x-\mu|^2\over 4(r-s)}}\over (r-s)^{n+1\over2}}|\d\Th_x(r,\mu)|d\mu dr,\\
\nn |\d\Th^0_x(t,s,\ti x,x,y)|&\les K\sqrt{T-s}\big[ \|\d\th(\cd)\|_{L^\i}+\|\d\th_x(\cd)\|_{L^\i}+\|\d \th^0(\cd)\|_{L^\i}
+\|\d\th^0_x(\cd)\|_{L^\i}+\|\d\th^0_y(\cd)\|_{L^\i}\big]\\
\label{D-V-x-esti}
&\q+K\int_s^T\int_{\dbR^n}{e^{-{\l|x-\mu|^2\over 4(r-s)}}\over (r-s)^{n+1\over2}}|\d\Th^0_x(t,r,\ti x,\mu,y)|d\mu dr.
\end{align}
Combining \rf{D-Th-esti}, \rf{D-V-esti}, \rf{D-Th-x-esti} and \rf{D-V-x-esti} together,
and then by the Gr\"{o}nwall's inequality, we get
\begin{align}
&\|\d\Th(\cd)\|_{C^{0,1}([T-\e,T])}+\|\d\Th^0(\cd)\|_{C^{0,0,0,1,0}([T-\e,T])}\nn\\
&\q\les K\sqrt{\e}\big[\|\d\th(\cd)\|_{C^{0,1}([T-\e,T])}+\|\d\th^0(\cd)\|_{C^{0,0,0,1,1}([T-\e,T])}\big].
\label{D-V-Th-0,1}
\end{align}

\ms
\textbf{Step 2.}
From the proof of \autoref{lem:HJB-par}, by some direct computations, it is easily seen that
\begin{align}
\nn\d\Th^0_{y}(t,s,\ti x,x,y)&=\int_s^T\int_{\dbR^n}\Xi(s,x,r,\mu)\[\Th^{0,1}_{xy}(t,r,\ti x,\mu,y)\d\ti b(r,\mu)
\\
&\q+\d\Th_{xy}^{0}(t,r,\ti x,\mu,y)\ti b(r,\mu;\th^2,\th^{0,2})
+ \d\ti g^0_{p^0}(t,r,\ti x,\mu,y)\Th^{0,1}_{xy}(t,r,\ti x,\mu,y)\nn\\
&\q+ \ti g^0_{p^0}(t,r,\ti x,\mu,y;\th^1,\th^{0,1},\Th^{0,1})\d\Th_{xy}^{0}(t,r,\ti x,\mu,y)\]d\mu dr,\nn\\
\nn\d\Th^0_{xy}(t,s,\ti x,x,y)&=\int_s^T\int_{\dbR^n}\Xi_x(s,x,r,\mu)
\[\Th^{0,1}_{xy}(t,r,\ti x,\mu,y)\d\ti b(r,\mu)\\
&\q+\d\Th_{xy}^{0}(t,r,\ti x,\mu,y)\ti b(r,\mu;\th^2,\th^{0,2})
+ \d\ti g^0_{p^0}(t,r,\ti x,\mu,y)\Th^{0,1}_{xy}(t,r,\ti x,\mu,y)\nn\\
&\q+ \ti g^0_{p^0}(t,r,\ti x,\mu,y;\th^1,\th^{0,1},\Th^{0,1})\d\Th_{xy}^{0}(t,r,\ti x,\mu,y)\]d\mu dr,\nn
\end{align}
and
\begin{align}
\nn\d\Th^0_{yy}(t,s,\ti x,x,y)&=\int_s^T\int_{\dbR^n}\Xi(s,x,r,\mu)\[\Th^{0,1}_{xyy}(t,r,\ti x,\mu,y)\d\ti b(r,\mu)
\\
&\q+\d\Th^{0}_{xyy}(t,r,\ti x,\mu,y)\ti b(r,\mu;\th^2,\th^{0,2})
+\d\ti g^0_{p^0}(t,r,\ti x,\mu,y)\Th^{0,1}_{xyy}(t,r,\ti x,\mu,y)\nn\\
&\q + \ti g^0_{p^0}(t,r,\ti x,\mu,y;\th^1,\th^{0,1},\Th^{0,1})\d\Th_{xyy}^{0}(t,r,\ti x,\mu,y)\nn\\
&\q+ \lan \d\ti g^0_{p^0p^0}(t,r,\ti x,\mu,y)\Th^{0,1}_{xy}(t,r,\ti x,\mu,y),\, \Th^{0,1}_{xy}(t,r,\ti x,\mu,y)\ran \nn\\
&\q + \lan\ti g^0_{p^0p^0}(t,r,\ti x,\mu,y;\th^2,\th^{0,2},\Th^{0,2})
\d\Th^{0}_{xy}(t,r,\ti x,\mu,y),\, \Th^{0,1}_{xy}(t,r,\ti x,\mu,y)\ran \nn\\
&\q + \lan\ti g^0_{p^0p^0}(t,r,\ti x,\mu,y;\th^2,\th^{0,2},\Th^{0,2})
\Th^{0,2}_{xy}(t,r,\ti x,\mu,y),\, \d\Th^{0}_{xy}(t,r,\ti x,\mu,y)\ran\]d\mu dr,\nn\\
\nn\d\Th^0_{xyy}(t,s,\ti x,x,y)&=\int_s^T\int_{\dbR^n}\Xi_x(s,x,r,\mu)
\[\Th^{0,1}_{xyy}(t,r,\ti x,\mu,y)\d\ti b(r,\mu)\\
&\q+\d \Th^{0}_{xyy}(t,r,\ti x,\mu,y)\ti b(r,\mu;\th^2,\th^{0,2})
+\d\ti g^0_{p^0}(t,r,\ti x,\mu,y)\Th^{0,1}_{xyy}(t,r,\ti x,\mu,y)\nn\\
&\q + \ti g^0_{p^0}(t,r,\ti x,\mu,y;\th^1,\th^{0,1},\Th^{0,1})\d\Th_{xyy}^{0}(t,r,\ti x,\mu,y)\nn\\
&\q+ \lan \d\ti g^0_{p^0p^0}(t,r,\ti x,\mu,y)\Th^{0,1}_{xy}(t,r,\ti x,\mu,y),\, \Th^{0,1}_{xy}(t,r,\ti x,\mu,y)\ran \nn\\
&\q + \lan\ti g^0_{p^0p^0}(t,r,\ti x,\mu,y;\th^2,\th^{0,2},\Th^{0,2})
\d\Th^{0}_{xy}(t,r,\ti x,\mu,y),\, \Th^{0,1}_{xy}(t,r,\ti x,\mu,y)\ran \nn\\
&\q + \lan\ti g^0_{p^0p^0}(t,r,\ti x,\mu,y;\th^2,\th^{0,2},\Th^{0,2})
\Th^{0,2}_{xy}(t,r,\ti x,\mu,y),\, \d\Th^{0}_{xy}(t,r,\ti x,\mu,y)\ran\]d\mu dr.\nn
\end{align}
Then by \autoref{lem:HJB-par} and the estimate \rf{G-G-x-est}, we get
\begin{align}
\nn |\d\Th^0_y(t,s,\ti x,x,y)|&\les K(T-s)\[ \|\d\th(\cd)\|_{L^\i}+\|\d\th_x(\cd)\|_{L^\i}
+\|\d\th^0(\cd)\|_{L^\i}+\|\d\th^0_x(\cd)\|_{L^\i}+\|\d\th^0_y(\cd)\|_{L^\i}\]\\
&\q+ K\int_s^T\int_{\dbR^n}{e^{-{\l|x-\mu|^2\over 4(r-s)}}\over (r-s)^{{n\over 2}}}
\[|\d\Th^0_{x}(t,r,\ti x,\mu,y)|+|\d\Th^0_{xy}(t,r,\ti x,\mu,y)|\]d\mu dr,\nn\\
\nn |\d\Th^0_{xy}(t,s,\ti x,x,y)|&\les K\sqrt{T-s}\[ \|\d\th(\cd)\|_{L^\i}+\|\d\th_x(\cd)\|_{L^\i}
+\|\d\th^0(\cd)\|_{L^\i}+\|\d\th^0_x(\cd)\|_{L^\i}+\|\d\th^0_y(\cd)\|_{L^\i}\]\\
&\q+ K\int_s^T\int_{\dbR^n}{e^{-{\l|x-\mu|^2\over 4(r-s)}}\over (r-s)^{{n+1\over 2}}}
\[|\d\Th^0_{x}(t,r,\ti x,\mu,y)|+|\d\Th^0_{xy}(t,r,\ti x,\mu,y)|\]d\mu dr,\nn
\end{align}
and
\begin{align}
\nn |\d\Th^0_{yy}(t,s,\ti x,x,y)|&\les K(T-s)\[ \|\d\th(\cd)\|_{L^\i}+\|\d\th_x(\cd)\|_{L^\i}
+\|\d\th^0(\cd)\|_{L^\i}+\|\d\th^0_x(\cd)\|_{L^\i}+\|\d\th^0_y(\cd)\|_{L^\i}\]\\
&\q+ K\int_s^T\int_{\dbR^n}{e^{-{\l|x-\mu|^2\over 4(r-s)}}\over (r-s)^{{n\over 2}}}
\[|\d\Th^0_{x}(t,r,\ti x,\mu,y)|+|\d\Th^0_{xy}(t,r,\ti x,\mu,y)|+|\d\Th^0_{xyy}(t,r,\ti x,\mu,y)|\]d\mu dr,\nn\\
\nn |\d V_{xyy}(t,y,z,s,x)|&\les K\sqrt{T-s}\[ \|\d\th(\cd)\|_{L^\i}+\|\d\th_x(\cd)\|_{L^\i}
+\|\d\th^0(\cd)\|_{L^\i}+\|\d\th^0_x(\cd)\|_{L^\i}+\|\d\th^0_y(\cd)\|_{L^\i}\]\\
&\q+ K\int_s^T\int_{\dbR^n}{e^{-{\l|x-\mu|^2\over 4(r-s)}}\over (r-s)^{{n+1\over 2}}}
\[|\d\Th^0_{x}(t,r,\ti x,\mu,y)|+|\d\Th^0_{xy}(t,r,\ti x,\mu,y)|+|\d\Th^0_{xyy}(t,r,\ti x,\mu,y)|\]d\mu dr.\nn
\end{align}
Combining the above with the estimate \rf{D-V-Th-0,1} together, by Gr\"{o}nwall's inequality again we get
\begin{align*}
&\|\d\Th(\cd)\|_{C^{0,1}([T-\e,T])}+\|\d\Th^0(\cd)\|_{C^{0,,0,1,2}([T-\e,T])}\\
&\q\les \sqrt{\e} K\big[\|\d\th(\cd)\|_{C^{0,1}([T-\e,T])}+\|\d\th^0(\cd)\|_{C^{0,0,0,1,1}([T-\e,T])}\big].
\end{align*}
For any $\ti x_1,\ti x_2\in\dbR^n$, denote
$$
\d\h\Th^0(t,s,x,y)=\d\Th^0(t,s,\ti x_1,x,y)-\d\Th^0(t,s,\ti x_2,x,y).
$$
Then from \rf{D-V-x}, we have
\begin{align*}
\d\h\Th_x^0(t,s,x,y)&=\int_s^T\int_{\dbR^n}\Xi_x(s,x,r,\mu)
\Big\{[\Th_x^{0,1}(t,r,\ti x_1,\mu,y)-\Th_x^{0,1}(t,r,\ti x_2,\mu,y)]\d\ti b(r,\mu)\\
&\q +\d\h\Th^0_x(t,r,\mu,y)\ti b(r,\mu;\th^2,\th^{0,2})
+\d\ti g^0(t,r,\ti x_1,\mu,y)-\d\ti g^0(t,r,\ti x_2,\mu,y)\Big\}d\mu dr.
\end{align*}
Note that
\begin{align*}
&|\d\ti g^0(t,r,\ti x_1,\mu,y)-\d\ti g^0(t,r,\ti x_2,\mu,y)|\nn\\
&\q\les K|\d\h\Th_x^0(t,s,x,y)|+K |\d\Th_x^0(t,s,\ti x_1,x,y)|\big[1+\|\Th^{0,1}_x(t,s,\cd,x,y)\|_\a+\|\Th^{0,2}_x(t,s,\cd,x,y)\|_\a\big]
|\ti x_1-\ti x_2|^\a.
\end{align*}
Then by the fact that $\|\Th^{0,1}_x(t,s,\cd,x,y)\|_\a+\|\Th^{0,2}_x(t,s,\cd,x,y)\|_\a$
is uniformly bounded (see \autoref{lem:HJB-par}), we get
\begin{align*}
\d\h\Th_x^0(t,s,x,y)&\les K\int_s^T\int_{\dbR^n}{e^{-{\l|x-\mu|^2\over 4(r-s)}}\over (r-s)^{{n+1\over 2}}}
\Big\{|\d\h\Th_x^0(t,r,\mu,y)|+|\d\Th_x^0(t,r,\ti x_1,\mu,y)|
|\ti x_1-\ti x_2|^\a\\
&\q+\big[ \|\d\th(\cd)\|_{L^\i}+\|\d\th_x(\cd)\|_{L^\i}+\|\d \th^0(\cd)\|_{L^\i}+\|\d\th^0_x(\cd)\|_{L^\i}+\|\d\th^0_y(\cd)\|_{L^\i}\big]|\ti x_1-\ti x_2|^\a\Big\}d\mu dr.
\end{align*}
It follows that
\begin{align*}
&\|\d\Th^0(\cd)\|_{C^{0,0,\a,1,0}([T-\e,T])}\les
\sqrt{\e} K\big[\|\d\th(\cd)\|_{C^{0,1}([T-\e,T])}+\|\d\th^0(\cd)\|_{C^{0,0,0,1,1}([T-\e,T])}\big].
\end{align*}
By continuing the above arguments, we have
\begin{align*}
&\|\d\Th(\cd)\|_{C^{0,1}([T-\e,T])}+\|\d\Th^0(\cd)\|_{C^{{\a\over 2},0,\a,1,2}([T-\e,T])}\nn\\
&\q\les \sqrt{\e}K\big[\|\d\th(\cd)\|_{C^{0,1}([T-\e,T])}+\|\d\th^0(\cd)\|_{C^{0,0,0,1,1}([T-\e,T])}\big].
\end{align*}

\ms

\textbf{Step 3.} Recalling \rf{proof-step1-Th-x-change}--\rf{proof-step1-V-x-change},
similar to \rf{D-Th-x}--\rf{D-V-x},  we have
\begin{align}
\nn \d\Th_x(s,x)&=\int_s^T\int_{\dbR^n}\ti\Xi_x(s,x,r,x-\sqrt{r-s}\mu)\[\d\ti g(r,x-\sqrt{r-s}\mu)\nn\\
&\q +\Th^1_x(r,x-\sqrt{r-s}\mu)\d\ti b(r,x-\sqrt{r-s}\mu)\nn\\
&\q +\d\Th_x(r,x-\sqrt{r-s}\mu)\ti b(r,x-\sqrt{r-s}\mu;\th^2,\th^{0,2})\]d\mu dr,\label{D-Thx}\\
\nn\d\Th^0_x(t,s,\ti x,x,y)&=\int_s^T\int_{\dbR^n}\ti\Xi_x(s,x,r,x-\sqrt{r-s}\mu)\[\d\ti g^0(t,r,\ti x,x-\sqrt{r-s}\mu,y)\\
&\q+\Th^{0,1}_x(t,r,\ti x,x-\sqrt{r-s}\mu,y) \d\ti b(r,x-\sqrt{r-s}\mu)\nn\\
&\q+\d\Th^0_x(t,r,\ti x,x-\sqrt{r-s}\mu,y)\ti b(r,x-\sqrt{r-s}\mu;\th^2,\th^{0,2})\]d\mu dr,\label{D-Vx}
\end{align}
where $\ti\Xi(\cd)$ is defined by \rf{ti-G}. For any $x_1,x_2\in\dbR^n$, we have
\begin{align}
\nn &\d\Th^0_x(t,s,\ti x,x_1,y)-\d\Th^0_x(t,s,\ti x,x_2,y)\\
\nn&\q=\int_s^T\int_{\dbR^n}\big[\ti\Xi_x(s,x_1,r,x_1-\sqrt{r-s}\mu)-\ti\Xi_x(s,x_2,r,x_2-\sqrt{r-s}\mu)\big]\\
\nn&\qq\times\[\d\ti g^0(t,r,\ti x,x_1-\sqrt{r-s}\mu,y)+
\Th^{0,1}_x(t,r,\ti x,x_1-\sqrt{r-s}\mu,y)\d\ti b(r,x_1-\sqrt{r-s}\mu)\\
\nn&\qq+\d\Th^0_x(t,r,\ti x,x_1-\sqrt{r-s}\mu,y)\ti b(r,x_1-\sqrt{r-s}\mu;\th^2,\th^{0,2})\]d\mu dr\nn\\
\nn&\qq+\int_s^T\int_{\dbR^n}\ti\Xi_x(s,x_2,r,x_2-\sqrt{r-s}\mu)\Big\{
\big[\d\ti g^0(t,r,\ti x,x_1-\sqrt{r-s}\mu,y)\\
\nn&\qq-\d\ti g^0(t,r,\ti x,x_2-\sqrt{r-s}\mu,y)\big]
+\big[\Th^{0,1}_x(t,r,\ti x,x_1-\sqrt{r-s}\mu,y) \d\ti b(r,x_1-\sqrt{r-s}\mu)\\
\nn&\qq-\Th^{0,1}_x(t,r,\ti x,x_2-\sqrt{r-s}\mu,y) \d\ti b(r,x_2-\sqrt{r-s}\mu)\big]
\\
\nn &\qq+\big[\d\Th^0_x(t,r,\ti x,x_1-\sqrt{r-s}\mu,y)\ti b(r,x_1-\sqrt{r-s}\mu;\th^2,\th^{0,2})\nn\\
&\qq-\d\Th^0_x(t,r,\ti x,x_2-\sqrt{r-s}\mu,y)\ti b(r,x_2-\sqrt{r-s}\mu;\th^2,\th^{0,2})
\big]\Big\}d\mu dr.\label{D-Th-x-a}
\end{align}
Note that on $[T-\e,T]$, by \autoref{lem:HJB-par} (or \rf{lem:v-par-estimate}) and \autoref{Prop:Holder} we have
\begin{align}
\nn &\big|[\th^{0,1}_x(r,x_1,r,x_1,\th^1(r,x_1))-\th^{0,2}_x(r,x_1,r,x_1,\th^2(r,x_1))]\\
\nn &-[\th^{0,1}_x(r,x_2,r,x_2,\th^1(r,x_2))-\th^{0,2}_x(r,x_2,r,x_2,\th^2(r,x_2))]\big|\\
\nn &\q\les \big|[\th^{0,1}_x(r,x_1,r,x_1,\th^1(r,x_1))-\th^{0,2}_x(r,x_1,r,x_1,\th^1(r,x_1))]\nn\\
&\qq-[\th^{0,1}_x(r,x_2,r,x_2,\th^1(r,x_2))-\th^{0,2}_x(r,x_2,r,x_2,\th^1(r,x_2))]\big|\nn\\
\nn &\qq+\big|[\th^{0,2}_x(r,x_1,r,x_1,\th^1(r,x_1))-\th^{0,2}_x(r,x_1,r,x_1,\th^2(r,x_1))]\\
\nn &\qq-[\th^{0,2}_x(r,x_2,r,x_2,\th^1(r,x_2))-\th^{0,2}_x(r,x_2,r,x_2,\th^2(r,x_2))]\big|\\
\nn&\q\les \|\d\th^0_{xy}(\cd)\|_{L^\i}\|\th^1(\cd)\|_{C^{0,\a}}|x_1-x_2|^\a
+\|\d\th^0_{x}(\cd)\|_{C^{0,0,\a,\a,0}}|x_1-x_2|^\a\\
\nn&\qq +\Big|\int_0^1 \th^{0,2}_{xy}\big(r,x_1,r,x_1,l \th^1(r,x_1)+(1-l)\th^2(r,x_1)\big)dl
[\th^1(r,x_1)- \th^2(r,x_1)]\nn\\
\nn &\qq-\int_0^1  \th^{0,2}_{xy}\big(r,x_2,r,x_2,l \th^1(r,x_2)+(1-l)\th^2(r,x_2)\big)dl
[\th^1(r,x_2)- \th^2(r,x_2)]\Big|\nn\\
\nn &\q\les \Big\{\|\d\th^0_{xy}(\cd)\|_{L^\i}\|\th^1(\cd)\|_{C^{0,\a}}
+ \|\d\th^0_{x}(\cd)\|_{C^{0,0,\a,\a,0}}+\|\th^{0,2}_{xy}(\cd)\|_{L^\i}\|\d\th(\cd)\|_{C^{0,\a}}\\
\nn&\qq +\big[\|\th^{0,2}_{xyy}(\cd)\|_{L^\i}(\|\th^1(\cd)\|_{C^{0,\a}}+\|\th^2(\cd)\|_{C^{0,\a}})
+\|\th^{0,2}_{xy}(\cd)\|_{C^{0,0,\a,\a,0}}\big]\|\d\th(\cd)\|_{L^\i}\Big\}|x_1-x_2|^\a\nn\\
&\q\les K\big[ \|\d\th^0_{x}(\cd)\|_{C^{0,0,\a,\a,1}}+\|\d\th(\cd)\|_{C^{0,\a}}\big]|x_1-x_2|^\a.
\label{Dv-estimate}
\end{align}
By the same arguments as the above, we have
\begin{align}
\nn &\big|[\th^{0,1}(r,x_1,r,x_1,\th^1(r,x_1))-\th^{0,2}(r,x_1,r,x_1,\th^2(r,x_1))]\\
\nn &-[\th^{0,1}(r,x_2,r,x_2,\th^1(r,x_2))-\th^{0,2}(r,x_2,r,x_2,\th^2(r,x_2))]\big|\\
&\q\les  K\big[ \|\d\th^0(\cd)\|_{C^{0,0,\a,\a,1}}+\|\d\th(\cd)\|_{C^{0,\a}}\big]|x_1-x_2|^\a,
\end{align}
and
\begin{align}
\nn &\big|[\th_y^{0,1}(r,x_1,r,x_1,\th^1(r,x_1))-\th_y^{0,2}(r,x_1,r,x_1,\th^2(r,x_1))]\\
\nn &-[\th_y^{0,1}(r,x_2,r,x_2,\th^1(r,x_2))-\th_y^{0,2}(r,x_2,r,x_2,\th^2(r,x_2))]\big|\\
&\q\les  K\big[ \|\d\th_y^0(\cd)\|_{C^{0,0,\a,\a,1}}+\|\d\th(\cd)\|_{C^{0,\a}}\big]|x_1-x_2|^\a.
\label{Dv-estimate1}
\end{align}
For any $\f(\cd)\in C^{1,2}$, from \rf{Dv-estimate} and the fact
\begin{align}
\d\f(x_1)-\d\f(x_2)&:= \big[\f(x_1,\th^{0,1}_x(r,x_1,r,x_1,\th^1(r,x_1)))
-\f(x_1,\th^{0,2}_x(r,x_1,r,x_1,\th^2(r,x_1)))\big]\nn\\
\nn &\q-\big[\f(x_2,\th^{0,1}_x(r,x_2,r,x_2,\th^1(r,x_2)))-\f(x_2,\th^{0,2}_x(r,x_2,r,x_2,\th^2(r,x_2)))\big]\\
\nn &=\int_0^1 \f_{\th^0}\big(x_1,\ell \th^{0,1}_x(r,x_1,r,x_1,\th^1(r,x_1))+(1-\ell) \th^{0,2}_x(r,x_1,r,x_1,\th^2(r,x_1))\big)d\ell\nn\\
\nn&\q\times \Big\{[\th^{0,1}_x(r,x_1,r,x_1,\th^1(r,x_1))-\th^{0,2}_x(r,x_1,r,x_1,\th^2(r,x_1))]\\
\nn &\qq-[\th^{0,1}_x(r,x_2,r,x_2,\th^1(r,x_2))-\th^{0,2}_x(r,x_2,r,x_2,\th^2(r,x_2))]\Big\}\\
\nn &\q+\int_0^1 \Big[\f_{\th^0}\big(x_1,\ell \th^{0,1}_x(r,x_1,r,x_1,\th^1(r,x_1))+(1-\ell) \th^{0,2}_x(r,x_1,r,x_1,\th^2(r,x_1))\big)\nn\\
\nn&\q-\f_{\th^0}\big(x_2,\ell\th^{0,1}_x(r,x_2,r,x_2,\th^1(r,x_2))+(1-\ell) \th^{0,2}_x(r,x_2,r,x_2,\th^2(r,x_2))\big)\Big]d\ell\nn\\
&\q\times\big[\th^{0,1}_x(r,x_2,r,x_2,\th^1(r,x_2))-\th^{0,2}_x(r,x_2,r,x_2,\th^2(r,x_2))\big],\nn
\end{align}
we have the following estimate:
\begin{align}
|\d\f(x_1)-\d\f(x_2)|\les
K\big[\|\d\th^0_{x}(\cd)\|_{C^{0,0,\a,\a,1}}+\|\d\th(\cd)\|_{C^{0,\a}}\big]|x_1-x_2|^\a\q \hbox{on}\q [T-\e,T],
\label{esti-phi}
\end{align}
where $K$ depends on $\|\f(\cd)\|_{C^{0,1}}$.
Recall that
\begin{align}
\nn &\d\ti g^0(t,r,\ti x,x_1,y)-\d\ti g^0(t,r,\ti x,x_2,y)\\
\nn &\q=\[\ti g^0\big(t,r,\ti x,x_1,\th^1(r,x_1),\th^1_x(r,x_1), \th^{0,1}(r,x_1,r,x_1,\th^1(r,x_1)),
\th_x^{0,1}(r,x_1,r,x_1,\th^1(r,x_1)),\nn\\
&\qq\, \th_y^{0,1}(r,x_1,r,x_1,\th^1(r,x_1)),\Th^{0,1}_{x}(t,r,\ti x,x_1,y)\big)
-\ti g^0\big(t,r,\ti x,x_1,\th^2(r,x_1),\th^2_x(r,x_1),\nn\\
&\qq\,\th^{0,2}(r,x_1,r,x_1,\th^2(r,x_1)), \th_x^{0,2}(r,x_1,r,x_1,\th^2(r,x_1)),\th_y^{0,2}(r,x_1,r,x_1,\th^2(r,x_1)),\Th^{0,2}_{x}(t,r,\ti x,x_1,y)\big)\]\nn\\
&\qq\,-\[\ti g^0\big(t,r,\ti x,x_2,\th^1(r,x_2),\th^1_x(r,x_2), \th^{0,1}(r,x_2,r,x_2,\th^1(r,x_2)),\th_x^{0,1}(r,x_2,r,x_2,\th^1(r,x_2)),\nn\\
&\qq\qq \th_y^{0,1}(r,x_2,r,x_2,\th^1(r,x_2)),\Th^{0,1}_{x}(t,r,\ti x,x_2,y)\big)
-\ti g^0\big(t,r,\ti x,x_2,\th^2(r,x_2),\th^2_x(r,x_2),\nn\\
&\qq\qq\th^{0,2}(r,x_2,r,x_2,\th^2(r,x_2)), \th_x^{0,2}(r,x_2,r,x_2,\th^2(r,x_2)),\th_y^{0,2}(r,x_2,r,x_2,\th^2(r,x_2)),\Th^{0,2}_{x}(t,r,\ti x,x_2,y)\big)\].\nn
\end{align}
Using \rf{esti-phi}, by the estimates \rf{Dv-estimate}--\rf{Dv-estimate1},
we get
\begin{align}
\nn& |\d\ti g^0(t,r,\ti x,x_1,y)-\d\ti g^0(t,r,\ti x,x_2,y)|\\
&\q\les K\big[ \|\d\th^0(\cd)\|_{C^{0,0,\a,1+\a,2}}+\|\d\th(\cd)\|_{C^{0,1+\a}}
+\|\d\Th^0_x(t,r,\ti x,\cd,y)\|_{\a}\big]|x_1-x_2|^\a.
\label{D-f-esti}
\end{align}
Similarly,
\begin{align}
 |\d\ti b(r,x_1)-\d\ti b(r,x_2)|
 \les K\big[ \|\d\th^0(\cd)\|_{C^{0,0,\a,1+\a,2}}+\|\d\th(\cd)\|_{C^{0,1+\a}}\big]|x_1-x_2|^\a.
 \label{D-b-esti}
\end{align}
By \rf{Th-x-alpha}, \autoref{lem:HJB-Th-x} and \autoref{Prop:Holder}, from \rf{D-Th-x-a} we obtain
\begin{align*}
\|\d\Th^0_x(t,s,\ti x,\cd,y)\|_{\a}&\les\int_s^T\int_{\dbR^n}{K\over \sqrt{r-s}}e^{-\l|\mu|^2}\[| \d\ti b(r,x-\sqrt{r-s}\mu)|\\
&\q+|\d\ti g^0(t,r,\ti x,x-\sqrt{r-s}\mu,y)|+|\d\Th^0_x(t,r,\ti x,x-\sqrt{r-s}\mu,y)|\]d\mu dr\nn\\
&\q+\int_s^T\int_{\dbR^n}{K\over \sqrt{r-s}}e^{-\l|\mu|^2}\[| \d\ti b(r,x_1-\sqrt{r-s}\mu)|+\| \d\ti b(r,\cd)\|_{\a}\\
&\q+|\d\Th^0_x(t,r,\ti x,x_1-\sqrt{r-s}\mu,y)|+\|\d\Th^0_x(t,r,\ti x,\cd,y)\|_{\a}+\|\d\ti g^0(t,r,\ti x,\cd,y)\|_{\a}\]d\mu dr.
\end{align*}
Substituting the estimates \rf{D-b}, \rf{D-f}, \rf{D-f-esti} and \rf{D-b-esti} into the above,
and then by Gr\"{o}nwall's inequality, we get
\begin{align}
\|\d\Th^0_x(t,s,\ti x,\cd,y)\|_{\a}\les K\sqrt{\e}\big[\|\d\th^0(\cd)\|_{C^{0,0,\a,1+\a,2}}
+\|\d\th(\cd)\|_{C^{0,1+\a}}\big].\label{D-V-Holder}
\end{align}
By the same arguments as the above, we have
$$\|\d\Th_x(s,\cd)\|_{\a}\les K\sqrt{\e}\big[\|\d\th^0(\cd)\|_{C^{0,0,\a,1+\a,2}}
+\|\d\th(\cd)\|_{C^{0,1+\a}}\big].
$$
By continuing the above arguments, we get
$$
\|\d\Th(\cd)\|_{C^{0,1+\a}}+\|\d\Th^0(\cd)\|_{C^{0,0,0,1+\a,2}}\les K\sqrt{\e}\big[\|\d\th^0(\cd)\|_{C^{0,0,\a,1+\a,2}}+\|\d\th(\cd)\|_{C^{0,1+\a}}\big].
$$

\ms
\textbf{Step 4.}
For any $T-\e\les s_1\les s_2\les T$, by \rf{D-Vx} we have
\begin{align}
\nn&\d\Th^0_x(t,s_1,\ti x,x,y)-\d\Th^0_x(t,s_2,\ti x,x,y)\\
&\q=\int_{s_1}^{s_2}\int_{\dbR^n}\ti\Xi_x(s_1,x,r,x-\sqrt{r-s_1}\mu)\[\d\ti g^0(t,r,\ti x,x-\sqrt{r-s_1}\mu,y)
+ \Th^{0,1}_x(t,r,\ti x,x-\sqrt{r-s_1}\mu,y)\nn\\
&\qq\q\times \d\ti b(r,x-\sqrt{r-s_1}\mu)
+\d\Th^0_x(t,r,\ti x,x-\sqrt{r-s_1}\mu,y)\ti b(r,x-\sqrt{r-s_1}\mu;\th^2,\th^{0,2})\]d\mu dr\nn\\
\nn&\qq+\int_{s_2}^T\int_{\dbR^n}\big[\ti\Xi_x(s_1,x,r,x-\sqrt{r-s_1}\mu)-\ti\Xi_x(s_2,x,r,x-\sqrt{r-s_2}\mu)\big]\\
\nn&\qq\q\times\[\d\ti g^0(t,r,\ti x,x-\sqrt{r-s_1}\mu,y)+ \Th^{0,1}_x(t,r,\ti x,x-\sqrt{r-s_1}\mu,y)
\d\ti b(r,x-\sqrt{r-s_1}\mu)\nn\\
&\qq\q
+\d\Th^0_x(t,r,\ti x,x-\sqrt{r-s_1}\mu,y)\ti b(r,x-\sqrt{r-s_1}\mu;\th^2,\th^{0,2})\]d\mu dr\nn\\
\nn&\qq+\int_{s_2}^T\int_{\dbR^n}\ti\Xi_x(s_2,x,r,x-\sqrt{r-s_2}\mu)\Big\{\big[\d\ti g^0(t,r,\ti x,x-\sqrt{r-s_1}\mu,y)
-\d\ti g^0(t,r,\ti x,x-\sqrt{r-s_2}\mu,y)\big]\\
\nn&\qq\q+ \big[ \Th^{0,1}_x(t,r,\ti x,x-\sqrt{r-s_1}\mu,y)\d\ti b(r,x-\sqrt{r-s_1}\mu)
-\Th^{0,1}_x(t,r,\ti x,x-\sqrt{r-s_2}\mu,y)\\
\nn&\qq\q\times \d\ti b(r,x-\sqrt{r-s_2}\mu)\big]
+\big[\d\Th^0_x(t,r,\ti x,x-\sqrt{r-s_1}\mu,y)\ti b(r,x-\sqrt{r-s_1}\mu;\th^2,\th^{0,2})\\
\nn &\qq\q-\d\Th^0_x(t,r,\ti x,x-\sqrt{r-s_2}\mu,y)\ti b(r,x-\sqrt{r-s_2}\mu;\th^2,\th^{0,2})\big]\Big\}d\mu dr\nn\\
&\q=: (I)+(II)+(III). \label{D-Th-s-a}
\end{align}
By \autoref{lem:HJB-Th-x} and \rf{Th-x-alpha}, we get
\bel{(I)}
|(I)|\les \sqrt{s_2-s_1}K\big[ \|\d\th^0(\cd)\|_{C^{0,0,0,1,1}}+\|\d\th(\cd)\|_{C^{0,1}}
+\|\d\Th^0(\cd)\|_{ C^{0,0,0,0,1}}\big].
\ee
By  \autoref{lem:HJB-Th-x} and \rf{ti-G}--\rf{Th-x-alpha}, we get
\begin{align}
|(II)|&\les K\[\Big|\int_{s_2}^T\int_{\dbR^n}\big[\ti\Xi_x(s_1,x,r,x-\sqrt{r-s_1}\mu)-\ti\Xi_x(s_2,x,r,x-\sqrt{r-s_1}\mu)\big]d\mu dr\Big|+|s_2-s_1|^{\a\over 2}\sqrt{T-s_1}\]
\nn\\
&\q\times \[ \|\d\th^0(\cd)\|_{C^{0,0,0,1,1}}+\|\d\th(\cd)\|_{C^{0,1}}+\|\d\Th^0(\cd)\|_{ C^{0,0,0,1,0}}\]\nn\\
&\les K\[\int_{s_2}^T\int_{\dbR^n}e^{-\l|\mu|^2\over 4}\Big|{1\over\sqrt{r-s_1}}-{1\over\sqrt{r-s_2}}\Big|d\mu dr+|s_2-s_1|^{\a\over 2}\sqrt{T-s_1}\]
\nn\\
&\q\times \[ \|\d\th^0(\cd)\|_{C^{0,0,0,1,1}}+\|\d\th(\cd)\|_{C^{0,1}}+\|\d\Th^0(\cd)\|_{ C^{0,0,0,1,0}}\]\nn\\
%
&\les K\big[|s_1-s_2|^{1\over 2}+|s_2-s_1|^{\a\over 2}\sqrt{\e}\big]
\big[ \|\d\th^0(\cd)\|_{C^{0,0,0,1,1}}+\|\d\th(\cd)\|_{C^{0,1}}+\|\d\Th^0(\cd)\|_{ C^{0,0,0,1,0}}\big].
\label{(II)}
\end{align}
By the estimates \rf{Th-x-alpha}, \rf{D-f-esti} and \rf{D-b-esti}, we get
\begin{align}
&|(III)|\les K\sqrt{\e}\big[ \|\d\th^0(\cd)\|_{C^{0,0,\a,1+\a,2}}
+\|\d\th(\cd)\|_{C^{0,1+\a}}+\|\d\Th^0(\cd)\|_{ C^{0,0,0,1+\a,0}}\big]|s_1-s_2|^{\a\over 2}.
\label{(III)}
\end{align}
Combining the estimates \rf{(I)}--\rf{(III)} together, we have
\begin{align*}
&\|\d\Th^0_x(t,\cd,\ti x,x,y)\|_{\a\over 2}\les K\e^{1-\a\over2}
\big[ \|\d \th^0(\cd)\|_{C^{0,0,\a,1+\a,2}}+\|\d\th(\cd)\|_{C^{0,1+\a}}+\|\d\Th^0(\cd)\|_{ C^{0,0,0,1+\a,0}}\big].
\end{align*}
Then it follows from \rf{D-V-Holder} that
\begin{align*}
&\|\d\Th^0_x(t,\cd,\ti x,x,y)\|_{\a\over 2}\les K\e^{1-\a\over2}
\big[ \|\d \th^0(\cd)\|_{C^{0,0,\a,1+\a,2}}+\|\d\th(\cd)\|_{C^{0,1+\a}}\big].
\end{align*}
With \rf{D-Thx}, by the same arguments as the above, we have
\begin{align*}
&\|\d\Th_x(\cd,x)\|_{\a\over 2}\les K\e^{1-\a\over2}
\big[ \|\d\th^0(\cd)\|_{C^{0,0,\a,1+\a,2}}+\|\d\th(\cd)\|_{C^{0,1+\a}}\big].
\end{align*}
By continuing the above arguments, we get
\begin{align*}
\|\d\Th(\cd)\|_{C^{{\a\over 2},1}}+\|\d\Th^0(\cd)\|_{C^{0,{\a\over 2},0,1,2}}\les K\e^{1-\a\over2}
\big[ \|\d \th^0(\cd)\|_{C^{0,0,\a,1+\a,2}}+\|\d\th(\cd)\|_{C^{0,1+\a}}\big].
\end{align*}

\ms
\textbf{Step 5.} Combining the estimates in Steps 1--4 together, we get
\begin{align*}
\|\d\Th(\cd)\|_{C^{{\a\over 2},1}}+\|\d\Th^0(\cd)\|_{C^{{\a\over 2},{\a\over 2},\a,1+\a,2}}\les K\e^{1-\a\over2}
\big[ \|\d \th^0(\cd)\|_{C^{0,0,\a,1+\a,2}}+\|\d\th(\cd)\|_{C^{0,1+\a}}\big].
\end{align*}
Then by choosing an $0<\bar\e\les \e$ small enough, we get that on $[T-\bar\e,T]$,
\begin{align*}
\|\d\Th(\cd)\|_{C^{{\a\over 2},1}}+\|\d\Th^0(\cd)\|_{C^{{\a\over 2},{\a\over 2},\a,1+\a,2}}\les {1\over 2}
\big[ \|\d \th^0(\cd)\|_{C^{0,0,\a,1+\a,2}}+\|\d\th(\cd)\|_{C^{0,1+\a}}\big].
\end{align*}
Thus, \rf{Contraction} holds and this completes the proof.
\end{proof}

\ms

\no
{\bf Complete the proof of \autoref{thm:HJB}.}
By a routine argument, we can prove that equilibrium HJB equation \rf{HJB-si-ti} admits a unique classical solution
$(\Th(\cd),\Th^0(\cd))$ on $[T-\e,T]$, where $\e$ is given by \autoref{Prop:Holder}.
Thus, to extend the solution to the whole time interval $[0,T]$,
it suffices to prove a global prior estimate for $\|\Th(\cd)\|_{C^{0,1+\a}}$ and  $\|\Th^0(\cd)\|_{C^{0,0,0,1+\a,0}}$.

\ms
By \rf{Th-x-holder} and \rf{V-x-holder}, we get
\begin{align*}
\|\Th_x(t,s,\ti x,\cd,y)\|_{\a}+ \|\Th_x(s,\cd)\|_{\a}
&\les K\big[1+\|h(\cd)\|_{C^{1+\a}}+\|h^0(\cd)\|_{C^{0,0,1+\a,0}}\big]\\
&\q+\int_s^T{2 K\over \sqrt{r-s}}
\big[\|\Th^0_x(t,r,\ti x,\cd,y)\|_\a+\|\Th_x(r,\cd)\|_\a\big]dr.
\end{align*}
Let $\cM(\cd)$ be the unique solution of the following integral equation
\begin{align*}
\cM(s)=K\big[1+\|h(\cd)\|_{C^{1+\a}}+\|h^0(\cd)\|_{C^{0,0,1+\a,0}}\big]+\int_s^T{2K\over \sqrt{r-s}}\cM(r)dr,
\end{align*}
and define
$$
\h\cM\deq \sup_{s\in[0,T]}\cM(s).
$$
Then by the comparison theorem for Volterra integral equations (see \cite{Wang-Yong2015}, for example), we get
$$
\|\Th_x(s,\cd)\|_{C^{\a}}+\|\Th^0_x(\cd,s,\cd,\cd,\cd)\|_{C^{0,0,0,\a,0}}\les\cM(s)\les\h\cM \q\hbox{on}\q [0,T].
$$
This completes the proof.

\subsection{Proof of \autoref{lem:convergence}}

We first establish a priori estimate for the second-order derivative terms $\Th^\e_{xx}(\cd)$.
Recalling from \rf{Th-e}, similar to \rf{proof-step1-Th-x-change}, we have
\begin{align}
\nn\Th^\e_x(s,x)&=-\int_{\dbR^n}\ti\Xi(s,x,t+\e,x-\sqrt{t+\e-s}\mu;u)\[\Th_x(t+\e,x-\sqrt{t+\e-s}\mu)\nn\\
\nn&\q-\ti\rho(s,x,t+\e,x-\sqrt{t+\e-s}\mu;u)\Th(t+\e,x-\sqrt{t+\e-s}\mu)\]d\mu\\
\nn&\q+\int_s^{t+\e}\int_{\dbR^n}\ti\Xi_x(s,x,r,x-\sqrt{r-s}\mu;u)
\[\Th^\e_{x}(r,x-\sqrt{r-s}\mu) b(r,x-\sqrt{r-s}\mu,u)\nn\\
&\q+g\big(r,x-\sqrt{r-s}\mu,u,\Th^\e(r,x-\sqrt{r-s}\mu),
\Th^\e_x(r,x-\sqrt{r-s}\mu)\si(r,x-\sqrt{r-s}\mu,u)\big)\]d\mu dr,\nn
\end{align}
where
\begin{align}
\nn&\wt \Xi(s,x,r,x-\sqrt{r-s}\mu;u):={1 \over (4\pi)^{n\over 2}(\det[a(r,x-\sqrt{r-s}\mu,u)])^{1\over 2}}
e^{-{\lan a(r,x-\sqrt{r-s}\mu,u)^{-1}\mu,\mu\ran\over 4}},\\
\nn&\ti\rho(s,x,r,x-\sqrt{r-s}\mu;u)\nn\\
&\q:={(\det[a(r,x-\sqrt{r-s}\mu,u)])_x\over 2\det[a(r,x-\sqrt{r-s}\mu,u)]}
+{\lan[a(r,x-\sqrt{r-s}\mu,u)^{-1}]_x \mu,\mu\ran\over 4},\nn\\
\nn&\wt\Xi_x(s,x,r,x-\sqrt{r-s}\mu;u)\\
\nn&\q:={-1 \over (4\pi)^{n\over 2}(\det[a(r,x-\sqrt{r-s}\mu,u)])^{1\over 2}}
e^{-{\lan a(r,x-\sqrt{r-s}\mu,u)^{-1}\mu,\mu\ran\over 4}}{a(r,x-\sqrt{r-s}\mu,u)^{-1}\over 2\sqrt{r-s}}\mu.
\end{align}
Note that the functions $\ti\Xi(\cd)$ and $\ti\rho(\cd)$ depend on the variable $x$ only through the function $a(\cd)$.
Recall that $a(\cd)$ is a smooth function with bounded derivatives and it satisfies the non-degenerate condition \rf{non-degen-condition}.
Thus, it will not cause any difficulties in the proof.
For the ease of presentation, we assume that
$$
a(s,x,u)=a(s,u),\q (s,x,u)\in[0,T]\times\dbR^n\times U.
$$
Then by some straightforward calculations, it is easy to obtain that
\begin{align}
\nn\Th^\e_{xx}(s,x)&=-\int_{\dbR^n}\ti\Xi(s,x,t+\e,x-\sqrt{t+\e-s}\mu;u)\[\Th_{xx}(t+\e,x-\sqrt{t+\e-s}\mu)\nn\\
\nn&\qq-\ti\rho(s,x,t+\e,x-\sqrt{t+\e-s}\mu;u)\Th_x(t+\e,x-\sqrt{t+\e-s}\mu)\]d\mu\\
\nn&\q+\int_s^{t+\e}\int_{\dbR^n}\ti\Xi_x(s,x,r,x-\sqrt{r-s}\mu;u)
\[\Th^\e_{x}(r,x-\sqrt{r-s}\mu) b_x(r,x-\sqrt{r-s}\mu,u)\nn\\
\nn&\qq+ \Th^\e_{xx}(r,x-\sqrt{r-s}\mu) b(r,x-\sqrt{r-s}\mu,u)+ g_\th(r,x-\sqrt{r-s}\mu,u)
\Th^\e_{x}(r,x-\sqrt{r-s}\mu)\nn\\
&\qq+g_p(r,x-\sqrt{r-s}\mu,u)
\Th^\e_{xx}(r,x-\sqrt{r-s}\mu)+g_x(r,x-\sqrt{r-s}\mu,u)\]d\mu dr,\label{proof-lem-Con3}
\end{align}
where $g(r,x-\sqrt{r-s}\mu,u)\equiv g(r,x-\sqrt{r-s}\mu,u,\Th^\e(r,x-\sqrt{r-s}\mu),
\Th^\e_x(r,x-\sqrt{r-s}\mu)\si(r,x-\sqrt{r-s}\mu,u))$.
It follows that
\begin{align}
\nn|\Th^\e_{xx}(s,x)|&\les K \int_{\dbR^n}e^{-\l|\mu|^2\over 4}\|\Th(t+\e,\cd)\|_{C^2}d\mu\\
&\q+K\int_s^{t+\e}\int_{\dbR^n}{e^{-\l|\mu|^2\over4}\over \sqrt{r-\mu}}\[|\Th^\e_{x}(r,x-\sqrt{r-s}\mu)|
+|\Th^\e_{xx}(r,x-\sqrt{r-s}\mu)|+1\]d\mu dr.\nn
\end{align}
Then by Gr\"{o}nwall's inequality, we get
\bel{proof-lem-Con6}
\|\Th^\e_{xx}(\cd,\cd)\|_{L^\i}\les K\big[1+\|\Th(t+\e,\cd)\|_{C^{2}}\big].
\ee
By \autoref{lem:HJB-Th-x}, we also have
\bel{proof-lem-Con61}
\|\Th^\e(\cd,\cd)\|_{C^{0,1}}\les K\big[1+\|\Th(t+\e,\cd)\|_{C^{1}}\big].
\ee
Thus, similar to \rf{Th-x-holder}, using \rf{proof-lem-Con3} and the estimates \rf{proof-lem-Con6}--\rf{proof-lem-Con61}, we have
\begin{align}
\nn\|\Th^\e_{xx}(s,\cd)\|_{\a}&\les K \int_{\dbR^n}e^{-\l|\mu|^2\over 4}\|\Th(t+\e,\cd)\|_{C^{2+\a}}d\mu\\
&\q+K\int_s^{t+\e}\int_{\dbR^n}{e^{-\l|\mu|^2\over4}\over \sqrt{r-s}}\[\|\Th^\e(r,\cd)\|_{\a}+\|\Th^\e_{x}(r,\cd)\|_{\a}
+\|\Th^\e_{xx}(r,\cd)\|_{\a}+1\]d\mu dr.\nn
\end{align}
Then by Gr\"{o}nwall's inequality again, we have
\bel{proof-lem-Con7}
\|\Th^\e_{xx}(\cd,\cd)\|_{C^{\a}}\les K\big[1+\|\Th(t+\e,\cd)\|_{C^{2+\a}}\big].
\ee
For any $t\les s_1\les s_2\les t+\e$, by \rf{proof-lem-Con3} and \rf{proof-lem-Con7} we have
\begin{align}
\nn& |\Th^\e_{xx}(s_1,x)-\Th^\e_{xx}(s_2,x)|\nn\\
&\q\les K\int_{\dbR^n}e^{-\l|\mu|^2\over 4}\big[\|\Th_{xx}(t+\e,\cd)\|_{\a}
+\|\Th_{x}(t+\e,\cd)\|_{\a}\big]|s_1-s_2|^{\a\over 2}|\mu|^{\a\over 2}d\mu\nn\\
\nn&\qq
+K\int_{s_1}^{t+\e}\int_{\dbR^n}{e^{-\l|\mu|^2\over4}\over \sqrt{r-s_1}}\big[\|\Th^\e(r,\cd)\|_{\a}+\|\Th^\e_{x}(r,\cd)\|_{\a}
+\|\Th^\e_{xx}(r,\cd)\|_{\a}+1\big]|s_1-s_2|^{\a\over 2}|\mu|^{\a\over 2}d\mu dr\nn\\
\nn&\qq+K\sqrt{s_2-s_1}+K\int_{s_1}^{s_2}\int_{\dbR^n}{e^{-\l|\mu|^2\over4}\over \sqrt{r-s_1}}d\mu dr\nn\\
&\q\les K\big[1+\|\Th(t+\e,\cd)\|_{C^{2+\a}}\big]|s_1-s_2|^{\a\over 2}.\label{proof-lem-Con8}
\end{align}
Similar to \rf{Prop:Holder-t-11}, we have that for any $t\les s_1\les s_2\les t+\e$,
\begin{align}
|\Th^\e(s_1,x)-\Th^\e(s_2,x)|+ |\Th^\e_{x}(s_1,x)-\Th^\e_{x}(s_2,x)|
\les K\big[1+\|\Th(t+\e,\cd)\|_{C^{1+\a}}\big]|s_1-s_2|^{\a\over 2}.
\end{align}
Moreover, similar to \rf{V-x-y-s-holder}, we have that for any $t\les r \les s_1\les s_2\les t+\e$,
\begin{align}
\nn&|\Th^{0,\e}(r,s_1,\ti x,x,y)-\Th^{0,\e}(r,s_2,\ti x,x,y)|+
|\Th_x^{0,\e}(r,s_1,\ti x,x,y)-\Th_x^{0,\e}(r,s_2,\ti x,x,y)|\nn\\
&\q\les K\big[1+\|\Th(t+\e,\cd)\|_{C^{1+\a}}+\|\Th^0(r,t+\e,\cd\,,\cd\,,\cd)\|_{C^{\a,1+\a,2}}\big]|s_1-s_2|^{\a\over 2}.\label{proof-lem-Con9}
\end{align}
Note that $\Th^\e(t+\e,x)=\Th(t+\e,x)$, $\Th^{0,\e}(r,t+\e,\ti x,x,y)=\Th^0(r,t+\e,\ti x,x,y)$,
and that the functions $\Th(\cd)$ and $\Th^0(\cd)$ are smooth with bounded derivatives.
Then the desired results are obtained from \rf{proof-lem-Con8}--\rf{proof-lem-Con9}.

\ms

\end{document}